\documentclass[10pt,twoside,a4paper,english,leqno]{article}

\title{A new fully justified asymptotic model for the propagation of internal waves in the Camassa-Holm regime}
\author{Vincent Duch\^ene%
\thanks{IRMAR - UMR6625, Univ. Rennes 1, CNRS, Campus de Beaulieu, F-35042 Rennes cedex, France. \url{vincent.duchene@univ-rennes1.fr}}, 
Samer Israwi%
\thanks{Math\'ematiques, Facult\'e des sciences I et Ecole doctorale des sciences et technologie, Universit\'e Libanaise, Beyrouth, Liban and CRAMS: Center for Research in Applied Mathematics and Statistics, AUL, Lebanon.
\url{s_israwi83@hotmail.com}},
Raafat Talhouk%
\thanks{Math\'ematiques, Facult\'e des sciences I et Ecole doctorale des sciences et technologie, Universit\'e Libanaise, Beyrouth, Liban. \url{rtalhouk@ul.edu.lb}}}
\date{\today}

\usepackage{babel}      
\usepackage[utf8]{inputenc}
\usepackage[ec]{aeguill}   
\usepackage{amsmath, amsthm}
\usepackage{amsfonts,amssymb}
\usepackage{stmaryrd}
\usepackage{float}          
\usepackage[pdftex]{graphicx} 
\usepackage{epstopdf}
\usepackage{fancyhdr}       
\usepackage{url}          
\usepackage{cite}	        
\usepackage{hyperref}

\numberwithin{equation}{section}

\fancypagestyle{thestyle}{%
	\fancyfoot{}
	\fancyhead[CE]{A new fully justified model in the Camassa-Holm regime} 
	\fancyhead[CO]{\scriptsize Vincent Duch\^ene, Samer Israwi, Raafat Talhouk} 
	\fancyhead[RE,LO]{\scriptsize\today}  
	\fancyhead[LE,RO]{\scriptsize\thepage}
	}
\newcommand{\RR}{\mathbb{R}}
\newcommand{\NN}{\mathbb{N}}
\renewcommand{\t}{\tilde}
\renewcommand{\b}{\bar}

\renewcommand{\P}{\mathcal{P}}
\newcommand{\N}{\mathcal{N}}
\newcommand{\Q}{\mathcal{Q}}
\newcommand{\R}{\mathcal{R}}
\renewcommand{\O}{\mathcal{O}}
\newcommand{\mfT}{\mathfrak T}
\newcommand{\p}{\mathfrak{p}}
\newcommand{\mfQ}{\mathfrak{Q}}
\DeclareMathOperator*{\esssup}{ess\,sup}
\DeclareMathOperator{\Bo}{Bo}
\DeclareMathOperator{\bo}{bo}
\newcommand{\dsp}{\displaystyle}
\newcommand{\nn}{\nonumber}

\newcommand{\id}[1]{\left\vert_{\scriptstyle #1}\right.}

\newtheorem{Theorem}{Theorem}[section]
\newtheorem{Definition}[Theorem]{Definition}
\newtheorem{Proposition}[Theorem]{Proposition}
\newtheorem{Lemma}[Theorem]{Lemma}
\newtheorem{Corollary}[Theorem]{Corollary}
\newtheorem{Remark}[Theorem]{Remark}

\begin{document}

\maketitle
\begin{abstract}
This study deals with asymptotic models for the propagation of one-dimensional internal waves at the interface between two layers of immiscible fluids of different densities, under the rigid lid assumption and with a flat bottom. 
We present a new Green-Naghdi type model in the Camassa-Holm (or medium amplitude) regime. This model is fully justified, in the sense that it is consistent, well-posed, and that its solutions remain close to exact solutions of the full Euler system with corresponding initial data. Moreover, our system allows to fully justify any well-posed and consistent lower order model; and in particular the so-called Constantin-Lannes approximation, which extends the classical Korteweg-de Vries equation in the Camassa-Holm regime.
\end{abstract}

\section{Introduction}

 \subsection{Presentation of the problem}
In the present paper, we study the propagation of internal waves in a two-fluid system, which consists in two layers of immiscible fluids of different densities, under the only influence of gravity. The domain of the two layers is infinite in the horizontal space variable (assumed to be of dimension $d=1$) and delimited above by a flat, rigid lid, a below by a flat bottom. Moreover, we assume that the fluids are homogeneous, ideal, incompressible and irrotational. We let the reader refer to~\cite{HelfrichMelville06}, and references therein, for a good overview of the ins and outs concerning density-stratified fluids in oceanography, and the relevance of our setup as a model for such system.

The governing equations describing the evolution of the flow under the aforementioned configuration may be reduced to a system of two evolution equations located at the interface between the two layers (following a strategy initiated in the water-wave configuration in~\cite{Zakharov68,CraigSulem93}, and achieved in the bi-fluidic case in~\cite{BonaLannesSaut08}), named {\em full Euler system}. However, the study of this system is extremely challenging. In particular, the well-posedness of the Cauchy problem has been answered satisfactorily (that is, with an existence of solutions on a time scale consistent with physical observations) only recently; see~\cite{Lannes13}. 
\medskip

Under these circumstances, a great deal of interests has been drawn to {\em asymptotic models}, in order to predict accurately the main behavior of the system, provided some parameters describing the domain and nature of the flow are small. Parameters of interests include
\[ \mu \ = \ \frac{d_1^2}{\lambda^2}, \quad \epsilon \ = \ \frac{a}{d_1}, \quad \delta \ = \ \frac{d_1}{d_2},\quad \gamma \ = \ \frac{\rho_1}{\rho_2},\quad \Bo\ =\ \dfrac{g(\rho_2-\rho_1)\lambda^2}{\sigma},\]
where $a$ is the maximal vertical deformation of the interface with respect to its rest position; $\lambda$ is a characteristic horizontal length; $d_1$ (resp. $d_2$) is the depth of the upper (resp. lower) layer; and $\rho_1$ (resp. $\rho_2$) is the density of the upper (resp. lower) layer, $g$ the gravitational acceleration and $\sigma$ the interfacial tension coefficient. Mathematically speaking, $\mu$ and $\epsilon$ measure respectively the amount of dispersion and nonlinearity which will contribute to the evolution of internal waves, and $\Bo^{-1}$ (the Bond number) expresses the ratio of surface tension forces to gravitational forces.
It would be quite tedious to make an attempt at acknowledging every work on this aspect, but let us introduce some earlier results directly related to the present paper.
\medskip

Shallow water ($\mu\ll1$) asymptotic models for uni-dimensional internal waves have been derived and studied in the pioneer works of~\cite{Miyata85,Malcprimetseva89,Matsuno93}. More recently, Choi and Camassa developed models with weakly ($\epsilon=\O(\mu)$) and strongly ($\epsilon\sim1$) nonlinear terms, respectively in~\cite{ChoiCamassa96,ChoiCamassa99} with horizontal dimension $d=2$. They obtain bi-fluidic extensions of the classical shallow water (or Saint-Venant~\cite{Saint-Venant71}), Boussinesq~\cite{Boussinesq71,Boussinesq72} and Green-Naghdi~\cite{Serre53,GreenNaghdi76} models. Similar systems have been derived in~\cite{OstrovskyGrue03} (with the additional assumption of $\gamma\approx 1$) and in~\cite{CraigGuyenneKalisch05} (using a different approach, {\em i.e.} making use of the Hamiltonian structure of the full Euler equations). However, the aforementioned results are limited to the formal level. Let us mention now the work of Bona, Lannes and Saut~\cite{BonaLannesSaut08} who, following a strategy initiated in~\cite{BonaChenSaut02,BonaChenSaut04} in the water-wave setting  (one layer of fluid, with free surface), derived a large class of models for different regimes, under the rigid-lid assumption, neglecting surface tension effects and with flat bottom (see also~\cite{Anh09} where a topography and surface tension is added to the system, and~\cite{Duchene10} where the rigid-lid assumption is removed). The models derived in these papers are systematically justified by a consistency result: roughly speaking, sufficiently smooth solutions of the full Euler system satisfy the equations of the asymptotic model, up to a small remainder.
\medskip

Yet the consistency is only one of the properties that an asymptotic model shall satisfy, for its validity to be ascertained. Indeed, it leaves two important questions unanswered: for a large class of initial data (typically bounded in suitable Sobolev spaces)
\begin{itemize}
\item[1.] {\em (well-posedness)} do the full Euler system, as well as the asymptotic model, define a unique solution on a relevant time scale?
\item[2.] {\em (convergence)} is the difference between these two solutions small over the relevant time scale?
\end{itemize}
As mentioned earlier, Lannes has recently proved~\cite{Lannes13} that the Cauchy problem for bi-fluidic full Euler system is well-posed in Sobolev spaces, in the presence of a small amount of surface tension. Thus the full justification of a consistent system of equation as an asymptotic model, in the sense described above, follows from its well-posedness and a stability result; see~\cite[Appendix~C]{Lannes} for a detailed discussion and state of the art in the water-wave setting. 
\medskip

A striking discrepancy between the water-wave and the bi-fluidic setting is that in the latter, large amplitude internal waves are known to generate Kelvin-Helmholtz instabilities, so that surface tension is necessary in order to regularize the flow. A crucial contribution of~\cite{Lannes13} consists in asserting that ``the Kelvin-Helmholtz instabilities appear above a frequency threshold for which surface tension is relevant, while the main (observable) part of the wave involves low frequencies located below this frequency threshold''. It is therefore expected that the surface tension does not play an essential role in the dominant evolution of the flow, especially in the shallow water regime.
This intuition is confirmed by the fact that well-posedness and stability results have been proved for the bi-fluidic shallow-water system~\cite{GuyenneLannesSaut10}, and a class of Boussinesq-type systems~\cite{Duchene11a}, without surface tension and under reasonable assumptions on the flow (typically, the shear velocity must be sufficiently small). However, the original bi-fluidic Green-Naghdi model is known to be unconditionally ill-posed~\cite{LiskaMargolinWendroff95}, which has led to various propositions in order to overcome this difficulty; see~\cite{ChoiBarrosJo09,CotterHolmPercival10} and references therein. Let us recall here that Green-Naghdi models consist in higher order extensions of the shallow water equation, thus are consistent with precision $\O(\mu^2)$ instead of $\O(\mu)$, and allow strong nonlinearities (whereas Boussinesq models are limited to the long wave regime: $\epsilon=\O(\mu)$). 
Finally, we mention the work of Xu~\cite{Xu12}, which studies and rigorously justifies the so-called intermediate long wave system, obtained in a regime similar to ours: $\epsilon\sim \sqrt{\mu}$, but $\delta\sim\sqrt{\mu}$.
\medskip

{\em In this work, we present a new Green-Naghdi type model in the Camassa-Holm (or medium amplitude) regime, $\epsilon=\O(\sqrt{\mu})$, for the propagation of internal waves.} More precisely, the regime of validity of our model is, with fixed $\mu_{\max},M,\delta_{\min}^{-1},\delta_{\max}<\infty$:
\[  0\ < \ \mu \ \leq \ \mu_{\max},\quad   0 \ \leq \ \epsilon \ \leq \ \min(M\sqrt\mu,1), \quad \delta_{\min}\ \leq\ \delta\ \leq \ \delta_{\max},\quad  0\ \leq \ \gamma\ <\ 1.
\]
 Our model is fully justified: we prove that the full Euler system is consistent with our model, and that our system is well-posed (in the sense of Hadamard) in Sobolev spaces, and stable with respect to perturbations of the equations. These results hold identically without or with (small) surface tension.
\medskip

Let us emphasize that in addition to its relevance as an asymptotic model, our system offers an important tool for the justification of other models. Indeed, it suffices to check that a given approximate solution solves our system up to a small remainder, to ensure that it is truly close to the solution of our new model, and therefore to the corresponding solution of the full Euler system.

Such strategy has been used in particular in~\cite{BonaColinLannes05} in order to rigorously justify the historical Korteweg-de Vries equation as a model for the propagation of surface wave in the long wave regime (with a Boussinesq model as the intermediary system); and this result has been extended to the bi-fluidic case in~\cite{Duchene11a}.
Higher order models in the Camassa-Holm regime have been introduced and justified in the sense of consistency in~\cite{ConstantinLannes09} in the water-wave case, and in~\cite{Duchene13} in the bi-fluidic case. We are therefore able to fully justify the latter model, in the sense described above.

 \subsection{Organization of the paper} 
The paper is organized as follows.

\begin{itemize}
\item[] {\em Section 2: The full Euler system.} 
\item[] {\em Section 3: Main results.}
\item[] {\em Section 4: Construction of our model.}
\item[] {\em Section 5: Preliminary results.}
\item[] {\em Section 6: Linear analysis.}
\item[] {\em Section 7: Proof of existence, stability and convergence.}
\item[] {\em Section 8: Full justification of asymptotic models.}
\item[] {\em Appendix A: Product and commutator estimates in Sobolev spaces.}
\end{itemize}
 We start by introducing in Section 2 the non-dimensionalized full Euler system, describing the evolution of the two-fluid system we consider.\\
 In Section 3, we present our new model, and we announce the main results of this paper.\\
 This asymptotic model is precisely derived and motivated in Section 4.\\
 Sections 5 and 6 contain the necessary ingredients for the proof of our results, which are completed in Section 7.\\
 In Section 8, we explain how our system allows us to justify any well-posed and consistent lower order model. \\
 Finally the title of Appendix A is self-explanatory.
 \medskip
 
We conclude this section with an inventory of the notations used throughout the present paper.
 \paragraph{Notations}
 In the following, $C_0$ denotes any nonnegative constant whose exact expression is of no importance. The notation $a\lesssim b$ means that 
 $a\leq C_0\ b$.\\
 We denote by $C(\lambda_1, \lambda_2,\dots)$ a nonnegative constant depending on the parameters
 $\lambda_1$, $\lambda_2$,\dots and whose dependence on the $\lambda_j$ is always assumed to be nondecreasing.\\
We use the condensed notation
\[
A_s=B_s+\left\langle C_s\right\rangle_{s>\underline{s}},
\]
to express that $A_s=B_s$ if $s\leq \underline{s}$ and $A_s=B_s+C_s$ if $s> \underline{s}$.\\
 Let $p$ be any constant
 with $1\leq p< \infty$ and denote $L^p=L^p(\RR)$ the space of all Lebesgue-measurable functions
 $f$ with the standard norm \[\vert f \vert_{L^p}=\big(\int_{\RR}\vert f(x)\vert^p dx\big)^{1/p}<\infty.\] The real inner product of any functions $f_1$
 and $f_2$ in the Hilbert space $L^2(\RR)$ is denoted by
\[
 (f_1,f_2)=\int_{\RR}f_1(x)f_2(x) dx.
 \]
 The space $L^\infty=L^\infty(\RR)$ consists of all essentially bounded, Lebesgue-measurable functions
 $f$ with the norm
\[
 \vert f\vert_{L^\infty}= \esssup \vert f(x)\vert<\infty.
\]
 We denote by $W^{1,\infty}(\RR)=\{f, \mbox{ s.t. }f,\partial_x f\in L^{\infty}(\RR)\}$ endowed with its canonical norm.\\
 For any real constant $s\geq0$, $H^s=H^s(\RR)$ denotes the Sobolev space of all tempered
 distributions $f$ with the norm $\vert f\vert_{H^s}=\vert \Lambda^s f\vert_{L^2} < \infty$, where $\Lambda$
 is the pseudo-differential operator $\Lambda=(1-\partial_x^2)^{1/2}$.\\
For a given $\mu>0$, we denote by $H^{s+1}_\mu(\RR)$ the space $H^{s+1}(\RR)$ endowed with the norm
\[\big\vert\ \cdot\ \big\vert_{H^{s+1}_\mu}^2 \ \equiv \ \big\vert\ \cdot\ \big\vert_{H^{s}}^2\ + \ \mu \big\vert\ \cdot\ \big\vert_{H^{s+1}}^2\ .\]
 For any functions $u=u(t,x)$ and $v(t,x)$
 defined on $ [0,T)\times \RR$ with $T>0$, we denote the inner product, the $L^p$-norm and especially
 the $L^2$-norm, as well as the Sobolev norm,
 with respect to the spatial variable $x$, by $(u,v)=(u(t,\cdot),v(t,\cdot))$, $\vert u \vert_{L^p}=\vert u(t,\cdot)\vert_{L^p}$,
 $\vert u \vert_{L^2}=\vert u(t,\cdot)\vert_{L^2}$ , and $ \vert u \vert_{H^s}=\vert u(t,\cdot)\vert_{H^s}$, respectively.\\
 We denote $L^\infty([0,T);H^s(\RR))$ the space of functions such that $u(t,\cdot)$ is controlled in $H^s$, uniformly for $t\in[0,T)$:
 \[\big\Vert u\big\Vert_{L^\infty([0,T);H^s(\RR))} \ = \ \esssup_{t\in[0,T)}\vert u(t,\cdot)\vert_{H^s} \ < \ \infty.\]
 Finally, $C^k(\RR)$ denote the space of
 $k$-times continuously differentiable functions.\\
 For any closed operator $T$ defined on a Banach space $X$ of functions, the commutator $[T,f]$ is defined
  by $[T,f]g=T(fg)-fT(g)$ with $f$, $g$ and $fg$ belonging to the domain of $T$. The same notation is used for $f$ an operator mapping the domain of $T$ into itself.
  \medskip

\section{The full Euler system}
\label{ssec:FullEulerSystem}

We recall that the system we study consists in two layers of immiscible, homogeneous, ideal, incompressible fluids only under the influence of gravity (see Figure~\ref{fig:SketchOfDomain}). We restrict ourselves to the two-dimensional case, {\em i.e.} the horizontal dimension $d=1$. The derivation of the governing equations of such a system is not new. We briefly recall it below, and refer to~\cite{BonaLannesSaut08,Anh09,Duchene13} for more details.

\begin{figure}[!htb]
\centering
 \includegraphics[width=.9\textwidth]{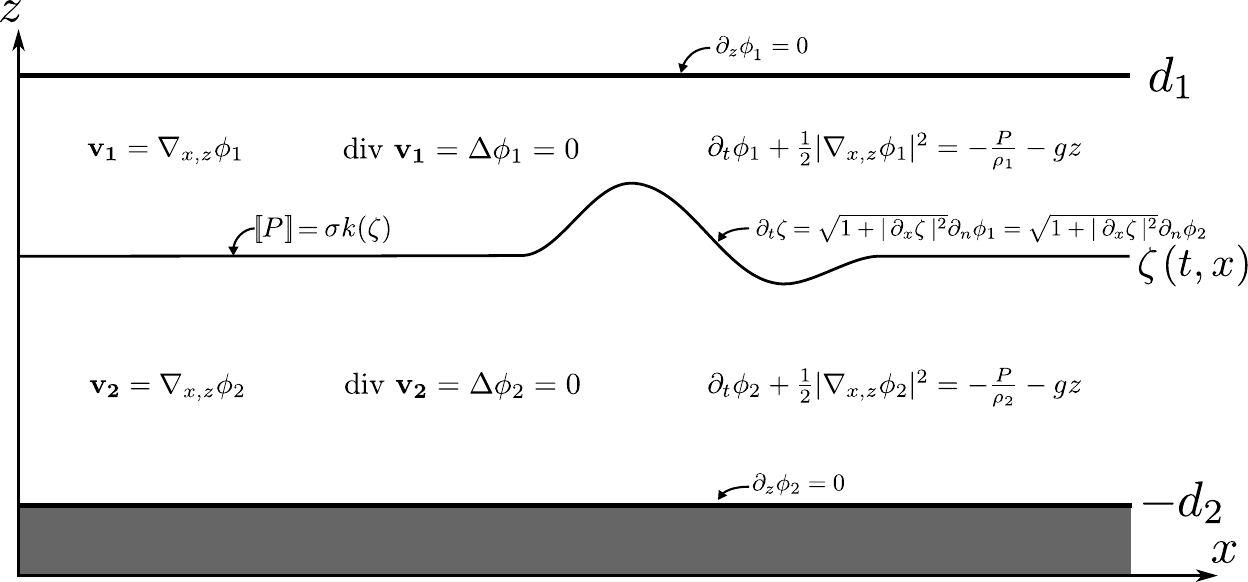}
 \caption{Sketch of the domain and governing equations}
\label{fig:SketchOfDomain}
\end{figure}

We assume that the interface is given as the graph of a function $\zeta(t,x)$ which expresses the deviation from its rest position $(x,0)$ at the spatial coordinate $x$ and at time $t$. The bottom and surface are assumed to be rigid and flat.
Therefore, at each time $t\ge 0$, the domains of the upper and lower fluid (denoted, respectively, $\Omega_1^t$ and $\Omega_2^t$), are given by
\begin{align*}
 \Omega_1^t \ &= \ \{\ (x,z)\in\RR\times\RR, \quad \zeta(t,x)\ \leq\ z\ \leq \ d_1\ \}, \\
 \Omega_2^t \ &= \ \{\ (x,z)\in\RR\times\RR, \quad -d_2 \  \leq\ z\ \leq \ \zeta(t,x)\ \}.
\end{align*}
We assume that the two domains are strictly connected, that is there exists $h>0$ such that
\[ d_1- \zeta(t,x) \geq h>0, \quad \text{and} \quad d_2+\zeta(t,x)\geq h>0.\]

We denote by $(\rho_1,{\mathbf v}_1)$ and $(\rho_2,{\mathbf v}_2)$ the mass density and velocity fields of, respectively, the upper and  the lower fluid. The two fluids are assumed to be homogeneous and incompressible, so that the mass densities $\rho_1,\ \rho_2$ are constant, and the velocity fields ${\mathbf v}_1,\ {\mathbf v}_2$ are divergence free.
As we assume the flows to be irrotational, one can express the velocity field as gradients of a potential: ${\mathbf v}_i=\nabla\phi_i$ $(i=1,2)$, and the velocity potentials satisfy Laplace's equation
\[\partial_x^2 \phi_i \ + \ \partial_z^2 \phi_i \ = \ 0.\]

 The fluids being ideal, they satisfy the Euler equations. Integrating the momentum equation yields the so-called Bernoulli equation, written in terms of the velocity potentials:
 \[   \partial_t \phi_i+\frac{1}{2} |\nabla_{x,z} \phi_i|^2=-\frac{P}{\rho_i}-gz \quad \text{  in }\Omega^t_i \quad (i=1,2),\]
 where $P$ denotes the pressure inside the fluids.

 From the assumption that no fluid particle crosses the surface, the bottom or the interface, one deduces kinematic boundary conditions, and the set of equations is closed by the continuity of the stress tensor at the interface, which reads
\[  \llbracket P(t,x) \rrbracket \ \equiv \ \lim\limits_{\varepsilon\to 0} \Big( \ P(t,x,\zeta(t,x)+\varepsilon) \ - \ P(t,x,\zeta(t,x)-\varepsilon) \ \Big) \ = \ \sigma k\big(\zeta(t,x)\big),\]
where $ k(\zeta)=-\partial_x \Big(\frac1{\sqrt{1+|\partial_x\zeta|^2}}\partial_x\zeta\Big)$ denotes the mean curvature of the interface and $\sigma$ the surface (or interfacial) tension coefficient.

Altogether, the governing equations of our problem are given by the following
  \begin{equation}  \label{eqn:EulerComplet}
\left\{\begin{array}{ll}
         \partial_x^2 \phi_i \ + \ \partial_z^2 \phi_i \ = \ 0 & \mbox{ in }\Omega^t_i, \ i=1,2,\\
         \partial_t \phi_i+\frac{1}{2} |\nabla_{x,z} \phi_i|^2=-\frac{P}{\rho_i}-gz & \mbox{ in }\Omega^t_i, \ i=1,2, \\
         \partial_{z}\phi_1 \ = \ 0  & \mbox{ on } \Gamma_{\rm t}\equiv\{(x,z),z=d_1\}, \\
         \partial_t \zeta  \ = \ \sqrt{1+|\partial_x\zeta|^2}\partial_{n}\phi_1 \ = \ \sqrt{1+|\partial_x\zeta|^2}\partial_{n}\phi_2  & \mbox{ on } \Gamma \equiv\{(x,z),z=\zeta(t,x)\},\\
         \partial_{z}\phi_2  \ = \ 0 &  \mbox{ on } \Gamma_{\rm b}\equiv\{(x,z),z=-d_2\}, \\
         \llbracket P(t,x) \rrbracket  \ = \ \sigma k(\zeta) & \mbox{ on } \Gamma,
         \end{array}
\right.
\end{equation}
where $n$ denotes the unit upward normal vector at the interface.
\bigskip

The next step consists in {\em nondimensionalizing the system}.
Thanks to an appropriate scaling, the two-layer full Euler system~\eqref{eqn:EulerComplet} can be written in dimensionless form.
The study of the linearized system (see~\cite{Lannes13} for example), which can be solved explicitly, leads to a well-adapted rescaling.
\medskip

Let $a$ be the maximum amplitude of the deformation of the interface. We denote by $\lambda$ a characteristic horizontal length, say the wavelength of the interface. Then the typical velocity of small propagating internal waves (or wave celerity) is given by
\[c_0 \ = \ \sqrt{g\frac{(\rho_2-\rho_1) d_1 d_2}{\rho_2 d_1+\rho_1 d_2}}.\]
Consequently, we introduce the dimensionless variables\footnote{We choose $d_1$ as the reference vertical length. This choice is harmless as we assume in the following that the two layers of fluid have comparable depth: the depth ratio $\delta$ do not approach zero or infinity. 
}
\[\begin{array}{cccc}
 \t z \ \equiv\  \dfrac{z}{d_1}, \quad\quad & \t x\ \equiv \ \dfrac{x}{\lambda}, \quad\quad & \t t\ \equiv\ \dfrac{c_0}{\lambda}t,
\end{array}\]
the dimensionless unknowns
\[\begin{array}{cc}
 \t{\zeta}(\t t,\t x)\ \equiv\ \dfrac{\zeta(t,x)}{a}, \quad\quad& \t{\phi_i}(\t t,\t x,\t z)\ \equiv\ \dfrac{d_1}{a\lambda c_0}\phi_i(t,x,z) \quad (i=1,2),
\end{array}\]
as well as five dimensionless parameters
\[
 \gamma\ =\ \dfrac{\rho_1}{\rho_2}, \quad \epsilon\ \equiv\ \dfrac{a}{d_1},\quad   \mu\ \equiv\ \dfrac{d_1^2}{\lambda^2},  \quad \delta\ \equiv \ \dfrac{d_1}{d_2}, \quad \bo\ =\ \dfrac{g(\rho_2-\rho_1)d_1^2}{\sigma}.
\]
\begin{Remark}We use here $\bo=\mu \Bo$ instead of the classical Bond number, $\Bo$, which measures the ratio of gravity forces over capillary forces. As we assume later on that $\bo$ is bounded from below, this amounts to the assumption $\Bo^{-1}=\O(\mu)$. 
\footnote{Such assumption is very natural in the context of internal gravity waves in the ocean. For example, using the values of the experiment of Koop and Butler~\cite{KoopButler81} and a typical surface tension coefficient $\sigma=0.005 N.m^{-1}$ as in~\cite{Lannes13}, one has
\[ \bo^{-1}\ =\ \dfrac{\sigma}{g(\rho_2-\rho_1)d_1^2} \ \approx \ \dfrac{0.005}{9.81(1\,563-998)\times 0.06948^2} \approx 1.87 \times 10^{-4}. \]}
\end{Remark}
\bigskip

We conclude by remarking that the system can be reduced into two evolution equations coupling Zakharov's canonical variables~\cite{Zakharov68,CraigSulem93}, namely  (withdrawing the tildes for the sake of readability)
the deformation of the free interface from its rest position, $\zeta$, and the trace of the dimensionless upper potential at the interface, $\psi$, defined as follows:
\[ \psi \ \equiv \ \phi_1(t,x,\zeta(t,x)).\]
Indeed, $\phi_1$ and $\phi_2$ are uniquely deduced from $(\zeta,\psi)$ as solutions of the following Laplace's problems:
\begin{align}
\label{eqn:Laplace1} &\left\{
\begin{array}{ll}
 \left(\ \mu\partial_x^2 \ +\  \partial_z^2\ \right)\ \phi_1=0 & \mbox{ in } \Omega_1\equiv \{(x,z)\in \RR^{2},\ \epsilon{\zeta}(x)<z<1\}, \\
\partial_z \phi_1 =0  & \mbox{ on } \Gamma_{\rm t}\equiv \{(x,z)\in \RR^{2},\ z=1\}, \\
 \phi_1 =\psi & \mbox{ on } \Gamma\equiv \{(x,z)\in \RR^{2},\ z=\epsilon \zeta\},
\end{array}
\right.\\
\label{eqn:Laplace2}&\left\{
\begin{array}{ll}
 \left(\ \mu\partial_x^2\ + \ \partial_z^2\ \right)\ \phi_2=0 & \mbox{ in } \Omega_2\equiv\{(x,z)\in \RR^{2},\ -\frac{1}{\delta}<z<\epsilon\zeta\}, \\
\partial_{n}\phi_2 = \partial_{n}\phi_1 & \mbox{ on } \Gamma, \\
 \partial_{z}\phi_2 =0 & \mbox{ on } \Gamma_{\rm b}\equiv \{(x,z)\in \RR^{2},\ z=-\frac{1}{\delta}\}.
\end{array}
\right.
\end{align}
More precisely, we define the so-called Dirichlet-Neumann operators.
\begin{Definition}[Dirichlet-Neumann operators]
Let $\zeta\in H^{t_0+1}(\RR)$, $t_0>1/2$, such that there exists $h>0$ with
$h_1 \ \equiv\  1-\epsilon\zeta \geq h>0$ and $h_2 \ \equiv \ \frac1\delta +\epsilon \zeta\geq h>0$, and let $\psi\in L^2_{\rm loc}(\RR),\partial_x \psi\in H^{1/2}(\RR)$.
 Then we define
  \begin{align*}
 G^{\mu}\psi \ \equiv\  &G^{\mu}[\epsilon\zeta]\psi \ \equiv \ \sqrt{1+\mu|\epsilon\partial_x\zeta|^2}\big(\partial_n \phi_1 \big)\id{z=\epsilon\zeta}  \ = \ -\mu\epsilon(\partial_x\zeta) (\partial_x\phi_1)\id{z=\epsilon\zeta}+(\partial_z\phi_1)\id{z=\epsilon\zeta},\\
H^{\mu,\delta}\psi \ \equiv\   &H^{\mu,\delta}[\epsilon\zeta]\psi\ \equiv\  \partial_x \big(\phi_2\id{z=\epsilon\zeta}\big) \ = \  (\partial_x\phi_2)\id{z=\epsilon\zeta}+\epsilon(\partial_x \zeta)(\partial_z\phi_2)\id{z=\epsilon\zeta},
\end{align*}
where $\phi_1$ and $\phi_2$ are uniquely defined (up to a constant for $\phi_2$) as the solutions in $H^2(\RR)$ of the Laplace's problems~\eqref{eqn:Laplace1}--\eqref{eqn:Laplace2}.
\end{Definition}
The existence and uniqueness of a solution to~\eqref{eqn:Laplace1}--\eqref{eqn:Laplace2}, and therefore the well-posedness of the Dirichlet-Neumann operators follow from classical arguments detailed, for example, in~\cite{Lannes}.
\medskip

Using the above definition, and after straightforward computations, one can rewrite the nondimensionalized version of~\eqref{eqn:EulerComplet} as a simple system of two coupled evolution equations, namely
\begin{equation}\label{eqn:EulerCompletAdim}
\left\{ \begin{array}{l}
\displaystyle\partial_{ t}{\zeta} \ -\ \frac{1}{\mu}G^{\mu}\psi\ =\ 0,  \\ \\
\displaystyle\partial_{ t}\Big(H^{\mu,\delta}\psi-\gamma \partial_x{\psi} \Big)\ + \ (\gamma+\delta)\partial_x{\zeta} \ + \ \frac{\epsilon}{2} \partial_x\Big(|H^{\mu,\delta}\psi|^2 -\gamma |\partial_x {\psi}|^2 \Big) \\
\displaystyle\hspace{7cm} = \ \mu\epsilon\partial_x\N^{\mu,\delta}-\mu\frac{\gamma+\delta}{\bo}\frac{\partial_x \big(k(\epsilon\sqrt\mu\zeta)\big)}{{\epsilon\sqrt\mu}} \ ,
\end{array}
\right.
\end{equation}
where we denote
\[  \N^{\mu,\delta} \ \equiv \ \dfrac{\big(\frac{1}{\mu}G^{\mu}\psi+\epsilon(\partial_x{\zeta})H^{\mu,\delta}\psi \big)^2\ -\ \gamma\big(\frac{1}{\mu}G^{\mu}\psi+\epsilon(\partial_x{\zeta})(\partial_x{\psi}) \big)^2}{2(1+\mu|\epsilon\partial_x{\zeta}|^2)}.
      \]
We will refer to~\eqref{eqn:EulerCompletAdim} as the {\em full Euler system}, and solutions of this system will be exact solutions of our problem.

\section{Main results}\label{mr}
We now present our new Green-Naghdi type model, that we fully justify as an asymptotic model for the full Euler system~\eqref{eqn:EulerCompletAdim}, for a set of dimensionless parameters limited to the so-called {\em Camassa-Holm regime}, that we describe precisely below. 

Let us introduce first the so-called {\em shallow water regime} for two layers of comparable depths:
\begin{multline} \label{eqn:defRegimeSWmr}
\P_{\rm SW} \ \equiv \ \Big\{ (\mu,\epsilon,\delta,\gamma,\bo):\ 0\ < \ \mu \ \leq \ \mu_{\max}, \ 0 \ \leq \ \epsilon \ \leq \ 1, \ \delta \in (\delta_{\min},\delta_{\max}),  \Big.\\
 \Big. \ 0\ \leq \ \gamma\ <\ 1,\  \bo_{\min}\leq \bo\leq \infty \ \Big\},
\end{multline}
with given $0\leq \mu_{\max},\delta_{\min}^{-1},\delta_{\max},\bo_{\min}^{-1}<\infty$. 
The two additional key restrictions for the validity of our model define the {\em Camassa-Holm regime}:
\begin{equation} \label{eqn:defRegimeCHmr}
\P_{\rm CH} \ \equiv \ \left\{ (\mu,\epsilon,\delta,\gamma,\bo) \in \P_{\rm SW}:\ \epsilon \ \leq \ M \sqrt{\mu}  \quad \text{ and } \quad \nu\ \equiv \ \frac{1+\gamma\delta}{3\delta(\gamma+\delta)}-\frac1{\bo} \ge \ \nu_{0}  \ \right\},
\end{equation}
with given $0\leq M,\nu_0^{-1}<\infty$. We denote for convenience
\[ M_{\rm SW} \ \equiv \ \max\big\{\mu_{\max},\delta_{\min}^{-1},\delta_{\max},\bo_{\min}^{-1}\big\}, \quad M_{\rm CH} \ \equiv \ \max\big\{M_{\rm SW},M,\nu_0^{-1}\big\}.\]

Our new system is:
\begin{equation}\label{eq:Serre2mr}\left\{ \begin{array}{l}
\partial_{ t}\zeta +\partial_x\left(\dfrac{h_1 h_2}{h_1+\gamma h_2}\b v\right)\ =\  0,\\ \\
\mathfrak T[\epsilon\zeta] \left( \partial_{ t} \b v + \epsilon \varsigma {\b v } \partial_x {\b v } \right) + (\gamma+\delta)q_1(\epsilon\zeta)\partial_x \zeta   \\
\qquad \qquad+\frac\epsilon2 q_1(\epsilon\zeta) \partial_x \left(\frac{h_1^2  -\gamma h_2^2 }{(h_1+\gamma h_2)^2}|\b v|^2-\varsigma |\b v|^2\right)=  -  \mu \epsilon\frac23\frac{1-\gamma}{(\gamma+\delta)^2} \partial_x\big((\partial_x \b v)^2\big) ,
\end{array} \right. \end{equation}
where $h_1\equiv1-\epsilon\zeta$ (resp. $h_2\equiv\frac1\delta+\epsilon\zeta$) denotes the depth of the upper (resp. lower) fluid, and $\bar v$ is the {\em shear mean velocity}\footnote{$\bar v$ is equivalently defined as $\bar v\equiv \bar u_2 - \gamma \bar u_1$ where $\bar u_1, \ \bar u_2$ are the mean velocities integrated across the vertical layer in each fluid:
$\bar u_1(t,x) \ = \ \frac{1}{h_1(t,x)}\int_{\epsilon\zeta(t,x)}^{1} \partial_x \phi_1(t,x,z) \ dz$ and $  \bar u_2(t,x) \ = \ \frac{1}{h_2(t,x)}\int_{-\frac1\delta}^{\epsilon\zeta(t,x)} \partial_x \phi_2(t,x,z) \ dz$.} defined by 
\begin{equation}\label{eqn:defbarvmr}
\frac1\mu G^{\mu}[\epsilon\zeta]\psi \ = \ -\partial_x \Big(\frac{h_1h_2}{h_1+\gamma h_2}\bar v\Big). 
\end{equation}
The operator ${\mathfrak T}$ is as follows:
\begin{equation}\label{defTmr}
{\mathfrak T}[\epsilon\zeta]V \ = \ q_1(\epsilon\zeta)V \ - \ \mu\nu \partial_x \Big(q_2(\epsilon\zeta)\partial_xV \Big),
\end{equation}
with $q_i(X)\equiv 1+\kappa_i X $ ($i=1,2$) and $\nu,\kappa_1,\kappa_2,\varsigma$ are constants displayed in~\eqref{deftnu},~\eqref{defkappa} and~\eqref{defvarsigma}, later on.
\medskip

This model is fully justified by the following results:
 \begin{Theorem}[Consistency]\label{th:ConsSerreVarmr}
For $\p= (\mu,\epsilon,\delta,\gamma,\bo) \in \P_{\rm SW}$, let $U^\p\equiv(\zeta^\p,\psi^\p)$ be a family of solutions of the full Euler system~\eqref{eqn:EulerCompletAdim} such that there exists $C_0,T>0$ with 
\[ \big\Vert \zeta^\p \big\Vert_{L^\infty([0,T);H^{s+\frac92})}  +  \big\Vert \partial_t\zeta^\p \big\Vert_{L^\infty([0,T);H^{s+\frac72})}  +  \big\Vert \partial_x\psi^\p \big\Vert_{L^\infty([0,T);H^{s+\frac{11}2})}  +  \big\Vert \partial_t\partial_x\psi^\p \big\Vert_{L^\infty([0,T);H^{s+\frac92})}   \leq C_0 ,\]
for any $s\geq s_0+1/2$, $s_0>1/2$, and uniformly with respect to $\p\in \P_{\rm SW}$. Moreover, assume 
\begin{equation}\tag{H1}
\exists h_{01}>0 \text{ such that }\quad h_1 \ \equiv\  1-\epsilon\zeta^\p \geq\ h_{01}\ >\ 0, \quad h_2 \ \equiv \ \frac1\delta +\epsilon \zeta^\p\geq\ h_{01}\ >\ 0.
\end{equation}
 Define $\bar v^\p$ as in~\eqref{eqn:defbarvmr}.
Then $(\zeta^\p,\bar v^\p)$ satisfies~\eqref{eq:Serre2mr}, up to a remainder $R$, bounded by
\[ \big\Vert R \big\Vert_{L^\infty([0,T);H^s)} \ \leq \ (\mu^2+\mu\epsilon^2 )\ C_1,\]
with $C_1=C(M_{\rm SW},h_{01}^{-1},C_0)$, uniformly with respect to the parameters $\p\in\P_{\rm SW}$.
\end{Theorem}
\medskip

\noindent For parameters in the Camassa-Holm regime~\eqref{eqn:defRegimeCHmr}, our system is well-posed (in the sense of Hadamard) in the energy space $X^s\ = \ H^s(\RR)\times H^{s+1}(\RR)$, endowed with the norm
\[
\forall\; U=(\zeta,v)^\top \in X^s, \quad \vert U\vert^2_{X^s}\equiv \vert \zeta\vert^2 _{H^s}+\vert v\vert^2 _{H^s}+ \mu\vert \partial_xv\vert^2 _{H^s},
\]
provided the following ellipticity condition (for the operator $\mathfrak{T}$) holds:
\begin{equation}\tag{H2}
\exists h_{02}>0 \text{ such that }\quad \inf_{x\in \RR} \left(1+\epsilon\kappa_2\zeta\right) \ge \ h_{02} \ > \ 0 \ ; \qquad    \inf_{x\in \RR}  \left( 1+\epsilon \kappa_1\zeta  \right)\ge \ h_{02} \ > \ 0.
\end{equation}
\begin{Theorem}[Existence and uniqueness]\label{thbi1mr}
Let $\p= (\mu,\epsilon,\delta,\gamma,\bo) \in \P_{\rm CH}$ and  $s\geq s_0+1$, $s_0>1/2$, and assume
	 $U_0=(\zeta_0,v_0)^\top\in X^s$ satisfies~\eqref{CondDepth},\eqref{CondEllipticity}. Then there exists
         a maximal time $T_{\max}>0$, uniformly bounded from below with respect to $\p\in \P_{\rm CH}$, such that the system of
         equations~\eqref{eq:Serre2mr} admits
	 a unique solution $U=(\zeta,v)^\top \in C^0([0,T_{\max});X^s)\cap C^1([0,T_{\max});X^{s-1})$ with the initial value $(\zeta,v)\id{t=0}=(\zeta_0,v_0)$
         and preserving the conditions~\eqref{CondDepth},\eqref{CondEllipticity} (with different lower bounds) for any $t\in [0,T_{\max})$.
         
   Moreover, there exists $T^{-1},C_0,\lambda=  C(M_{\rm CH},h_{01}^{-1},h_{02}^{-1},\big\vert U_0\big\vert_{X^{s}})$, independent of $\p\in\P_{\rm CH}$, such that $T_{\max}\geq T/\epsilon$ and one has the energy estimate
\[\forall\ 0\leq t\leq\frac{T}{\epsilon}\ , \qquad 
\big\vert U(t,\cdot)\big\vert_{X^{s}} \ + \ 
\big\vert \partial_t U(t,\cdot)\big\vert_{X^{s-1}}  \leq C_0 e^{\epsilon\lambda t}\ .
\]
If $T_{\max}<\infty$, one has
         \[ \vert U(t,\cdot)\vert_{X^{s}}\longrightarrow\infty\quad\hbox{as}\quad t\longrightarrow T_{\max},\]
         or one of the two conditions~\eqref{CondDepth},\eqref{CondEllipticity} ceases to be true as $t\longrightarrow T_{\max}$.
\end{Theorem}
\begin{Theorem}[Stability] \label{th:stabilityWPmr}
Let $ (\mu,\epsilon,\delta,\gamma,\bo) \in \P_{\rm CH}$ and $s\geq s_0+1$ with $s_0>1/2$, and assume
	 $U_{0,1}=(\zeta_{0,1},v_{0,1})^\top\in X^{s}$ and $U_{0,2}=(\zeta_{0,2},v_{0,2})^\top\in X^{s+1}$ satisfies~\eqref{CondDepth},\eqref{CondEllipticity}. Denote $U_j$ the solution to~\eqref{eq:Serre2mr} with $U_j\id{t=0}=U_{0,j}$.Then there exists $T^{-1},\lambda,C_0= C(M_{\rm CH},h_{01}^{-1},h_{02}^{-1},\big\vert U_{0,1}\big\vert_{X^s},\vert U_{0,2}\vert_{X^{s+1}})$ such that
	  \begin{equation*}
	 	\forall t\in [0,\frac{T}{\epsilon}],\qquad
	 	\big\vert (U_1-U_2)(t,\cdot)\big\vert_{X^s} \leq C_0 e^{\epsilon\lambda t} \big\vert U_{1,0}-U_{2,0}\big\vert_{X^s}.
	 \end{equation*}
\end{Theorem}

Finally, the following ``convergence result'' states that the solutions of our system approach the solutions of the full Euler system, with the accuracy predicted by Theorem~\ref{th:ConsSerreVarmr}.

\begin{Theorem}[Convergence] \label{th:convergencemr}
Let $\p\equiv (\mu,\epsilon,\delta,\gamma,\bo)\in \P_{\rm CH}$ and $s\geq s_0+1$ with $s_0>1/2$, and $U^0\equiv(\zeta^0,\psi^0)^\top\in H^{s+N}$, $N$ sufficiently large, satisfy the hypotheses of Theorem~5 in~\cite{Lannes13}\footnote{in particular, 
it satisfies a stability criterion, which in the shallow water configuration ($\mu\ll1$), can be roughly expressed as 
$\Upsilon\equiv\frac{\epsilon^4}{4}\bo  \frac{\gamma^2(\delta+\gamma)^2}{(1+\gamma)^6}$ is sufficiently small ; see section~5.1.3 of~\cite{Lannes13}.}, as well as~\eqref{CondDepth},\eqref{CondEllipticity}. Then there exists $C,T>0$, independent of $\p$, such that
\begin{itemize}
\item There exists a unique solution $U\equiv (\zeta,\psi)^\top$ to the full Euler system~\eqref{eqn:EulerCompletAdim}, defined on $[0,T]$ and with initial data $(\zeta^0,\psi^0)^\top$ (provided by Theorem~5 in~\cite{Lannes13});
\item There exists a unique solution $U_a\equiv (\zeta_a,v_a)^\top$ to our new model~\eqref{eq:Serre2mr}, defined on $[0,T]$ and with initial data $(\zeta^0,v^0)^\top$ (provided by Theorem~\ref{thbi1mr});
\item With $\bar v\equiv\bar v[\zeta,\psi]$, defined as in~\eqref{eqn:defbarvmr},one has
\[ \forall t\in[0,T], \quad \big\vert (\zeta,\bar v)-(\zeta_a,v_a) \big\vert_{L^\infty([0,t];X^s)} \ \leq \ C\ \mu^2\  t.\]
\end{itemize}
\end{Theorem}

\begin{Remark}
The above proposition is valid for time interval $t\in[0,T/\epsilon]$ with $T$ bounded from below, independently of $\p\in\P_{\rm CH}$, provided that a stronger criterion is satisfied by the initial data; see criterion (5.5) and Theorem~6 in~\cite{Lannes13}.
\end{Remark}

\begin{Remark}
We would like to emphasize here that, contrarily to the full Euler system, our model is well-posed even in the absence of surface tension (the only modification consists in setting $\bo^{-1}=0$ in the constants $\nu,\kappa_1,\kappa_2,\varsigma$). Thus the subtle regularizing effect of surface tension highlighted by Lannes in~\cite{Lannes13} does not play a role in our model.
\end{Remark}

We conclude this section by asserting  that our new model allows to fully justify any well-posed asymptotic model, consistent with our model~\eqref{eq:Serre2mr} in the Camassa-Holm regime.

\begin{Proposition}\label{prop:convergence2mr}
Consider (S) a system of equations such that
\begin{itemize}
\item The Cauchy problem for (S) is well-posed in $X^r$, $r$ sufficiently large.
\item For $U^0\equiv(\zeta^0,v^0)^\top\in H^{s+N}$, $N$ sufficiently large, the solutions of (S) satisfy our model~\eqref{eq:Serre2mr}, up to a remainder $R$ of size $\O(\iota)$ in $L^\infty([0,T];H^s\times H^s)$.
\end{itemize}

Then under the assumptions of Theorem~\ref{th:convergencemr}, the difference between the solution of the full Euler system~\eqref{eqn:EulerCompletAdim}, $U\equiv (\zeta,\psi)^\top$, and the solution of the asymptotic model (S) with corresponding initial data, $U_a\equiv (\zeta_a,v_a)^\top$, is estimated as follows:
\[ \big\vert (\zeta,\bar v[\zeta,\psi])-(\zeta_a,v_a) \big\vert_{L^\infty([0,t];X^s)}\leq C(\iota+\mu^2) t.\]
\end{Proposition}
This procedure of full justification is precisely described in section~\ref{sec:fulljustification}, and applied to the so-called Constantin-Lannes decoupled approximation model.

\section{Construction of the model}
This section is dedicated to the construction of the model we study. 
The key ingredient for constructing shallow water asymptotic models lies in the expansion of the Dirichlet-to-Neumann operators, with respect to the shallowness parameter, $\mu$; see Proposition~\ref{prop:expGH}, below. Thanks to such an expansion, one is able to obtain the so-called Green-Naghdi model (for internal waves), displayed in~\eqref{eqn:GreenNaghdiMean}. This model has been introduced by one of the author in~\cite{Duchene13}, and generalized in~\cite{DucheneIsrawiTalhouk13}. It is justified by a consistency result recalled in Proposition~\ref{prop:ConsGreenNaghdiMean}: roughly speaking, any solution of the full Euler system satisfies the Green-Naghdi asymptotic model up to a small remainder, of size $\O(\mu^2)$. 

In a second step, we use the additional assumption of the Camassa-Holm regime, $\epsilon=\O(\sqrt \mu)$. Simplified models, with the same order of precision may then be deduced. We use several transformations which allow to obtain a well-prepared model: system~\eqref{eq:Serre2}, presented in the previous section~\eqref{eq:Serre2mr}. The justification of this model, in the sense of consistency, is stated in Theorem~\ref{th:ConsSerreVar}. 
The stronger results (well-posedness, stability, convergence) described in Section~\ref{mr} are proved in subsequent sections.

\subsection{The Green-Naghdi model}
The following Proposition is given in~\cite{Duchene13} (see also~\cite{DucheneIsrawiTalhouk13}), extending the result of~\cite{BonaLannesSaut08}.
\begin{Proposition}[Expansion of the Dirichlet-Neumann operators]\label{prop:expGH}
Let $s\geq s_0+1/2,\ s_0>1/2$. Let $\psi$ be such that $\partial_x \psi \in H^{s+11/2}(\RR)$, and $\zeta \in H^{s+9/2}(\RR)$. Let $h_1=1-\epsilon\zeta$ and $h_2=1/\delta+\epsilon\zeta$ such that there exists $h>0$ with $h_1,h_2\geq h>0$.
Then
\begin{align}\label{eqn:expG0}
\big\vert \frac1\mu G^{\mu}\psi  \ -\ \partial_x(h_1\partial_x\psi) \big\vert_{H^s} \ \leq \ \mu\ C_1,\\
\label{eqn:expG}
\big\vert \frac1\mu G^{\mu}\psi  \ -\ \partial_x(h_1\partial_x\psi)\ -\ \mu\frac13\partial_x^2(h_1^3\partial_x^2\psi)\big\vert_{H^s} \ \leq \ \mu^2\ C_3,\\
\label{eqn:expH0}
 \big\vert H^{\mu,\delta}\psi  +\frac{h_1}{h_2}\partial_x\psi  \big\vert_{H^s} \ \leq \ \mu\ C_0,\\
 \label{eqn:expH}
 \big\vert H^{\mu,\delta}\psi  +\frac{h_1}{h_2}\partial_x\psi -\frac\mu{3h_2}\partial_x\Big(h_2^3\partial_x\big(\frac{h_1}{h_2}\partial_x\psi\big)-h_1^3\partial_x^2\psi\Big) \big\vert_{H^s} \ \leq \ \mu^2\ C_2,
 \end{align}
with $C_j=C(\frac1{h},\mu_{\max},\frac1{\delta_{\min}},\delta_{\max},\big\vert \zeta \big\vert_{H^{s+3/2+j}},\big\vert \partial_x\psi \big\vert_{H^{s+5/2+j}} )$. The estimates are uniform with respect to the parameters $\epsilon\in [0,1],\mu \in[0,\mu_{\max}]$, $\delta\in (\delta_{\min},\delta_{\max})$.
\end{Proposition}
\bigskip

Plugging these expansions into the full Euler system~\eqref{eqn:EulerCompletAdim}, and withdrawing $\O(\mu^2)$ terms, immediately yields an asymptotic Green-Naghdi model. This model is presented in~\cite{Duchene13} and justified in the sense of consistency.
However, such a Green-Naghdi model is only one of the variety of models which satisfy such a property. In the following, we decide to introduce an equivalent model, using as unknown (instead of $\psi$, the trace of the upper velocity potential at the interface) the {\em shear mean velocity}, defined by
\begin{equation}\label{eqn:defbarv}
\bar v \ \equiv \  \bar u_2\ -\ \gamma \bar u_1,
\end{equation}
where $\bar u_1, \ \bar u_2$ are the mean velocities integrated across the vertical layer in each fluid:
\[
\bar u_1(t,x) \ = \ \frac{1}{h_1(t,x)}\int_{\epsilon\zeta(t,x)}^{1} \partial_x \phi_1(t,x,z) \ dz, \quad \text{ and }\quad  \bar u_2(t,x) \ = \ \frac{1}{h_2(t,x)}\int_{-\frac1\delta}^{\epsilon\zeta(t,x)} \partial_x \phi_2(t,x,z) \ dz.
\]
Equivalently (as shown in~\cite{Duchene13}), one has
\[ \frac1\mu G^{\mu,\epsilon}\psi \ = \ -\partial_x \Big(\frac{h_1h_2}{h_1+\gamma h_2}\bar v\Big).\]

Such a choice has been used in~\cite{ChoiCamassa96,ChoiCamassa99} for example, and present at least two benefits. First, the equation describing the evolution of the deformation of the interface is an exact equation, and not a $\O(\mu^2)$ approximation. Indeed, one has immediately from the full Euler system~\eqref{eqn:EulerCompletAdim}:
\[ \partial_t\zeta \ = \ \frac1\mu G^{\mu,\epsilon}\psi \ = \ -\partial_x \Big(\frac{h_1h_2}{h_1+\gamma h_2}\bar v\Big).\]
What is more, the system obtained using mean velocities have a nicer behavior with respect to its linear well-posedness, thus one can expect nonlinear well-posedness only for the latter. 
\medskip

Altogether, several technical but straightforward computations yield the following Green-Naghdi model
\begin{equation}\label{eqn:GreenNaghdiMean}
\left\{ \begin{array}{l}
\displaystyle\partial_{ t}{\zeta} \ + \ \partial_x \Big(\frac{h_1h_2}{h_1+\gamma h_2}\bar v\Big)\ =\ 0,  \\ \\
\displaystyle\partial_{ t}\Bigg( \bar v \ + \ \mu\overline{\Q}[h_1,h_2]\bar v \Bigg) \ + \ (\gamma+\delta)\partial_x{\zeta} \ + \ \frac{\epsilon}{2} \partial_x\Big(\dfrac{h_1^2 -\gamma h_2^2 }{(h_1+\gamma h_2)^2} |\bar v|^2\Big) \ = \ \mu\epsilon\partial_x\big(\overline{\R}[h_1,h_2]\bar v  \big) \\
\hfill \displaystyle+\mu\frac{\gamma+\delta}{\bo}\partial_x^3 \zeta \ ,
\end{array}
\right.
\end{equation}
where we define:
    \begin{align*}
        \overline{\Q}[h_1,h_2]V \  &\equiv \ \frac{-1}{3h_1 h_2}\Bigg(h_1 \partial_x \Big(h_2^3\partial_x\big(\frac{h_1\ V}{h_1+\gamma h_2} \big)\Big)\ +\ \gamma h_2\partial_x \Big( h_1^3\partial_x \big(\frac{h_2\ V}{h_1+\gamma h_2}\big)\Big)\Bigg), \\
        \overline{\R}[h_1,h_2]V  \  &\equiv \ \frac12  \Bigg( \Big( h_2\partial_x \big( \frac{h_1\ V}{h_1+\gamma h_2} \big)\Big)^2\ -\ \gamma\Big(h_1\partial_x \big(\frac{h_2\ V}{h_1+\gamma h_2} \big)\Big)^2\Bigg)\\
    &\qquad + \frac13\frac{V}{h_1+\gamma h_2}\ \Bigg( \frac{h_1}{h_2}\partial_x\Big( h_2^3\partial_x\big(\frac{h_1\ V}{h_1+\gamma h_2} \big)\Big) \ - \ \gamma\frac{h_2}{h_1}\partial_x\Big(h_1^3\partial_x \big(\frac{h_2\ V}{h_1+\gamma h_2}\big)\Big) \Bigg).
 \end{align*}
This model has been derived in~\cite{Duchene13} and justified in the sense of consistency, as follows.
 \begin{Proposition}\label{prop:ConsGreenNaghdiMean}
For $\p= (\mu,\epsilon,\delta,\gamma,\bo) \in \P_{\rm SW}$, let $U^\p\equiv(\zeta^\p,\psi^\p)$ be a family of solutions of the full Euler system~\eqref{eqn:EulerCompletAdim} such that
such that there exists $C_0,T>0$ with 
\[ \big\Vert \zeta^\p \big\Vert_{L^\infty([0,T);H^{s+\frac92})}  +  \big\Vert \partial_t\zeta^\p \big\Vert_{L^\infty([0,T);H^{s+\frac72})}  +  \big\Vert \partial_x\psi^\p \big\Vert_{L^\infty([0,T);H^{s+\frac{11}2})}  +  \big\Vert \partial_t\partial_x\psi^\p \big\Vert_{L^\infty([0,T);H^{s+\frac92})}   \leq C_0 ,\]
for given $s\geq s_0+1/2$, $s_0>1/2$, uniformly with respect to $\p\in \P_{\rm SW}$.
 Moreover, assume that there exists $h_{01}>0$ such that
\[
h_1 \ \equiv\  1-\epsilon\zeta^\p \geq h_{01}>0, \quad h_2 \ \equiv \ \frac1\delta +\epsilon \zeta^\p\geq h_{01}>0.
\]
 Define $\bar v^\p$ as in~\eqref{eqn:defbarv} or, equivalently, by
\[\frac1\mu G^{\mu}[\epsilon\zeta^\p]\psi^\p \ = \ -\partial_x \Big(\frac{h_1h_2}{h_1+\gamma h_2}\bar v^\p\Big). \]
Then $(\zeta,\bar v)$ satisfies~\eqref{eqn:GreenNaghdiMean}, up to a remainder $R$, bounded by
\[ \big\Vert R \big\Vert_{L^\infty([0,T);H^s)} \ \leq \  C_1 \mu^2 \]
with $C_1=C(M_{\rm SW},h_{01}^{-1},C_0)$,  uniformly with respect to the parameters $\p\in\P_{\rm SW}$.
\end{Proposition}
\begin{Remark}
In~\cite{Duchene13}, the author works with $\bo^{-1}=0$. However, it is clear that the results are still valid with the surface tension term, using
\[ \left\vert  \frac{-\mu}{\bo}\frac{\partial_x \big(k(\epsilon\sqrt\mu\zeta)\big)}{{\epsilon\sqrt\mu}} \ - \ \frac{\mu}{\bo}\partial_x^3\zeta \right\vert_{H^s} \ \leq \ \frac{\mu^2 \epsilon^2}{\bo}\ C(\mu\epsilon^2,\big\vert \zeta\big\vert_{H^{s+2}}),\] 
where we used Lemma~\ref{lem:f/h1}. Of course, one could have simply kept the surface tension term unchanged at this point, as in~\cite{DucheneIsrawiTalhouk13}. The smallness of surface tension, expressed by $\bo^{-1}\leq \bo_{\min}^{-1}$, is useful in the derivation of our new model, in the following subsection.
\end{Remark}

\subsection{Our new model}\label{ssec:Camassa-Holm}
As announced in the introduction, the present work is limited to the so-called Camassa-Holm regime, that is
using additional assumption $\epsilon=\O(\sqrt\mu)$. In this section, we manipulate the Green-Naghdi system~\eqref{eqn:GreenNaghdiMean}, systematically withdrawing $\O(\mu^2,\mu\epsilon^2)$ terms, in order to recover our model presented in~\eqref{eq:Serre2mr}.
In particular, one can check that the following approximations formally hold:
 \begin{align*}
        \overline{\Q}[h_1,h_2]\b v \ &= \ - \underline{\nu}\partial_x^2 \b v-\epsilon\frac{\gamma+\delta}3 \left((\beta-\alpha)\b v \partial_x^2\zeta \ + \ (\alpha+2\beta)\partial_x(\zeta\partial_x\b v)-\beta\zeta\partial_x^2\b v\right) \ + \ \O(\epsilon^2), \\
        \overline{\R}[h_1,h_2]\b v \  &= \ \alpha\left(\frac12 (\partial_x \b v)^2+\frac13\b v\partial_x^2 \b v\right) \ + \ \O(\epsilon).
 \end{align*}
with \begin{equation}\label{eqn:defnualphabeta}
\underline{\nu}=\frac{1+\gamma\delta}{3\delta(\gamma+\delta)}\ , \quad \alpha=\dfrac{1-\gamma}{(\gamma+\delta)^2} \quad \text{ and } \quad \beta=\dfrac{(1+\gamma\delta)(\delta^2-\gamma)}{\delta(\gamma+\delta)^3}\ .
\end{equation}

Plugging these expansions into system~\eqref{eqn:GreenNaghdiMean} yields a simplified model, with the same order of precision of the original model (that is $\O(\mu^2)$) in the Camassa-Holm regime. However, we will use several additional transformations, in order to produce an equivalent model (again, in the sense of consistency), which possess a structure similar to symmetrizable quasilinear systems, and allows the study of the subsequent sections.
\medskip

The first step is to introduce the following symmetric operator,
 \[
      {\mathfrak T}[\epsilon\zeta]V \ = \ q_1(\epsilon\zeta)V \ - \ \mu \nu \partial_x \Big(q_2(\epsilon\zeta)\partial_xV \Big),
      \]
      where $q_i(\epsilon\zeta)\equiv 1+\kappa_i\epsilon\zeta $ ($i=1,2$) and $\nu,\kappa_1,\kappa_2$ are constants to be determined,
      so as to write
      \[ q_1(\epsilon\zeta) \partial_{ t}\Big( \bar v \ + \ \mu\overline{\Q}[h_1,h_2]\bar v \Big) -q_1(\epsilon\zeta) \mu \frac{\gamma+\delta}{\bo}\partial_x^3\zeta\ = \ {\mathfrak T}[\epsilon\zeta]\partial_t \b v \ + \ \text{higher order terms.}\]

      More precisely, one can check
      \begin{multline*}
        {\mathfrak T}[\epsilon\zeta]\partial_t \b v \ -\ q_1(\epsilon\zeta) \partial_{ t}\Big( \bar v \ + \ \mu\overline{\Q}[h_1,h_2]\bar v \Big)+q_1(\epsilon\zeta) \mu \frac{\gamma+\delta}{\bo}\partial_x^3\zeta\\  =-\mu\nu\partial_x^2\partial_t \b v+\mu\underline{\nu}\partial_x^2\partial_t \b v +\mu\frac{\gamma+\delta}{\bo}\partial_x^3\zeta  - \ \mu\epsilon\nu \kappa_2 \partial_x \Big(\zeta\partial_x\partial_t  \b v \Big) \ + \ \mu\epsilon\underline{\nu} \kappa_1 \zeta \partial_x^2\partial_t  \b v\\
       +\mu\epsilon q_1(\epsilon\zeta)\frac{\gamma+\delta}3\partial_t\Big((\beta-\alpha) \b v \partial_x^2\zeta \ + \ (\alpha+2\beta)\partial_x(\zeta\partial_x \b v)-\beta\zeta\partial_x^2 \b v\Big)+\mu\epsilon\kappa_1\zeta\frac{\gamma+\delta}{\bo}\partial_x^3\zeta .
      \end{multline*}
      The first order ($\O(\mu)$) terms may be canceled with a proper choice of $\nu$, making use of the fact that the second equation of system~\eqref{eqn:GreenNaghdiMean} yields 
      \[ \partial_t  \b v=-(\gamma+\delta)\partial_x\zeta  \ - \ \frac{\epsilon}{2} \partial_x\Big(\dfrac{\delta^2 -\gamma }{(\delta+\gamma )^2} |\bar v|^2\Big) +\O(\epsilon^2,\mu).\]
      Indeed, it follows that 
      \[ \frac{\gamma+\delta}{\bo}\partial_x^3\zeta \ = \  \frac{-1}{\bo}\partial_x^2\partial_t \b v- \frac{\epsilon}{2\bo} \dfrac{\delta^2 -\gamma }{(\delta+\gamma )^2} \partial_x^3\Big(|\bar v|^2\Big) +\O(\epsilon^2,\mu),\]
      thus one defines
      \begin{equation}\label{deftnu}
 \nu \ = \ \underline{\nu}-\frac1{\bo} \ = \ \frac{1+\gamma\delta}{3\delta(\gamma+\delta)}-\frac1{\bo},
      \end{equation}
      and one has
       \begin{multline*}
              {\mathfrak T}[\epsilon\zeta]\partial_t\b v \ -\ q_1(\epsilon\zeta) \partial_{ t}\Big( \bar v \ + \ \mu\overline{\Q}[h_1,h_2]\bar v \Big)+q_1(\epsilon\zeta) \mu \frac{\gamma+\delta}{\bo}\partial_x^3\zeta\\  =\ - \mu\frac{\epsilon}{2\bo} \dfrac{\delta^2 -\gamma }{(\delta+\gamma )^2} \partial_x^3\Big(|\bar v|^2\Big) \  - \ \mu\epsilon\nu \kappa_2 \partial_x \Big(\zeta\partial_x\partial_t  \b v \Big) \ + \ \mu\epsilon\underline{\nu} \kappa_1 \zeta \partial_x^2\partial_t \b v+\mu\epsilon\kappa_1\zeta\frac{\gamma+\delta}{\bo}\partial_x^3\zeta\\
             +\mu\epsilon \frac{\gamma+\delta}3\partial_t\Big((\beta-\alpha) \b v \partial_x^2\zeta \ + \ (\alpha+2\beta)\partial_x(\zeta\partial_x \b v)-\beta\zeta\partial_x^2 \b v\Big)
             + \ \O(\mu^2,\mu\epsilon^2) .
            \end{multline*}
      Use again that~\eqref{eqn:GreenNaghdiMean} yields $ \partial_t \b v=-(\gamma+\delta)\partial_x\zeta  \ +\O(\epsilon,\mu)$ and $ \partial_t\zeta=\frac{-1}{\gamma+\delta}\partial_x \b v+\O(\epsilon,\mu)$, one obtains
      \begin{multline*}
          {\mathfrak T}[\epsilon\zeta]\partial_t \b v \ -\ q_1(\epsilon\zeta) \partial_{ t}\Big( \bar v \ + \ \mu\overline{\Q}[h_1,h_2]\bar v \Big)+q_1(\epsilon\zeta) \mu \frac{\gamma+\delta}{\bo}\partial_x^3\zeta\\ =  \ \mu\epsilon(\gamma+\delta)\Big(\nu \kappa_2 \partial_x \Big(\zeta\partial_x^2 \zeta \Big) \ - \ \underline{\nu} \kappa_1 \zeta\partial_x^3\zeta \Big) +\mu\epsilon\kappa_1\frac{\gamma+\delta}{\bo}\zeta\partial_x^3\zeta\\
       -\mu\epsilon \frac{(\gamma+\delta)^2}3\Big((\beta-\alpha)(\partial_x\zeta)( \partial_x^2\zeta) \ + \ (\alpha+2\beta)\partial_x(\zeta\partial_x^2 \zeta)-\beta\zeta\partial_x^3\zeta \Big)\\
      -\mu\epsilon\frac{1}3  \Big((\beta-\alpha)\b v \partial_x^3 \b v \ + \ (\alpha+2\beta)\partial_x((\partial_x \b v)^2)-\beta(\partial_x \b v)(\partial_x^2 v)
      \Big) \\
      -\mu\epsilon\frac{1}{2\bo} \dfrac{\delta^2 -\gamma }{(\delta+\gamma )^2} \partial_x^3\Big(|\bar v|^2\Big)\ + \ \O(\mu^2,\mu\epsilon^2) .
      \end{multline*}
      It becomes clear, now, that one can adjust $\kappa_1,\kappa_2\in\RR$ so that all terms involving $\zeta$ and its derivatives are withdrawn. More specifically, we set\footnote{Of course, the definition of $\kappa_1,\kappa_2$ in~\eqref{defkappa} and $\varsigma$ in~\eqref{defvarsigma} forbids the particular value $\bo^{-1}=\underline{\nu}=\frac{1+\gamma\delta}{3\delta(\gamma+\delta)}$. Thus in order to be completely rigorous, one should exclude a small neighborhood around this value as for the set parameters for which Theorem~\ref{th:ConsSerreVar} (Theorem~\ref{th:ConsSerreVarmr}) holds true. This restriction is automatically satisfied in the Camassa-Holm regime used thereafter; see~\eqref{eqn:defRegimeCHmr}}
      \begin{equation}\label{defkappa}
      (\underline{\nu}-\frac1{\bo})\kappa_1 \ = \ \frac{\gamma+\delta}{3}(2\beta-\alpha) , \quad   (\underline{\nu}-\frac1{\bo})\kappa_2 \ = \ (\gamma+\delta)\beta ,
      \end{equation}
      and one obtains consequently
      \begin{multline}\label{TversusQ}
        {\mathfrak T}[\epsilon\zeta]\partial_t \b v \ -\ q_1(\epsilon\zeta) \partial_{ t}\Big( \bar v \ + \ \mu\overline{\Q}[h_1,h_2]\bar v \Big) +q_1(\epsilon\zeta) \mu \frac{\gamma+\delta}{\bo}\partial_x^3\zeta \\
          = \ \mu\epsilon \Big(\big\{\frac{-1}{\bo} \dfrac{\delta^2 -\gamma }{(\delta+\gamma )^2}-\frac13(\beta-\alpha)\big\}\b v \partial_x^3 \b v \ + \ \big\{\frac{-3}{2\bo} \dfrac{\delta^2 -\gamma }{(\delta+\gamma )^2}-\frac13 (\alpha+\frac32\beta)\big\}\partial_x((\partial_x \b v)^2)\Big)+   \O(\mu^2,\mu\epsilon^2).
      \end{multline}
      
      However, one of the remaining terms in~\eqref{TversusQ}, as well as in $\partial_x\big(\overline{\R}[h_1,h_2] \b v\big) $, involves three derivatives on $\b v$. In order to deal with these terms, we introduce $  {\mathfrak T}[\epsilon\zeta](\epsilon\varsigma \b  v\partial_x \b v)$ where, again, $\varsigma\in\RR$ is to be determined. More precisely, one has
      \begin{multline*}
        {\mathfrak T}[\epsilon\zeta](\epsilon \varsigma \b v\partial_x  \b v) \ +\ \mu\epsilon q_1(\epsilon\zeta) \partial_{ x}\Big(\overline{\R}[h_1,h_2] \b v\Big)\\
       \quad = \ \epsilon \varsigma q_1(\epsilon\zeta) \b v\partial_x \b  v  \ - \ \mu\epsilon\nu \varsigma \partial_x \Big(q_2(\epsilon\zeta)\partial_x( \b v\partial_x \b v ) \Big) +\mu\epsilon q_1(\epsilon\zeta)\alpha\partial_x\Big(\frac12 (\partial_x  \b v)^2+\frac13 \b v\partial_x^2 \b v\Big) .
      \end{multline*}
This yields
      \begin{multline}\label{TversusR}
       {\mathfrak T}[\epsilon\zeta](\epsilon \varsigma \b v\partial_x \b v)  + \mu\epsilon q_1(\epsilon\zeta) \partial_{ x}\Big(\overline{\R}[h_1,h_2] \b v\Big)  \\ =  \epsilon \varsigma q_1(\epsilon\zeta) \b v\partial_x \b v   + \mu\epsilon \partial_x\Big((\frac\alpha2-\nu\varsigma )(\partial_x \b v)^2+(\frac\alpha3-\nu\varsigma)\b v\partial_x^2 \b v\Big)  +  \O(\mu^2,\mu\epsilon^2) .
      \end{multline}
      
      Combining~\eqref{TversusQ} with~\eqref{TversusR}, one can check that if we set
      \begin{equation}\label{defvarsigma}
  (\underline{\nu}-\frac1{\bo})\varsigma \ = \ \frac{2\alpha-\beta}{3} \ - \ \frac{1}{\bo} \dfrac{\delta^2 -\gamma }{(\delta+\gamma )^2} ,
      \end{equation}
      then the following approximation holds (withdrawing $  \O(\mu^2,\mu\epsilon^2)$ terms):
      \begin{multline*}
       {\mathfrak T}[\epsilon\zeta](\partial_t \b v +\epsilon \varsigma \b v\partial_x \b v) - q_1(\epsilon\zeta) \partial_{ t}\Big( \bar v \ + \ \mu\overline{\Q}[h_1,h_{2}]\bar v \Big) +q_1(\epsilon\zeta) \mu \frac{\gamma+\delta}{\bo}\partial_x^3\zeta+ \mu\epsilon q_1(\epsilon\zeta) \partial_{ x}\Big(\overline{\R}[h_1,h_2] \b v \Big) \\
         =  \epsilon \varsigma q_1(\epsilon\zeta) \b v\partial_x  \b v   -  \mu\epsilon \frac{2\alpha}3\partial_x\big((\partial_x  \b v)^2\big)+\O(\mu^2,\mu\epsilon^2).
      \end{multline*}
      
      When plugging this estimate in~\eqref{eqn:GreenNaghdiMean}, and after multiplying the second equation by $q_1(\epsilon\zeta)$, we obtain the following system of equations:
      \begin{equation}\label{eq:Serre2}\left\{ \begin{array}{l}
      \partial_{ t}\zeta +\partial_x\left(\dfrac{h_1 h_2}{h_1+\gamma h_2} \b v\right)\ =\  0,\\ \\
      \mathfrak T[\epsilon\zeta] \left( \partial_{ t}  \b v + \epsilon\varsigma {\b v } \partial_x { \b v } \right) + (\gamma+\delta)q_1(\epsilon\zeta)\partial_x 
      \zeta   \\
      \qquad \qquad+\frac\epsilon2 q_1(\epsilon\zeta) \partial_x \left(\frac{h_1^2  -\gamma h_2^2 }{(h_1+\gamma h_2)^2}| \b v|^2-\varsigma | \b v|^2\right)=     -  \mu \epsilon\frac23\frac{1-\gamma}{(\gamma+\delta)^2} \partial_x\big((\partial_x  \b v)^2\big) ,
      \end{array} \right. \end{equation}
      where we recall that
      \begin{equation}\label{defT}
      {\mathfrak T}[\epsilon\zeta]V \ = \ q_1(\epsilon\zeta)V \ - \ \mu\nu \partial_x \Big(q_2(\epsilon\zeta)\partial_xV \Big),
      \end{equation}
      with $q_i(X)\equiv 1+\kappa_i X $ ($i=1,2$) and $\nu,\kappa_1,\kappa_2,\varsigma$ are defined by~\eqref{deftnu},\eqref{defkappa} and~\eqref{defvarsigma}.
      \medskip

System~\eqref{eq:Serre2} has been introduced in Section~\ref{mr}, and is the system studied in the present work. We reproduce and prove below Theorem~\ref{th:ConsSerreVarmr}, which asserts the validity of~\eqref{eq:Serre2} as an asymptotic model for the full Euler system in the sense of consistency.
 \begin{Theorem}\label{th:ConsSerreVar}
For $(\mu,\epsilon,\delta,\gamma,\bo)\equiv \p\in \P_{\rm SW}$, let $U^\p\equiv(\zeta^\p,\psi^\p)$ be a family of solutions of the full Euler system~\eqref{eqn:EulerCompletAdim} such that
such that there exists $C_0,T>0$ with 
\[ \big\Vert \zeta^\p \big\Vert_{L^\infty([0,T);H^{s+\frac92})}  +  \big\Vert \partial_t\zeta^\p \big\Vert_{L^\infty([0,T);H^{s+\frac72})}  +  \big\Vert \partial_x\psi^\p \big\Vert_{L^\infty([0,T);H^{s+\frac{11}2})}  +  \big\Vert \partial_t\partial_x\psi^\p \big\Vert_{L^\infty([0,T);H^{s+\frac92})}   \leq C_0 ,\]
for any $s\geq s_0+1/2$, $s_0>1/2$, and uniformly with respect to $\p\in \P_{\rm SW}$. Moreover, assume that there exists $h_{01}>0$ such that
\[
h_1 \ \equiv\  1-\epsilon\zeta^\p \geq h_{01}>0, \quad h_2 \ \equiv \ \frac1\delta +\epsilon \zeta^\p\geq h_{01}>0.
\]

 Define $\bar v^\p$ as in~\eqref{eqn:defbarv} or, equivalently, by
\[\frac1\mu G^{\mu}[\epsilon\zeta^\p]\psi^\p \ = \ -\partial_x \Big(\frac{h_1h_2}{h_1+\gamma h_2}\bar v^\p\Big). \]

Then $(\zeta,\bar v)$ satisfies~\eqref{eq:Serre2}, up to a remainder $R$, bounded by
\[ \big\Vert R \big\Vert_{L^\infty([0,T);H^s)} \ \leq \ (\mu^2+\mu\epsilon^{2})\ C_1,\]
with $C_1=C(M_{\rm SW},h_{01}^{-1},C_0)$,  uniformly with respect to the parameters $\p\in\P_{\rm SW}$.
\end{Theorem}
\begin{proof}
Let $U\equiv(\zeta,\psi)$ satisfy the hypotheses of the Proposition (withdrawing the explicit dependence with respect to parameters $\p$ for the sake of readability). As a consequence of Proposition~\ref{prop:ConsGreenNaghdiMean}, we know that $(\zeta,\bar v)$ satisfies~\eqref{eqn:GreenNaghdiMean}, up to a remainder $R_0$, bounded by
\[ \big\Vert R_0 \big\Vert_{L^\infty([0,T);H^s)} \ \leq \ \mu^2\ C_{1},\]
with $C_1=C(M_{\rm CH},h_{01}^{-1},C_0)$,  uniformly with respect to $(\mu,\epsilon,\delta,\gamma,\bo)\in\P_{\rm SW}$.
The proof now consists in checking that all terms neglected in the above calculations can be rigorously estimated in the same way.

The formal expansions can easily be checked. When turning to control the remainder terms in $H^s$ norm we make great use of classical product estimates in $H^s(\RR)$, $s\geq s_0+1/2$, recalled in Lemma~\ref{Moser}. A technical issue appears when such products involve terms as $\frac1{h_1}$, since $ \frac1{h_1}$ is controlled in $ L^\infty$ (thanks to the non-vanishing depth condition), but not in $H^s$ (as it does not decay at infinity).
We detail in Lemmata~\ref{lem:f/h} and~\ref{lem:f/h1} how such difficulty can be treated.

For the sake of brevity, we do not develop each estimate, but rather provide the precise bound on the various remainder terms.
One has
 \begin{multline*}
 \big\vert    \partial_t\big(   \overline{\Q}[h_1,h_2]V\big) \b v  -   \big[- \underline{\nu}\partial_x^2\partial_t \b v-\epsilon\frac{\gamma+\delta}3 \partial_t\left((\beta-\alpha)\b v \partial_x^2\zeta  +  (\alpha+2\beta)\partial_x(\zeta\partial_x\b v)-\beta\zeta\partial_x^2\b v\right) \big] \big\vert_{H^{s}}\\
  \leq \epsilon^2 C(s+3), \ \end{multline*}
with $C(s+3)\equiv C\Big(M_{\rm CH},h_{01}^{-1},\big\vert \zeta \big\vert_{H^{s+3}},\big\vert \partial_t\zeta  \big\vert_{H^{s+2}},\big\vert \bar v \big\vert_{H^{s+3}},\big\vert \bar v \big\vert_{H^{s+2}}\Big)$, and
\[
    \big\vert     \partial_x\big(  \overline{\R}[h_1,h_2]\b v\big)   -  \partial_x\big[\alpha\big(\frac12 (\partial_x \b v)^2+\frac13\b v\partial_x^2 \b v\big)\big]\big\vert_{H^{s}} \\
    \leq
    \epsilon C(s+3) .\
\]
Then, since $(\zeta,\bar v)$ satisfies~\eqref{eqn:GreenNaghdiMean}, up to the remainder $R_0$, one has
\[\big\vert  \partial_t \b v+(\gamma+\delta)\partial_x\zeta  \big\vert_{H^{s}}+  \big\vert\partial_t\zeta+\frac{1}{\gamma+\delta}\partial_x \b v  \big\vert_{H^{s}}\ \leq \ \epsilon C(s+3) + \big\vert R_0 \big\vert_{H^{s}} .\]
It follows that~\eqref{TversusR} is valid up to a remainder $R_1$, bounded by
\[ \big\vert R_1\big\vert_{H^s} \ \leq \ (\mu^2+\mu\epsilon^2) C(s+3) +\mu(\epsilon+\mu) \big\vert R_0\big\vert_{H^s}.\]
Finally, $(\zeta,\bar v)$ satisfies~\eqref{eqn:GreenNaghdiMean}, up to the remainder $R_0+R_1$, and
\[ \big\vert R_0+R_1\big\vert_{H^s} \ \leq \ \mu^2 C(M_{\rm CH},h_{01}^{-1},C_0) , \]
where we use that
\[\big\vert \bar v  \big\vert_{H^{s+3}}+\big\vert \partial_t \bar v \big\vert_{H^{s+2}}\leq C(M_{\rm CH},h_{01}^{-1},C_0).\]
The estimate on $\bar v$ follows directly from the identity $\partial_x \Big(\frac{h_1h_2}{h_1+\gamma h_2}\bar v\Big)  =  -\frac1\mu G^{\mu,\epsilon}\psi=\partial_t\zeta$. The estimate on $\partial_t\bar v$ can be proved, for example, following~\cite[Prop. 2.12]{Duchene10}. 
This concludes the proof of Theorem~\ref{th:ConsSerreVar}.
\end{proof}

\section{Preliminary results}\label{pr}

In this section, we study the operator $\mfT[\epsilon\zeta]$, defined in~\eqref{defT} and recalled below:
\begin{equation}
{\mfT}[\epsilon\zeta]V= (1+\epsilon \kappa_1\zeta)V-\dsp \mu \nu\partial_x\left((1+\epsilon\kappa_2 \zeta)\partial_x V \right).
\end{equation}
with $\nu,\kappa_1,\kappa_2$ are constants. In our setting, $\nu,\kappa_1,\kappa_2$ depend on the parameters $\gamma,\delta,\bo$; but in what follows, we use only that the restrictions of the Camassa-Holm regime ensures that $\nu>0$ is bounded from below (by hypothesis):
\[
 \nu \ \equiv \ \frac{1+\gamma\delta}{3\delta(\gamma+\delta)}-\frac{1}{\bo} \ge \nu_0>0, 
\]
and  $\nu+|\kappa_1|+|\kappa_2|$ is bounded from above, uniformly with respect to $(\mu,\epsilon,\delta,\gamma,\bo)\in \P_{\rm CH}$ (see~\eqref{eqn:defRegimeCHmr}).
\medskip

When no confusion is possible, we write simply $\mfT\equiv\mfT[\epsilon\zeta]$.
In the following, we seek sufficient conditions to ensure the strong ellipticity of the operator  ${\mfT}$ which will yield to the well-posedness and continuity of the inverse ${\mfT}^{-1}$.

As a matter of fact, this condition, namely~\eqref{CondEllipticity} (and similarly the classical non-zero depth condition,~\eqref{CondDepth}) simply consists in assuming that the deformation of the interface is not too large. For fixed $\zeta\in L^\infty$, the restriction reduces to an estimate on $\epsilon_{\max}\big\vert \zeta\big\vert_{L^\infty}$, with $\epsilon_{\max}=\min(M\sqrt{\mu_{\max}},1)$, and~\eqref{CondDepth}-\eqref{CondEllipticity} hold uniformly with respect to $(\mu,\epsilon,\delta,\gamma,\bo)\in\P_{\rm CH}$; see Lemma~\ref{conditionh}, below. 

Let us shortly detail the argument. Recall the non-zero depth condition
\begin{equation}\label{CondDepth}\tag{H1}
\exists\ h_{01}>0 , \text{ such that } \quad   \min {(\inf_{x\in \RR}h_{1},\inf_{x\in \RR}h_2)}\ge h_{01}\,,
\end{equation}
where $h_1\equiv 1-\epsilon \zeta$ and $h_2\equiv \frac{1}{\delta}+\epsilon \zeta$ are the depth of, respectively, the upper and the lower layer of fluid.
It is straightforward to check that, since for all $(\mu,\epsilon,\delta,\gamma,\bo)\in\P_{\rm CH}$,  the following condition
\[
\epsilon_{\max}\big\vert \zeta\big\vert_{L^\infty} \ < \ \min(1,\frac{1}{\delta_{\max}})
\]
is sufficient to define $h_{01}>0$ such that~\eqref{CondDepth} is valid, independently of $(\mu,\epsilon,\delta,\gamma,\bo)\in\P_{\rm CH}$. 
Briefly, since $\epsilon\leq\epsilon_{\max}$, 
$\inf_{x\in \RR}h_{1} \geq 1-\epsilon \big\vert \zeta\big\vert_{L^\infty} \geq 1-\epsilon_{\max}\big\vert \zeta\big\vert_{L^\infty}, \ \inf_{x\in \RR}h_{2} \geq \frac1\delta-\epsilon \big\vert \zeta\big\vert_{L^\infty} \geq \frac1\delta-\epsilon_{\max}\big\vert \zeta\big\vert_{L^\infty}.$
Note that conversely, for~\eqref{CondDepth} to be satisfied for any $(\epsilon,\mu,\delta,\gamma,\bo)\in\P_{\rm CH}$, then one needs
\[ \epsilon_{\max}\big\vert \zeta\big\vert_{L^\infty} \ \leq \ \max(1,\frac{1}{\delta_{\min}}).\]
Indeed, one has for $\epsilon=\epsilon_{\max}$, $\epsilon_{\max}\zeta \ = \ 1-h_1 \ \leq \ 1-h_{01} $  and $ -\epsilon_{\max}\zeta \ = \ \frac1\delta-h_2 \ \leq \ \frac1\delta -h_{01}$.

In the same way, we introduce the condition
\begin{equation}\label{CondEllipticity}\tag{H2}
\exists\ h_{02}>0 , \text{ such that } \quad  \inf_{x\in \RR} \left(1+\epsilon\kappa_2\zeta\right) \ge  h_{02} \ > \ 0 \ ; \qquad    \inf_{x\in \RR}  \left( 1+\epsilon \kappa_1\zeta  \right)\ge h_{02} \ > \ 0.
\end{equation}

As above, such a condition is a consequence of a simple smallness assumption on $\epsilon\big\vert \zeta\big\vert_{L^\infty}$. More precisely, one has the following result.
\begin{Lemma}\label{conditionh} Let $\zeta\in L^{\infty}$ and $\epsilon_{\max}=\min(M\sqrt{\mu_{\max}},1)$ be such that there exists $h_0>0$ with
\[ \max(|\kappa_1|,|\kappa_2|,1,\delta_{\max})\epsilon_{\max}\big\vert \zeta\big\vert_{L^\infty} \  \leq \ 1-h_0 \ < \ 1.\]
 Then there exists $h_{01},h_{02}>0$ such that~\eqref{CondDepth}-\eqref{CondEllipticity} hold for any $(\mu,\epsilon,\delta,\gamma,\bo)\in\P_{\rm CH}$.
\end{Lemma}
\medskip

{\em In what follows, we will always assume that~\eqref{CondDepth} and~\eqref{CondEllipticity} are satisfied}. It is a consequence of our work that such assumption may be imposed only on the initial data, and then is automatically satisfied over the relevant time scale.
\bigskip

Before asserting the strong ellipticity of the operator  ${\mfT}$, let us
first recall the quantity $\vert \cdot \vert_{H^{1}_\mu}$, which is defined as
\[\forall v\in H^1(\RR),\ \quad 
\vert\ v\ \vert^2_{H^{1}_\mu}\ =\ \vert\ v\ \vert^2_{L^2}\ +\ \mu\ \vert\ \partial_x v\ \vert^2_{L^2},
\]
 and is equivalent to the $H^1(\RR)$-norm but not uniformly with respect to $\mu$.
 We define by $H^1_\mu(\RR)$ the space $H^1(\RR)$ endowed with this norm.

\begin{Lemma}\label{Lem:mfT}
Let $(\mu,\epsilon,\delta,\gamma,\bo)\in\P_{\rm CH}$ and $\zeta \in L^{\infty}(\RR)$ such that~\eqref{CondEllipticity} is satisfied.
 Then the operator
\[
{\mfT}[\epsilon\zeta]: H^1_\mu(\RR)\longrightarrow (H^1_\mu(\RR))^\star
\]
is uniformly continuous and coercive. More precisely, there exists $c_0>0$ such that
\begin{align}
({\mfT} u,v)  \ & \leq \  c_0\vert u\vert_{H^1_\mu}\vert v\vert_{H^1_\mu}  ; \label{continous} \\
({\mfT} u,u) \ & \geq  \frac1{c_0}\vert u\vert_{H^1_\mu}^{2} \label{coercive}
\end{align}
with $c_0=C(M_{\rm CH},h_{02}^{-1},\epsilon\big\vert \zeta\big\vert_{L^\infty} )$.
\medskip

Moreover, the following estimates hold:
\begin{enumerate}
\item[(i)] Let $ s_0>\frac{1}{2}$ and $s\geq0$. If $\zeta \in H^{s_0}(\RR)\cap H^{s} (\RR)$ and $u\in  H^{s+1}(\RR)$ and $v\in H^{1}(\RR)$, then:
\begin{align}
\label{eq:estLambdaT}\big\vert \big(\Lambda^s \mfT[\epsilon\zeta]  u, v \big) \big\vert \ &\leq \  C_0 \left((1+\epsilon \big\vert \zeta \big\vert_{H^{s_0}})\big\vert u\big\vert_{H^{s+1}_\mu} +\big\langle \epsilon \big\vert \zeta \big\vert_{H^{s}}\big\vert u\big\vert_{H^{s_0+1}_\mu} \big\rangle_{s>s_0}\right) \big\vert v\big\vert_{H^{1}_\mu}  \ ,
\end{align}
\item[(ii)] Let $ s_0>\frac{1}{2}$ and $s\geq0$. If $\zeta \in H^{s_0+1}\cap H^{s} (\RR)$, $u\in H^s(\RR)$ and $v\in H^{1}(\RR)$, then:
\begin{align}
\label{eq:estComT}\big\vert \big( \big[\Lambda^s, \mfT[\epsilon\zeta]\big]u, v\big) \big\vert \ &\leq \ \epsilon \ C_0 \left(\big\vert \zeta \big\vert_{H^{s_0+1}}\big\vert u\big\vert_{H^{s}_\mu} +\big\langle \big\vert \zeta \big\vert_{H^{s}}\big\vert u\big\vert_{H^{s_0+1}_\mu}\big\rangle_{s>s_0+1}\right) \big\vert v\big\vert_{H^{1}_\mu} \ ,
\end{align}
\end{enumerate}
where $C_0=C(M_{\rm CH},h_{02}^{-1})$.
\end{Lemma}
\begin{proof}
Let us define the bilinear form
\[
a(u,v)=\big(\ {\mfT} u\ ,\ v\ \big)=\big(\ (1+\epsilon \kappa_1\zeta)u\ ,\ v\ \big)+\nu\mu \big(\ (1+\epsilon\kappa_2\zeta)\partial_x u\ ,\ \partial_x v\ \big),
\]
where $\big(\ \cdot\ , \ \cdot \ ) $ denotes the $L^2$-based inner product. It is straightforward to check that
\[
\big| a(u,v)\big|\leq \sup_{x\in\RR}|1+\epsilon \kappa_1\zeta|\big(\ u\ ,\ v\ \big)+\mu\nu \sup_{x\in\RR}|1+\epsilon\kappa_2\zeta| \big(\ \partial_x u\ ,\ \partial_x v\ \big),\]
so that~\eqref{continous} is now straightforward, by Cauchy-Schwarz inequality.

The $H^1_\mu(\RR)$-coercivity of $a(\cdot,\cdot)$, inequality~\eqref{coercive}, is a consequence of conditions~\eqref{CondEllipticity}:
\begin{align*}
a(u,u)=\big(\ {\mfT} u\ ,\ u\ \big)& = \dsp\int_{\RR}(1+\epsilon \kappa_1\zeta)\vert u\vert^2\,dx + \nu\mu \dsp\int_{\RR}(1+\epsilon\kappa_2\zeta)\vert u_x\vert^2\,dx \\
&\geq  h_{02}\min(1,\nu_0)\vert u\vert_{H^1_\mu}^{2} .
\end{align*}
\medskip

Let us now prove the higher-order estimates of the Lemma, starting with the product estimates. One has
\[
\big(\ \Lambda^s {\mfT} u\ ,\  v\ \big)=\big(\ \Lambda^s\{(1+\epsilon \kappa_1\zeta)u\}\ ,\  v\ \big)+\nu\mu \big(\ \Lambda^s\{(1+\epsilon\kappa_2\zeta)\partial_x u\}\ ,\  \partial_x v\ \big).
\]
Estimate~\eqref{eq:estLambdaT} is now a straightforward consequence of Cauchy-Schwarz inequality, and Lemma~\ref{Moser}.

As for the commutator estimates, one uses
\[
\big(\ [\Lambda^s, {\mfT} ] u\ ,\  v\ \big)=\epsilon \kappa_1 \big(\  [\Lambda^s, \zeta] u \ ,\  v\ \big)+\nu\mu\epsilon \kappa_2 \big(\ [\Lambda^s,\zeta]\partial_x u\ ,\  \partial_x v\ \big).
\]
Estimates~\eqref{eq:estComT} follow, using again Cauchy-Schwarz inequality, and Lemma~\ref{K-P}.
\end{proof}
The following lemma offers an important invertibility result on ${\mfT}$.
\begin{Lemma}\label{proprim}
Let $(\mu,\epsilon,\delta,\gamma,\bo)\in\P_{\rm CH}$ and $\zeta \in L^{\infty}(\RR)$ such that~\eqref{CondEllipticity} is satisfied.
 Then the operator
\[
{\mfT}[\epsilon\zeta]: H^2(\RR)\longrightarrow L^2(\RR)
\]
is one-to-one and onto. Moreover, one has the following estimates:
\begin{enumerate}
\item[(i)] $({\mfT}[\epsilon\zeta])^{-1}:L^2\to H^1_\mu(\RR)$ is continuous. More precisely, one has
\[
\parallel {\mfT}^{-1}\parallel_{L^2(\RR)\rightarrow H^1_\mu(\RR)} \ \leq \  c_0,
\]
with $c_0=C(M_{\rm CH},h_{02}^{-1}, \epsilon\big\vert \zeta\big\vert_{L^\infty} )$.
\item[(ii)] Additionally, if $\zeta \in H^{s_0+1}(\RR)$ with $s_0>\frac{1}{2}$, then one has for any $0< s\leq s_0 +1$,
\[
\parallel {\mfT}^{-1}\parallel_{H^s(\RR)\rightarrow H^{s+1}_\mu(\RR)}\ \leq \ c_{s_0+1}.
\]
\item[(iii)] If $\zeta \in H^{s}(\RR)$ with $ s\ge s_0 +1,\; s_0>\frac{1}{2}$, then one has
\[
\parallel {\mfT}^{-1}\parallel_{H^s(\RR)\rightarrow H^{s+1}_\mu(\RR)}\ \leq \ c_{s}
\]
\end{enumerate}
where $c_{\bar s}=C(M_{\rm CH},h_{02}^{-1},\epsilon|\zeta|_{H^{\bar s}})$, thus uniform with respect to $(\mu,\epsilon,\delta,\gamma,\bo)\in \P_{\rm CH}$.
\end{Lemma}
\begin{proof}
 To show the invertibility of $\mfT$ we use the Lax-Milgram theorem. From the previous Lemma, we know that the bilinear form:
\[
a(u,v)=\big(\ {\mfT} u\ ,\ v\ \big)=\big(\ (1+\epsilon \kappa_1\zeta)u\ ,\ v\ \big)+\mu\nu \big(\ (1+\epsilon\kappa_2\zeta)\partial_x u\ ,\ \partial_x v\ \big)
\]
is continuous and uniformly coercive on $H^1_\mu(\RR)$.
For any $\mu>0$, the dual of $H^1_\mu(\RR)$ is $H^{-1}(\RR)$, of whom $L^2(\RR)$ is a subspace, and one has
$ \big( f,g\big)\leq \big\vert f\big\vert_{H^1_\mu}\big\vert g\big\vert_{L^2}$, {\em independently of }$\mu>0$. Therefore, using Lax-Milgram lemma, for all $f \in L^2(\RR)$,
there exists a unique $u\in H^1_\mu(\RR)$
such that, for all $v\in H^1_\mu(\RR)$
\[
a(u,v)=(f,v);
\]
equivalently, there is a unique variational solution to the equation
\begin{equation}\label{reeq}
{\mfT} u=f.
\end{equation}
We then get from the definition of ${\mfT}$ that
\begin{equation}\label{regularity}
\nu\mu\left( 1+\epsilon\kappa_2 \zeta\right) \partial_x^2u= (1+\epsilon \kappa_1\zeta)u-\mu\epsilon\nu\kappa_2(\partial_x\zeta)( \partial_x u)-f.
\end{equation}
Now, using condition~\eqref{CondEllipticity}, and since $u\in H^1(\RR)$, $\zeta\in L^{\infty}(\RR)$ and $f\in L^2(\RR)$, we deduce that $\partial_x^2u \in L^2(\RR)$, and thus $u\in H^2(\RR)$. We proved that ${\mfT}[\epsilon\zeta]: H^2(\RR)\longrightarrow L^2(\RR)$ is one-to-one and onto.
\medskip

Let us now turn to the proof of estimates in $(i)-(iii)$.

We start from the equality
$a(u,u)=(f,u)$. Using elliptic inequality~\eqref{coercive} and Cauchy-Schwarz inequality, one has
\[ \frac1{c_0}\vert u\vert_{H^1_\mu}^{2}  \ \leq \ a(u,u) \ = \ \big(\ f\ ,\ u\ \big) \ \leq \ \big\vert f\big\vert_{L^2}\big\vert u\big\vert_{L^2} \ \leq \ \big\vert f\big\vert_{L^2}\big\vert u\big\vert_{H^1_\mu}.\]
Dividing by $ c_0^{-1}\big\vert u\big\vert_{H^1_\mu}$ yields the estimate in $(i)$.

Let us now assume that $f\in H^s(\RR)$, for $s\geq 0$. We apply  $\Lambda^s$ to  equation~\eqref{reeq} and we write it under the form:
\[
{\mfT} (\Lambda^s u)=\Lambda^s f-[\Lambda^s,{\mfT}]u.
\]
Proceeding as above, we use the $L^2$ inner product with $\Lambda^s u$, and deduce
\begin{align}\frac1{c_0}\vert \Lambda^s u\vert_{H^{1}_\mu}^{2} \ &\leq \  a(\Lambda^s u,\Lambda^s u) \ = \ \big( {\mfT}\Lambda^s u,\Lambda^s u\big) \ = \  \big( \Lambda^s f-[\Lambda^s,{\mfT}]u \ , \ \Lambda^s u \big) \ \nn \\
 &\leq \ \big\vert \Lambda^s f \big\vert_{L^2} \big\vert \Lambda^s u \big\vert_{L^2} \ + \ \left|\big( [\Lambda^s,{\mfT}]u \ , \ \Lambda^s u \big) \right| .\label{eq:continuity}
\end{align}
The result is now a consequence of~\eqref{eq:estComT}.
\medskip

\noindent{\em  --- If $0\leq s\leq s_0+1$}, one has
\[ \frac1{c_0}\vert  u\vert_{H^{s+1}_\mu}^{2} \ \leq \  \big\vert  f \big\vert_{H^s} \big\vert  u \big\vert_{H^s} \ + \ \epsilon \ C_0\big\vert \zeta \big\vert_{H^{s_0+1}}\big\vert u\big\vert_{H^{s}_\mu}  \big\vert u\big\vert_{H^{s+1}_\mu}  ,\]
thus
\[ \vert  u\vert_{H^{s+1}_\mu} \ \leq \  c_0\ \Big(\ \big\vert  f \big\vert_{H^s} \ + \ \epsilon \ C_0 \big\vert \zeta \big\vert_{H^{s_0+1}}\big\vert u\big\vert_{H^{s}_\mu}\ \Big)   .\]
The estimate of {\em (ii)} for $0<s\leq1$ follows, using estimate {\em (i)} and $\big\vert \Lambda^{s-1} u \big\vert_{H^1_\mu}\leq \big\vert u \big\vert_{H^1_\mu}$. The result for greater values of $s$, $1<s\leq s_0+1$, follows by continuous induction.
\medskip

\noindent{\em --- If $s>s_0+1$}, then plugging~\eqref{eq:estComT} into~\eqref{eq:continuity} yields
\[ \frac1{c_0}\vert  u\vert_{H^{s+1}_\mu}^{2} \ \leq \  \big\vert  f \big\vert_{H^s} \big\vert  u \big\vert_{H^s} \ + \ \epsilon \ C_0\ \big(\big\vert \zeta \big\vert_{H^{s_0+1}}\big\vert u\big\vert_{H^{s}_\mu} + \big\vert \zeta \big\vert_{H^{s}}\big\vert u\big\vert_{H^{s_0+1}_\mu} \big)   \big\vert u\big\vert_{H^{s+1}_\mu}  ,\]
thus
\[ \vert  u\vert_{H^{s+1}_\mu} \ \leq \  c_0\ \Big(\ \big\vert  f \big\vert_{H^s} \ + \ \epsilon \ C_0\  \big(\big\vert \zeta \big\vert_{H^{s_0+1}}\big\vert u\big\vert_{H^{s}_\mu} + \big\vert \zeta \big\vert_{H^{s}}\big\vert u\big\vert_{H^{s_0+1}_\mu} \big) \Big)  .\]
As above, the result follows by continuous induction on $s$.
\end{proof}
\bigskip

Finally, let us introduce the following technical estimate, which is used several times in the subsequent sections.
\begin{Corollary}\label{col:comwithT}
Let $(\mu,\epsilon,\delta,\gamma,\bo)\in\P_{\rm CH}$ and $\zeta \in H^{s}(\RR)$ with $ s\ge s_0 +1,\; s_0>\frac{1}{2}$, such that~\eqref{CondEllipticity} is satisfied. Assume moreover that $u\in H^{s-1}(\RR)$ and that $v\in H^{1}(\RR)$. Then one has
\begin{align}
\big|\big( \ \big[\Lambda^s, \mfT^{-1}[\epsilon\zeta]\big] u\ ,\ \mfT[\epsilon\zeta]  v\ \big)\big|  \ &= \ \big|\big(\  \big[\Lambda^s,\mfT[\epsilon\zeta] \big] {\mfT}^{-1}[\epsilon\zeta]  u\ ,\  v\ \big)\big| \nn \\
& \leq \ \epsilon\ C(M_{\rm CH},h_{02}^{-1},\big\vert \zeta \big\vert_{H^s})\big\vert u \big\vert_{H^{s-1}} \big\vert v \big\vert_{H^{1}_\mu}
\label{eq:comwithT}\end{align}
\end{Corollary}
\begin{proof}
The first identity can be obtained through simple calculation: using that $\mfT[\epsilon\zeta] $ is symmetric,
\begin{align*} \big( \big[\Lambda^s, \mfT^{-1}[\epsilon\zeta]\big] u,\mfT[\epsilon\zeta] v\big) \ &= \ \big( \mfT[\epsilon\zeta]  \big[\Lambda^s, \mfT^{-1}[\epsilon\zeta]\big] u, v\big) \\
&= \ \big( \mfT[\epsilon\zeta]  \Lambda^s \mfT^{-1}[\epsilon\zeta] u \ -\ \Lambda^s u \  ,\  v\big) \\
&= \ \big(-\big[\Lambda^s,\mfT[\epsilon\zeta] \big] {\mfT}^{-1}[\epsilon\zeta]  u \  ,\  v\big) .
\end{align*}

The estimate is now a direct application of~\eqref{eq:estComT} and~Lemma~\ref{proprim}. From point {\it (ii)} and {\it (iii)} of Lemma~\ref{proprim}, one has
\[ \big\vert {\mfT}^{-1}[\epsilon\zeta]  u \big\vert_{H^s_\mu}\ \leq \ C \big\vert u \big\vert_{H^{s-1}} ,\]
with $C=C(M_{\rm CH},h_{02}^{-1},\epsilon\big\vert \zeta \big\vert_{H^{s-1}},\epsilon\big\vert \zeta \big\vert_{H^{s_0+1}})$. Now apply commutator estimate~\eqref{eq:estComT}, and one obtains straightforwardly our desired estimate.
\end{proof}

\section{Linear analysis}\label{la}
Let us recall the system~\eqref{eq:Serre2} introduced in section~\ref{ssec:Camassa-Holm}.
\begin{equation}\label{eqn:Serre2var}\left\{ \begin{array}{l}
\partial_{ t}\zeta +\partial_x\left(\dfrac{h_1 h_2}{h_1+\gamma h_2}\b v\right)\ =\  0,\\ \\
\mathfrak T[\epsilon\zeta] \left( \partial_{ t} \b v + \epsilon \varsigma {\b v } \partial_x {\b v } \right) + (\gamma+\delta)q_1(\epsilon\zeta)\partial_x \zeta   \\
\qquad \qquad+\frac\epsilon2 q_1(\epsilon\zeta) \partial_x \left(\frac{h_1^2  -\gamma h_2^2 }{(h_1+\gamma h_2)^2}|\b v|^2-\varsigma |\b v|^2\right)=  -  \mu \epsilon\frac23\frac{1-\gamma}{(\gamma+\delta)^2} \partial_x\big((\partial_x \b v)^2\big) ,
\end{array} \right. \end{equation}
with $h_1=1-\epsilon \zeta$ , $h_2=1/\delta+\epsilon \zeta$ $q_i(X)=1+\kappa_i X$ ($i=1,2$) , $\kappa_i,\varsigma$ defined in\eqref{defkappa},\eqref{defvarsigma}, and
\[
\mfT[\epsilon\zeta] V= q_1(\epsilon\zeta)V -\dsp \mu \nu\partial_x \left(q_2(\epsilon\zeta)\partial_x V \right).
\]

In order to ease the reading, we define the function
\[ f:X\to\frac{(1-X)(\delta^{-1}+X)}{1-X+\gamma(\delta^{-1}+X)}.\]
One can easily check that
\[
f(\epsilon\zeta) \ =\ \dsp\frac{h_1h_2}{h_1+\gamma h_2}, \qquad \text{ and } \qquad
f'(\epsilon\zeta) \ =\ \dsp\frac{h_1^2-\gamma h_2^2}{(h_1+\gamma h_2)^2}.\]
Additionally, let us denote \[
\kappa=\frac23\frac{1-\gamma}{(\delta+\gamma)^2}
\quad \text{ and } \quad q_3(\epsilon\zeta)=\frac12\big(\frac{h_1^2-\gamma h_2^2}{(h_1+\gamma h_2)^2}-\varsigma\big),\]
so that one can rewrite (here and in the following, we omit the bar on $v$ for the sake of readability)
\begin{equation}\label{eqn:Serre2varf}\left\{ \begin{array}{l}
\dsp \partial_{ t}\zeta +f(\epsilon\zeta)\partial_x  v+\epsilon \partial_x\zeta f'(\epsilon\zeta)  v  \ =\  0,\\ \\
\dsp {\mathfrak T} \left( \partial_{ t}  v + \frac{\epsilon}{2} \varsigma \partial_x({v }^2) \right) + (\gamma+\delta)q_1(\epsilon\zeta)\partial_x \zeta + \epsilon q_1(\epsilon\zeta)\partial_x(q_3(\epsilon\zeta)  {v}^2)  \ + \ \mu\epsilon\kappa\partial_x\big((\partial_x v)^2\big) \ = \ 0.
\end{array} \right. \end{equation}

The equations can be written after applying ${\mathfrak T}^{-1}$ to the second equation
in~\eqref{eqn:Serre2varf} as
\begin{equation}\label{condensedeq}
\partial_tU+A_0[U]\partial_xU +A_1[U]\partial_xU\ = \ 0,\end{equation}
 with
\begin{equation}\label{defA0A1}
A_0[U]=\begin{pmatrix}
\epsilon f'(\epsilon \zeta) v&f(\epsilon\zeta)\\
\mfT^{-1}(Q_0(\epsilon\zeta) \cdot)& \epsilon \mfT^{-1}(\mfQ[\epsilon\zeta,v] \cdot)
\end{pmatrix}, \quad A_1[U]=\begin{pmatrix}
0&0\\
\epsilon^2\mfT^{-1}(Q_1(\epsilon\zeta,v) \cdot)& \epsilon\varsigma v
\end{pmatrix},
\end{equation}
where $Q_0(\epsilon\zeta),Q_1(\epsilon\zeta,v)$ are defined as
\begin{equation}\label{defQ0Q1}
Q_0(\epsilon\zeta) \ = \ (\gamma+\delta)q_1(\epsilon\zeta),\quad Q_1(\epsilon\zeta,v)=q_1(\epsilon\zeta) q_3'(\epsilon \zeta){ v}^2
\end{equation}
and the operator $\mfQ[\epsilon\zeta,v]$ defined by
\begin{equation}\label{defmfQ}
\mfQ[\epsilon\zeta,v]f \ \equiv \ 2q_1(\epsilon\zeta)q_3(\epsilon\zeta)v f +\mu\kappa \partial_x(f \partial_x v).
\end{equation}

Following the classical theory of hyperbolic systems, the well-posedness of the initial value problem of the above system will rely on a precise study of the properties, and in particular energy estimates, for the linearized system around some reference state
$\underline{U}=(\underline{\zeta},\underline{v})^\top$:
\begin{equation}\label{SSlsys}
	\left\lbrace
	\begin{array}{l}
	\dsp\partial_t U+A_0[\underline{U}]\partial_x U+A_1[\underline{U}]\partial_x U=0;
        \\
	\dsp U_{\vert_{t=0}}=U_0.
	\end{array}\right.
\end{equation}

In the following subsection, we construct the natural energy space for our problem. Energy estimates are then proved in section~\ref{ssec:energyestimates}. Finally, we state the well-posedness of the linear system~\eqref{SSlsys} in section~\ref{ssec:linearWP}.

\subsection{Energy space}\label{ssec:energyspace}
Let us first remark that by construction, one has a pseudo-symmetrizer of the system, given by
\begin{equation}\label{defS}
S[\underline{U}]=\begin{pmatrix}
 \frac{Q_0(\epsilon\underline{\zeta})}{f(\epsilon\underline{\zeta})}& 0 \\
0&\mfT[\epsilon\underline{\zeta}]
\end{pmatrix}
, \qquad S[\underline{U}]A_0[\underline{U}]=\begin{pmatrix}
 \epsilon\frac{Q_0(\epsilon\underline{\zeta})}{f(\epsilon\underline{\zeta})}f'(\epsilon\underline{\zeta}) \underline{v}& Q_0(\epsilon\underline{\zeta}) \\
Q_0(\epsilon\underline{\zeta}) &\epsilon \mfQ[\epsilon\underline{\zeta},\underline{v}]
\end{pmatrix}.
\end{equation}

Notice that $S[\underline{U}]$ is symmetric, as $\mfT[\epsilon\underline{\zeta}]$ is symmetric, and one can easily check that $S[\underline{U}]A_0[\underline{U}]$ is symmetric as well. On the contrary, the operator
\[S[\underline{U}]A_1[\underline{U}] \ = \ \begin{pmatrix}
0& 0 \\
\epsilon^2 Q_1(\epsilon\underline{\zeta},\underline{v}) & \epsilon\varsigma \mfT[\epsilon\underline{\zeta}](\underline{v} \cdot)
\end{pmatrix}\] represents the defect of symmetry. However, $\mfT[\epsilon\underline{\zeta}] (\underline{v} \cdot)$ has the desired following property:
\begin{align} \big(\mfT[\epsilon\underline{\zeta}] (\underline{v} \partial_x V) ,V\big)&= \big(\ q_1(\epsilon\underline{\zeta})\underline{v} \partial_x V+\mu\partial_x(q_2(\epsilon\underline{\zeta})\partial_x(\underline{v} \partial_x V ))\ ,\ V\ \big) \nn \\
&= -\frac12\big(\ \partial_x(q_1(\epsilon\underline{\zeta})\underline{v}) V\ ,\ V\ \big)-\mu\big(\ q_2(\epsilon\underline{\zeta})\partial_x(\underline{v} \partial_x V )\ ,\ \partial_x V\ \big)\nn \\
&= -\frac12\big( \partial_x(q_1(\epsilon\underline{\zeta})\underline{v}) V , V \big)-\mu\big( q_2(\epsilon\underline{\zeta})(\partial_x\underline{v}) \partial_x V  , \partial_x V \big)+\mu\frac12\big( \partial_x( q_2(\epsilon\underline{\zeta})\underline{v} )\partial_x V  , \partial_x V \big).\label{eqn:symmetry-like}
\end{align}
Therefore, the inner product $\big(\mfT[\epsilon\underline{\zeta}] (\underline{v} \partial_x V) ,V\big)$ is controlled by $\big|V\big|_{H^1_\mu}^2$, which is bounded in our energy space, as defined below. In the same way, one can control the contribution of $\epsilon^2 Q_1(\epsilon\underline{\zeta},\underline{v}) $, using the smallness of $\epsilon$ through the assumption of the Camassa-Holm regime $\epsilon=\O(\sqrt\mu)$.
\begin{Remark}
In the analysis below, the only place where the smallness assumption of the Camassa-Holm regime, $\epsilon=\O(\sqrt\mu)$, is used (apart from the simplifications it offers when constructing system~\eqref{eqn:Serre2var}), stands in the estimation of the contribution of $\epsilon^2 Q_1(\epsilon\underline{\zeta},\underline{v}) $. 
As a matter of fact, this assumption is actually not required: one could replace $Q_0(\epsilon\underline{\zeta})$ by $Q_0(\epsilon\underline{\zeta})+\epsilon^2 Q_1(\epsilon\underline{\zeta},\underline{v})$ in the pseudo-symmetrizer $S[\underline{U}]$, thus canceling out the $\epsilon^2 Q_1(\epsilon\underline{\zeta},\underline{v}) $ term.
Note however 
that the energy would then be slightly different than the one defined below, and that in order for Lemma~\ref{lemmaes} to hold, one should add an additional smallness assumption on $\epsilon \underline{v}$, in order to ensure $Q_0(\epsilon\underline{\zeta})+\epsilon^2 Q_1(\epsilon\underline{\zeta},\underline{v})\geq h_{03} >0$.
\end{Remark}

Let us now define our energy space.
\begin{Definition}\label{defispace}
 For given $s\ge 0$ and $\mu,T>0$, we denote by $X^s$ the vector space $H^s(\RR)\times H^{s+1}(\RR)$ endowed with the norm
\[
\forall\; U=(\zeta,v) \in X^s, \quad \vert U\vert^2_{X^s}\equiv \vert \zeta\vert^2 _{H^s}+\vert v\vert^2 _{H^s}+ \mu\vert \partial_xv\vert^2 _{H^s},
\]
while $X^s_T$ stands for the space of $U=(\zeta,v)$ such that $U\in C^0([0,\frac{T}{\epsilon}];X^{s})$ and $\partial_t\zeta \in  L^\infty([0,\frac{T}{\epsilon}]\times \RR)$, endowed with the canonical norm
\[
 \Vert U\Vert_{X^s_T}\equiv \sup_{t\in [0,T/\epsilon]}\vert U(t,\cdot)\vert_{X^s}+\esssup_{t\in [0,T/\epsilon],x\in\RR}\vert \partial_t \zeta(t,x)\vert.
\]
\end{Definition}

A natural energy for the initial value problem~\eqref{SSlsys} is given by
\begin{equation}\label{es}
 E^s(U)^2=(\Lambda^sU,S[\underline{U}]\Lambda^sU)=( \Lambda^s\zeta,\frac{Q_0(\epsilon\underline{\zeta})}{f(\epsilon\underline{\zeta})}\Lambda^s\zeta)+\left(\Lambda^s v,\frac{\mfT[\epsilon\underline{\zeta}]}{\gamma+\delta} \Lambda^s v\right).
\end{equation}
The link between $ E^s(U)$ and the $X^s$-norm is investigated in the following Lemma.
\begin{Lemma}\label{lemmaes}
 Let $\p=(\mu, \epsilon,\delta,\gamma,\bo) \in \P_{\rm CH}$,  $s\geq 0$ and $ \underline{\zeta}\in L^{\infty}(\RR)$, satisfying~\eqref{CondDepth} and~\eqref{CondEllipticity}. Then
$E^s(U)$ is equivalent to the $\vert \cdot\vert_{X^s}$-norm  uniformly with respect to $\p\in \P_{\rm CH}$. More precisely, there exists $c_0=C(M_{\rm CH},h_{01}^{-1},h_{02}^{-1})>0$ such that
\[
\frac1{c_0}E^s(U) \ \leq \ \big\vert U \big\vert_{X^s} \ \leq \  c_0 E^s(U).
\]
\end{Lemma}

\begin{proof}
This is a straightforward application of Lemma~\ref{Lem:mfT}, and that for $Q_0(\epsilon\underline{\zeta}),f(\epsilon\underline{\zeta})>0$,
\begin{equation}\label{boundf}
\inf_{x\in\RR} \frac{Q_0(\epsilon\underline{\zeta})}{f(\epsilon\underline{\zeta})}\geq  (\inf_{x\in\RR} Q_0(\epsilon\underline{\zeta}))(\sup_{x\in\RR} f(\epsilon\underline{\zeta}))^{-1}\qquad \sup_{x\in\RR} \left|\frac{Q_0(\epsilon\underline{\zeta})}{f(\epsilon\underline{\zeta})}\right|\leq (\sup_{x\in\RR} Q_0(\epsilon\underline{\zeta}))(\inf_{x\in\RR} f(\epsilon\underline{\zeta}))^{-1},
\end{equation}
where we recall that if~\eqref{CondDepth} is satisfied, then $\underline{h_1}=1-\epsilon\underline{\zeta}$ and $\underline{h_2}=\frac1\delta+\epsilon\underline{\zeta}$ satisfy
\[ \inf_{x\in\RR} \underline{h_1}\geq  h_{01} , \quad  \sup_{x\in\RR} \left|\underline{h_1} \right|\leq 1+1/\delta , \quad \inf_{x\in\RR} \underline{h_2}\geq  h_{01} , \quad  \sup_{x\in\RR} \left|\underline{h_2}\right|\leq 1+1/\delta.
\]
\end{proof}

We conclude this section by proving some general estimates concerning our new operators, which will be useful in the following subsection.
\begin{Lemma}\label{lemS}
 Let $(\mu,\epsilon,\delta,\gamma,\bo)=\p\in\P_{\rm CH}$, and let $U=(\zeta_u,u)^\top$ such that $\zeta_u\in L^\infty$ satisfies~\eqref{CondDepth},\eqref{CondEllipticity}. Then for any
 $V,W\in X^{0}$, one has
\begin{equation}\label{eq:est-S}
\Big\vert \ \Big(\ S[U] V \ , \ W \ \Big) \ \Big\vert \  \leq \ \ C \ \big\vert V\big\vert_{X^0}\big\vert W\big\vert_{X^0},
\end{equation}
with $C=C(M_{\rm CH},h_{01}^{-1},h_{02}^{-1},\epsilon\big\vert \zeta_u\big\vert_{L^\infty})\ $.

Moreover, if $\zeta_u\in H^s,V\in X^{s-1}$ with $s\geq s_0+1,\ s_0>1/2$, then one has
\begin{align}
\label{eq:est-com-S}
\Big\vert  \Big(\ \big[ \Lambda^s,S[U]\big] V \ , \ W \ \Big) \Big\vert  \ &\leq \ C\ \big\vert V\big\vert_{X^{s-1}}\big\vert W\big\vert_{X^0}\\
\label{eq:est-com-S-S}
\Big\vert  \Big(\ \big[ \Lambda^s,S^{-1}[U]\big] V \ , \ S[U] W \ \Big) \Big\vert  \ &\leq \  C\ \big\vert V\big\vert_{H^{s-1}\times H^{s-1}}\big\vert W\big\vert_{X^0}
\end{align}
with $C=C(M_{\rm CH},h_{01}^{-1},h_{02}^{-1},\epsilon \big\vert \zeta_u\big\vert_{H^s})\ $.
\end{Lemma}
\begin{proof} Let $U=(\zeta_u,u)^\top\in X^{s}$, $V=(\zeta_v,v)^\top\in X^{s}$, $W=(\zeta_w,w)^\top\in X^{0}$. Then by definition of $S[\cdot]$ in~\eqref{defS}, one has
\[
\Big(\ S[U] V \ , \ W \ \Big) = \Big(\ \frac{Q_0(\epsilon\zeta_u)}{f(\epsilon\zeta_u)} \zeta_v \ , \ \zeta_w \ \Big) +\Big(\ \mfT[\epsilon\zeta_u] v \ , \ w \ \Big) .
\]
The first term is straightforwardly estimated by Cauchy-Schwarz inequality:
\[  \big\vert\Big(\ \frac{Q_0(\epsilon\zeta_u)}{f(\epsilon\zeta_u)} \zeta_v \ , \ \zeta_w \ \Big)  \big\vert\leq \big\vert \frac{Q_0(\epsilon\zeta_u)}{f(\epsilon\zeta_u)}\big\vert_{L^\infty} \big\vert \zeta_v\big\vert_{L^2}\big\vert \zeta_w\big\vert_{L^2},\]
and  $\frac{Q_0(\epsilon\zeta_u)}{f(\epsilon\zeta_u)}$ is uniformly bounded since $\zeta_u$ satisfies~\eqref{CondDepth}.

The second term is estimated by Lemma~\ref{Lem:mfT},\eqref{continous}.
\[
 \big\vert\Big(\ \mfT[\epsilon\zeta_u] v \ , \ w \ \Big) \big\vert\leq c_0\vert v\vert_{H^1_\mu}\vert w\vert_{H^1_\mu}\leq c_0\vert V\vert_{X^0}\vert W\vert_{X^0}.
\]
Estimate~\eqref{eq:est-S} is proved.
\medskip

Now, let us decompose
\[
\Big(\ \big[ \Lambda^s,S[U]\big] V \ , \ W \ \Big) = \Big(\ \big[ \Lambda^s,\frac{Q_0(\epsilon\zeta_u)}{f(\epsilon\zeta_u)}\big]  \zeta_v \ , \ \zeta_w \ \Big) +\Big(\ \big[ \Lambda^s,\mfT[\epsilon\zeta_u]\big] v \ , \ w \ \Big) .
\]
By Cauchy-Schwarz inequality and Lemma~\ref{K-P}, one has
\begin{align*}
\big\vert \Big(\ \big[ \Lambda^s,\frac{Q_0(\epsilon\zeta_u)}{f(\epsilon\zeta_u)}\big]  \zeta_v \ , \ \zeta_w \ \Big) \big\vert  &\leq \big\vert \big[ \Lambda^s,\frac{Q_0(\epsilon\zeta_u)}{f(\epsilon\zeta_u)}\big]  \zeta_v\big\vert_{L^2}\big\vert\zeta_w\big\vert_{L^2}\\
 &\leq \big\vert \partial_x\big\{\frac{Q_0(\epsilon\zeta_u)}{f(\epsilon\zeta_u)}\big\} \big\vert_{H^{s-1}} \big\vert \zeta_v\big\vert_{H^{s-1}}\big\vert\zeta_w\big\vert_{L^2}\\
 &\leq  C(\epsilon\big\vert \zeta_u\big\vert_{H^{s}} ) \big\vert \zeta_v\big\vert_{H^{s-1}}\big\vert\zeta_w\big\vert_{L^2},
\end{align*}
where we used Lemma~\ref{Moser} and continuous Sobolev embedding for the last inequality. 

The second term is estimated using Lemma~\ref{Lem:mfT},~\eqref{eq:estComT}:
\[\big\vert \Big(\ \big[ \Lambda^s,\mfT[\epsilon\zeta_u]\big] v \ , \ w \ \Big)\big\vert \leq C(\epsilon\vert \zeta_u \vert_{H^s})\vert V\vert_{X^{s-1}}\vert W\vert_{X^{0}}.\]
Estimate~\eqref{eq:est-com-S} is proved.
\medskip

Finally, one has
\[ \Big(\ \big[ \Lambda^s,S^{-1}[U]\big] V \ , \ S[U] W \ \Big)  \leq  \Big(\ \big[ \Lambda^s,\frac{f(\epsilon\zeta_u)}{Q_0(\epsilon\zeta_u)}\big]  \zeta_v  \ ,\frac{Q_0(\epsilon\zeta_u)}{f(\epsilon\zeta_u)}\zeta_w \ \Big) +\Big(\ \big[ \Lambda^s,\mfT^{-1}[\epsilon\zeta_u]\big] v \ , \ \mfT[\epsilon\zeta_u] w \ \Big) .
\]
The first term is estimated exactly as above, noticing that both $f(\epsilon\zeta_u)$ and $Q_0(\epsilon\zeta_u)$ are bounded from above and below since~\eqref{CondDepth} and~\eqref{CondEllipticity} are satisfied.

The second term is estimated using Corollary~\ref{col:comwithT},~\eqref{eq:comwithT}. Estimate~\eqref{eq:est-com-S-S} follows, and the Lemma is proved.
\end{proof}
\subsection{Energy estimates}\label{ssec:energyestimates}
Our aim is to establish {\em a priori} energy estimates concerning our linear system. In order to be able to use the linear analysis to both the well-posedness and stability of the nonlinear system, we consider the following modified system 
\begin{equation}\label{SSlsysm}
	\left\lbrace
	\begin{array}{l}
	\dsp\partial_t U+A_0[\underline{U}]\partial_x U+A_1[\underline{U}]\partial_x U= F;
        \\
	\dsp U_{\vert_{t=0}}=U_0.
	\end{array}\right.
\end{equation}
where we added a right-hand-side $F$, whose properties will be precised in the following Lemmas.
\medskip

We begin by asserting a basic $X^0$ energy estimate, that we extend to $X^s$ space ($s>3/2$) later on.
\begin{Lemma}[$X^0$ energy estimate]\label{Lem:L2}
Set $(\mu,\epsilon,\delta,\gamma,\bo)\in\P_{\rm CH}$. Let $T>0$ and $U\in L^\infty ( [0,T/\epsilon];X^0)$ and $\underline{U},\partial_x \underline{U}\in L^\infty([0,T/\epsilon]\times \RR)$ such that $\partial_t \underline{\zeta} \in L^\infty([0,T/\epsilon]\times \RR)$ and $\underline{\zeta}$ satisfies~\eqref{CondDepth},\eqref{CondEllipticity}, and $U,\underline{U}$ satisfy system~\eqref{SSlsysm}, with a right hand side, $F$, such that
\[
\big(F,S[\underline U] U\big) \ \leq \ C_F\ \epsilon \big\vert U \big\vert_{X^0}^2\ +\ f(t)\ \big\vert U \big\vert_{X^0},\]
with $C_F$ a constant and $f$ a positive integrable function on $[0,T/\epsilon]$.

 Then there exists $\lambda\equiv C(\big\Vert \partial_t \underline \zeta \big\Vert_{L^\infty([0,T/\epsilon]\times \RR)},\big\Vert \underline U \big\Vert_{L^\infty([0,T/\epsilon]\times \RR)},\big\Vert \partial_x\underline U \big\Vert_{L^\infty([0,T/\epsilon]\times \RR)},C_F)$ such that
\begin{equation}\label{energyestimateL2}
	\forall t\in [0,\frac{T}{\epsilon}],\qquad
	E^0(U)(t)\leq e^{\epsilon\lambda t}E^0(U_0)+ \int^{t}_{0} e^{\epsilon\lambda( t-t')}f(t')dt'.
\end{equation}
The constant $\lambda$ is independent of $(\mu,\epsilon,\delta,\gamma,\bo)\in\P_{\rm CH}$, but depends on $M_{\rm CH},h_{01}^{-1},h_{02}^{-1}$.
\end{Lemma}
\begin{proof}
Let us take the inner product of~\eqref{SSlsysm} by $ S[\underline U] U$:
\[ \big(\partial_t U,S[\underline U] U\big) \ + \ \big(A_0[\underline U]\partial_x U,S[\underline U]  U\big)\ + \ \big(A_1[\underline U]\partial_x U,S[\underline U]  U\big) 
\ =  \ \big( F ,S[\underline U] U\big) \ ,
\]
From the symmetry property of $S[\underline U]$, and using the  definition of $E^s(U)$, one deduces
\begin{equation}\label{qtsctrL2}
\frac12 \frac{d}{dt}E^0(U)   \ =  \frac12\big( U,\big[\partial_t,  S[\underline U]\big] U\big)-\big(S[\underline U]A_0[\underline U]\partial_xU,  U\big) -\big(S[\underline U]A_1[\underline U]\partial_x  U, U\big)  +
\big(F,S[\underline U] U\big). 
\end{equation}

Let us first estimate $\big(S[\underline U]A_0[\underline U]\partial_x U,  U\big) $.
 Let us recall that
 \begin{equation*}
S[\underline{U}]A_0[\underline{U}]=\begin{pmatrix}
 \epsilon\frac{Q_0(\epsilon\underline{\zeta}) }{f(\epsilon\underline{\zeta})}f'(\epsilon\underline{\zeta}) \underline{v}& Q_0(\epsilon\underline{\zeta}) \\
Q_0(\epsilon\underline{\zeta})  &\epsilon \mfQ[\epsilon\underline{\zeta},\underline{v}]
\end{pmatrix}
\end{equation*}
so that
\begin{align*}
\big(S[\underline U]A_0[\underline U]\partial_x U, U\big) \ &= -\frac{1}{2}\ \Big( \zeta, \epsilon\partial_x\big(\frac{Q_0(\epsilon\underline{\zeta}) }{f(\epsilon\underline{\zeta})}f'(\epsilon\underline{\zeta}) \underline{v} \big)  \zeta \Big)-\Big( \zeta, \partial_x(Q_0(\epsilon\underline{\zeta}) ) v\Big)\ + \ \epsilon\Big(\mfQ[\epsilon\underline{\zeta},\underline{v}] \partial_x  v,  v \Big) \\
&\equiv A_1+A_2+A_3
\end{align*}

The estimates concerning $A_1$ and $A_2$ are straightforward. Using Cauchy-Schwarz inequality, there exists $C=C(\big\Vert \underline U\big\Vert_{L^\infty}+\big\Vert \partial_x \underline U\big\Vert_{L^\infty})$ such that
\[
| A_1| +| A_2| \ \leq\ \epsilon C (\big\vert 
 \zeta \big\vert_{L^2}^2+\big\vert  v\big\vert_{L^2}^2 ) \ \leq \ \epsilon C \big\vert U\big\vert_{X^0}^2.
\]

As for $A_3$, one has (recalling the definition of $\mfQ$ in~\eqref{defmfQ}):
\[
 \big(\mfQ[\epsilon\underline{\zeta},\underline{v}] \partial_x v, v\big) \ = \ -\big( \partial_x(q_1(\epsilon\underline{\zeta})q_3(\epsilon\underline{\zeta})\underline{v} )  v,v\big) \ - \ \mu\kappa\big( (\partial_x\underline{v})(\partial_x v),\partial_x v \big)
\]
 Those two terms are controlled, again thanks to Cauchy-Schwarz inequality, so that
\[ |  \big(\mfQ[\epsilon\underline{\zeta},\underline{v}] \partial_x v, v\big)| \  \leq  C(\big\Vert \underline U\big\Vert_{L^\infty}+\big\Vert \partial_x \underline U\big\Vert_{L^\infty}) \big\vert v\big\vert_{H^{1}_\mu}^2 \ \leq \  C(\big\Vert \underline U\big\Vert_{L^\infty}+\big\Vert \partial_x \underline U\big\Vert_{L^\infty}) \big\vert U\big\vert_{X^0}^2 .\]

Altogether, we proved
\begin{equation}\label{est:0L2}
|\big(S[\underline U]A_0[\underline U]\partial_x  U,  U\big) |\leq \epsilon  C(\big\Vert \underline U\big\Vert_{L^\infty}+\big\Vert \partial_x \underline U\big\Vert_{L^\infty}) \big\vert U\big\vert_{X^0}^2.
\end{equation}
\bigskip

Let us now estimate $\big(S[\underline U]A_1[\underline U]\partial_x  U,  U\big) $.
One has
\begin{align*}
\big(S[\underline U]A_1[\underline U]\partial_x  U,  U\big) & \ = \  \big(\epsilon^2 Q_1(\epsilon\underline{\zeta},\underline{v})\partial_x  \zeta,  v\big) +\epsilon\varsigma\big(\mfT[\epsilon\underline{\zeta}](\underline{v}
\partial_x v),  v )\\ &\ \equiv \ A_{4}+A_{5}.
 \end{align*}
In order to control the term $A_4$, we write 
 \begin{align*}
 \big(\epsilon^2 Q_1(\epsilon\underline{\zeta},\underline{v})\partial_x   \zeta,   v\big)
 & \ = \ \epsilon^2\big(q_1(\epsilon\underline{\zeta})q'_3(\epsilon\underline{\zeta})\underline{v}^2\partial_x   \zeta,  v\big) \ =\ -\epsilon^2\big(\partial_x(q_1(\epsilon\underline{\zeta})q'_3(\epsilon\underline{\zeta})\underline{v}^2   v),  \zeta\big)\\
& \ =\ -\epsilon^2\big(\partial_x(q_1(\epsilon\underline{\zeta})q'_3(\epsilon\underline{\zeta})\underline{v}^2)   v,  \zeta\big)
-\epsilon^2\big(q_1(\epsilon\underline{\zeta})q'_3(\epsilon\underline{\zeta})\underline{v}^2   \partial_x v,  \zeta\big).
\end{align*}
Since $\p\in\P_{\rm CH}$, as defined in~\eqref{eqn:defRegimeCHmr}, one has $\epsilon\leq M \sqrt{\mu}$, and therefore 
\[
\vert A_4\vert \leq \epsilon C(\big\Vert \underline U\big\Vert_{L^\infty}+\big\Vert \partial_x \underline U\big\Vert_{L^\infty}) \big\vert U\big\vert_{X^0}^2.
\]
(where we used, once again, Cauchy-Schwarz inequality.)

 In order to control $A_5$ one makes use of the identity given in~\eqref{eqn:symmetry-like}, applied to $V=  v$, and deduce easily
 \[
\vert A_5\vert \leq \epsilon C(\big\Vert \underline U\big\Vert_{L^\infty}+\big\Vert \partial_x \underline U\big\Vert_{L^\infty}) E^0(U)^2.
\]
Altogether, one has
\begin{equation}\label{est:1L2}
|\big(S[\underline U]A_1[\underline U]\partial_x   U,   U\big) |\leq \epsilon C(\big\Vert \underline U\big\Vert_{L^\infty}+\big\Vert \partial_x \underline U\big\Vert_{L^\infty}) \big\vert U\big\vert_{X^0}^2.
\end{equation}
\bigskip

The last term to estimate is $\big(  U,\big[\partial_t,  S[\underline U]\big]  U\big)$.

One has
\begin{align*}
\big(  U,\big[\partial_t,  S[\underline U]\big]  U\big)&\equiv\dsp( v,\Big[\partial_t,\underline{\mfT}\Big] v)
+( \zeta,\Big[\partial_t,\frac{Q_0(\epsilon\underline{\zeta})}{f(\epsilon\underline{\zeta})}\Big] \zeta)\\
&=\Big(  v,\big(\partial_t q_1(\epsilon\underline{\zeta}) \big)  v\Big)+\mu\Big(  v,\partial_x\big((\partial_tq_2(\epsilon\underline{\zeta}))( \partial_x  v)\big) \Big)
+\Big( \zeta,\partial_t \Big(\frac{Q_0(\epsilon\underline{\zeta})}{f(\epsilon\underline{\zeta})}\Big) \zeta\Big)\\
&=\epsilon\kappa_1\Big(  v,(\partial_t \underline{\zeta})   v\Big)-\mu\epsilon\kappa_2\Big(  \partial_x v,(\partial_t\underline{\zeta})  \partial_x v) \Big)\\
&\qquad \qquad+\epsilon\Big( \zeta,\frac{Q_0'(\epsilon\underline{\zeta})f(\epsilon\underline{\zeta})-Q_0(\epsilon\underline{\zeta})f'(\epsilon\underline{\zeta})}{f(\epsilon\underline{\zeta})^2}(\partial_t\underline{\zeta}) \zeta\Big)
\end{align*}

From Cauchy-Schwarz inequality and since $\underline{\zeta}$ satisfies~\eqref{CondDepth}, one deduces
\begin{equation}\label{est:2L2}
\big\vert \ \frac12\big(  U,\big[\partial_t,  S[\underline U]\big]  U\big)\ \big\vert\ \leq\  \epsilon C(\big\Vert \partial_t \underline \zeta \big\Vert_{L^\infty([0,T/\epsilon]\times \RR)},\big\Vert \underline \zeta \big\Vert_{L^\infty([0,T/\epsilon]\times \RR)}) \big\vert U\big\vert_{X^0}^2 .
\end{equation}

One can now conclude with the proof of the energy estimate. Plugging~\eqref{est:0L2},~\eqref{est:1L2} and~\eqref{est:2L2} into~\eqref{qtsctrL2}, invoking Lemma~\ref{lemmaes} and making use of the assumption of the Lemma on $F$, yields
\[ \frac12 \frac{d}{dt}E^0(U)^2 \ \leq \ \epsilon \ C_0 E^0(U)^2+ f(t) E^0(U),\]
where $C_0\equiv  C(\big\Vert \partial_t \underline \zeta \big\Vert_{L^\infty([0,T/\epsilon]\times \RR)},\big\Vert \underline U \big\Vert_{L^\infty([0,T/\epsilon]\times \RR)},\big\Vert \partial_x\underline U \big\Vert_{L^\infty([0,T/\epsilon]\times \RR)},C_F)$. Consequently
\[ \frac{d}{dt}E^0(U) \ \leq \ C_0\epsilon E^0(U)+ f(t).\]
Making use of the usual trick, we compute for any $\lambda \in \RR$,
\[
e^{\epsilon\lambda t}\partial_t(e^{-\epsilon\lambda t}E^0(U))=-\epsilon\lambda E^0(U) +\frac{d}{dt} E^0(U).
\]
Thanks to the above inequality, one can choose  $\lambda= C_0$, so that for all $t\in [0,\frac{T}{\epsilon}]$, one deduces
\[
\frac{d}{dt} (e^{-\epsilon\lambda t}E^0(U)) \ \leq \  f(t)e^{-\epsilon\lambda t} .\]
Integrating this differential inequality yields
\begin{equation}\label{energyestimateL2inproof}
	\forall t\in [0,\frac{T}{\epsilon}],\qquad
	E^0(U)(t)\leq e^{\epsilon\lambda t}E^0(U_0)+ \int^{t}_{0} e^{\epsilon\lambda( t-t')}f(t')dt'.
\end{equation}
This proves the energy estimate~\eqref{energyestimateL2}.
\end{proof}
\medskip

Let us now turn to the {\em a priori} energy estimate in ``large'' $X^s$ norm.
\begin{Lemma}[$X^s$ energy estimate]\label{Lem:Hs}
Set $(\mu,\epsilon,\delta,\gamma,\bo)\in \P_{\rm CH}$, and  $s\geq s_0+1$, $s_0>1/2$. Let $U=(\zeta,v)^\top$ and $\underline{U}=(\underline{\zeta},\underline{v})^\top$ be such that $U,\underline{U}\in L^\infty([0,T/\epsilon];X^{s}) $, $\partial_t\underline{\zeta}\in L^\infty([0,T/\epsilon]\times\RR)$ and $\underline{\zeta}$ satisfies~\eqref{CondDepth},\eqref{CondEllipticity} uniformly on $[0,T/\epsilon]$, and such that system~\eqref{SSlsysm} holds with a right hand side, $F$, with
\[
\big( \Lambda^s F,S[\underline U]\Lambda^s U\big)\ \leq  \ C_F\ \epsilon \big\vert  U\big\vert_{X^s}^2\ +\ f(t)\ \big\vert U\big\vert_{X^s} ,
\]
where $C_F$ is a constant and $f$ is an integrable function on $[0,T/\epsilon]$.

Then there exists $\lambda=C(\big\Vert \underline U\big\Vert_{X^s_T}+C_F)$ such that the following energy estimate holds:
\begin{align}\label{energyestimate}
E^s(U)(t)&\leq e^{\epsilon\lambda t}E^s(U_0)+
\int^{t}_{0} e^{\epsilon\lambda( t-t')}f(t')dt'.\end{align}
The constant $\lambda$ is independent of $(\mu,\epsilon,\delta,\gamma,\bo)\in\P_{\rm CH}$, but depends on $M_{\rm CH},h_{01}^{-1},h_{02}^{-1}$.
\end{Lemma}
\begin{Remark}
In this Lemma, and in the proof below, the norm $\big\Vert \underline U\big\Vert_{X^s_T}$ is to be understood as essential sup:
\[
 \Vert U\Vert_{X^s_T}\equiv \esssup_{t\in [0,T/\epsilon]}\vert U(t,\cdot)\vert_{X^s}+\esssup_{t\in [0,T/\epsilon],x\in\RR}\vert \partial_t \zeta(t,x)\vert.
\]
\end{Remark}
\begin{proof}
Let us multiply the system~\eqref{SSlsysm} on the right by $\Lambda^s S[\underline U]\Lambda^s U$, and integrate by parts. One obtains
\begin{multline} \big(\Lambda^s\partial_t U,S[\underline U]\Lambda^s U\big) \ + \ \big(\Lambda^sA_0[\underline U]\partial_x U,S[\underline U]\Lambda^s U\big)\ + \ \big(\Lambda^sA_1[\underline U]\partial_x U,S[\underline U]\Lambda^s U\big)  \\
\ =  \ \epsilon\big(\Lambda^s F,S[\underline U]\Lambda^s U\big)\ ,
\end{multline}
from which we deduce, using the symmetry property of $S[\underline U]$, as well as the definition of $E^s(U)$:
\begin{multline}\label{qtsctr}
\frac12 \frac{d}{dt}E^s(U)   \ =  \frac12\big(\Lambda^s U,\big[\partial_t,  S[\underline U]\big]\Lambda^s U\big)-\big(S[\underline U]A_0[\underline U]\partial_x \Lambda^s U, \Lambda^s U\big) -\big(S[\underline U]A_1[\underline U]\partial_x \Lambda^s U, \Lambda^s U\big)  \\
-\big(\big[\Lambda^s ,A_0[\underline U]+A_1[\underline U]\big]\partial_x U,S[\underline U]\Lambda^s U\big)
+\epsilon\big( \Lambda^s F,S[\underline U]\Lambda^s U\big). 
\end{multline}
We now estimate each of the different components of the r.h.s of the above identity.

\bigskip

\noindent $\bullet$ {\em Estimate of $\big(S[\underline U]A_0[\underline U]\partial_x \Lambda^s U, \Lambda^s U\big) $ and $\big(S[\underline U]A_1[\underline U]\partial_x \Lambda^s U, \Lambda^s U\big) $.}
One can use the $L^2$ estimate derived in~\eqref{est:0L2}, applied to $\Lambda^s U$. One deduces 
\begin{equation}\label{est:0}
|\big(S[\underline U]A_0[\underline U]\partial_x \Lambda^s  U, \Lambda^s  U\big) |\leq \epsilon  C(\big\Vert \underline U\big\Vert_{L^\infty}+\big\Vert \partial_x \underline U\big\Vert_{L^\infty}) \big\vert U\big\vert_{X^s}^2.
\end{equation}
  Now, thanks to Sobolev embedding, one has for $s>s_0+1$, $s_0>1/2$
\[C(\big\Vert \underline U\big\Vert_{L^\infty}+\big\Vert \partial_x \underline U\big\Vert_{L^\infty})\ \leq \ C(\big\Vert \underline U\big\Vert_{X^s_T}), \]
so that
\begin{equation}\label{est:1}
|\big(S[\underline U]A_0[\underline U]\partial_x \Lambda^s U, \Lambda^s U\big) |\leq \epsilon C(\big\Vert \underline U\big\Vert_{X^s_T}) \big\vert  U\big\vert_{X^s}^2.
\end{equation}
\bigskip

Similarly, using~\eqref{est:1L2}, applied to $\Lambda^s U$ and continuous Sobolev embedding, one has
\begin{equation}\label{est:1'}
|\big(S[\underline U]A_1[\underline U]\partial_x \Lambda^s U, \Lambda^s U\big) |\leq \epsilon C(\big\Vert \underline U\big\Vert_{X^s_T}) \big\vert U\big\vert_{X^s}^2.
\end{equation}
\bigskip

\noindent $\bullet$ {\em Estimate of $ \big(\big[\Lambda^s,A[\underline{U}]\big]\partial_x U,S[\underline{U}]\big]\Lambda^s U\big)$}, where $A[\underline{U}]=A_0[\underline{U}]+A_1[\underline{U}]$. Using the definition of $A[\cdot]$ and $S[\cdot]$ in~\eqref{defA0A1},\eqref{defS}, one has 
 \begin{align*}
\big(\big[\Lambda^s,A[\underline{U}]\big]\partial_x U,S[\underline{U}]\Lambda^s U\big)&= \Big([\Lambda^s,\epsilon f'(\epsilon\underline{\zeta}) \underline{v}]\partial_x\zeta +[\Lambda^s, f(\epsilon\underline{\zeta})]\partial_x v\ ,\ \frac{Q(\epsilon\underline{\zeta},\underline{v})}{f(\epsilon\underline{\zeta})}\Lambda^s\zeta\Big) \\
&\quad +\Big([\Lambda^s, \underline{\mfT}^{-1}\big(Q(\epsilon\underline{\zeta},\underline{v})\partial_x \zeta\big)\ ,\ \underline{\mfT}\Lambda^s v\Big)\ +\ \epsilon\Big([\Lambda^s, \underline{\mfT}^{-1}\mfQ[\epsilon\underline{\zeta},\underline{v}] +\varsigma \underline v]\partial_x v, \underline{\mfT}\Lambda^s v\Big)\\
&\equiv B_1+B_2+B_3.
\end{align*}
Here and in the following, we denote  $\underline{\mfT}\equiv \mfT[\epsilon\underline{\zeta}]$ and $Q(\epsilon\underline{\zeta},\underline{v})=Q_0(\epsilon\underline{\zeta})+\epsilon^2 Q_1(\epsilon\underline{\zeta},\underline{v}).$
Let us treat each of these terms separately.

$-$ Control of $B_1=\Big([\Lambda^s,\epsilon f'(\epsilon\underline{\zeta}) \underline{v}]\partial_x\zeta +[\Lambda^s, f(\epsilon\underline{\zeta})]\partial_x v\ ,\ \frac{Q(\epsilon\underline{\zeta},\underline{v})}{f(\epsilon\underline{\zeta})}\Lambda^s\zeta\Big)$.

From Cauchy-Schwarz inequality, one has
\[ |B_1| \ \leq \ \Big\vert[\Lambda^s,\epsilon f'(\epsilon\underline{\zeta}) \underline{v}]\partial_x\zeta +[\Lambda^s, f(\epsilon\underline{\zeta})]\partial_x v\Big\vert_{L^2}\Big\vert \frac{Q(\epsilon\underline{\zeta},\underline{v})}{f(\underline{\zeta})}\Lambda^s\zeta\Big\vert_{L^2}.\]
Since $s\geq s_0+1$, we can use the commutator estimate Lemma~\ref{K-P} to get
\begin{align*}
\Big\vert [\Lambda^s,\epsilon f'(\epsilon\underline{\zeta}) \underline{v}]\partial_x\zeta +[\Lambda^s, f(\epsilon\underline{\zeta})]\partial_x v\Big\vert_{L^2}
&\lesssim \left(\vert \partial_x(\epsilon f'(\epsilon\underline{\zeta}))\vert_{H^{s-1}}+\vert \partial_x(f(\epsilon\underline{\zeta}))\vert_{H^{s-1}}\right) \vert \partial_x U\vert_{H^{s-1}} \\
 &\lesssim \epsilon C(\big\Vert \underline U\big\Vert_{X^s_T})\big\vert U\big\vert_{X^s}.
\end{align*}
where we used, for the last inequality,
\[ \partial_x(f(\epsilon\underline{\zeta})) \ = \ \epsilon(\partial_x\underline{\zeta}) f'(\epsilon\underline{\zeta}),\]
It follows, using that $\frac{Q(\epsilon\underline{\zeta},\underline{v})}{f(\underline{\zeta})}\in L^\infty$ since $\underline{\zeta}$ satisfies~\eqref{CondDepth}:
\[
|B_1| \  \leq\  \epsilon C(\big\Vert \underline U\big\Vert_{X^s_T})\big\vert U\big\vert_{X^s}^2.
\]
\medskip

$-$ Control of $B_2=\Big([\Lambda^s, \underline{\mfT}^{-1}(Q(\epsilon\underline{\zeta},\underline{v})\cdot)]\partial_x \zeta\ ,\ \underline{\mfT}\Lambda^s v\Big)$.

By symmetry of $\underline{\mfT}$, one has
\[ B_2=\Big(\underline{\mfT}[\Lambda^s, \underline{\mfT}^{-1}(Q(\epsilon\underline{\zeta},\underline{v})\cdot)]\partial_x \zeta\ ,\ \Lambda^s v\Big).\]
Now, one can check that, by definition of the commutator,
\begin{align*}
\underline{\mfT}[\Lambda^s, \underline{\mfT}^{-1}(Q(\epsilon\underline{\zeta},\underline{v})\cdot)]\partial_x \zeta &= \underline{\mfT}\Lambda^s \underline{\mfT}^{-1} Q(\epsilon\underline{\zeta},\underline{v})\partial_x\zeta - Q(\epsilon\underline{\zeta},\underline{v})\Lambda^s\partial_x\zeta \\
&= \underline{\mfT}\Lambda^s \underline{\mfT}^{-1} Q(\epsilon\underline{\zeta},\underline{v})\partial_x\zeta -\Lambda^s \underline{\mfT} \underline{\mfT}^{-1}(Q(\epsilon\underline{\zeta},\underline{v})\partial_x\zeta )+\Lambda^s (Q(\epsilon\underline{\zeta},\underline{v})\partial_x\zeta ) - Q(\epsilon\underline{\zeta},\underline{v})\Lambda^s\partial_x\zeta\\
&=  -\big[\Lambda^s, \underline{\mfT}\big] \underline{\mfT}^{-1}(Q(\epsilon\underline{\zeta},\underline{v})\partial_x\zeta )+ \big[\Lambda^s,Q(\epsilon\underline{\zeta},\underline{v}) \big]\partial_x\zeta
\end{align*}
We can now use Corollary~\ref{col:comwithT}, and deduce
\[ \left|\big([\Lambda^s, \underline{\mfT}] \underline{\mfT}^{-1}(Q(\epsilon\underline{\zeta},\underline{v})\partial_x\zeta ) \ , \ \Lambda^s v\Big)\right| \ \leq \ C(\epsilon\big\vert \underline{\zeta} \big\vert_{H^s})\big\vert Q(\epsilon\underline{\zeta},\underline{v})\partial_x\zeta \big\vert_{H^{s-1}}\big\vert v\big\vert_{H^{s+1}_\mu}\leq C(\big\vert \underline{U} \big\vert_{X^s})\big\vert \zeta \big\vert_{H^s}\big\vert v\big\vert_{H^{s+1}_\mu}.\]
The last inequality is obtained using Lemma~\ref{Moser}.

From Lemma~\ref{K-P} and the explicit definition of $Q=Q_0+\epsilon^2 Q_1$ in~\eqref{defQ0Q1}, one has
\[ \big\vert\ \big[\Lambda^s,Q(\epsilon\underline{\zeta},\underline{v}) \big]\partial_x\zeta\ \big\vert_{L^2}=\epsilon \big\vert [\Lambda^s, (\gamma+\delta)\kappa_1\underline{\zeta}+\epsilon q_1(\epsilon\underline{\zeta} )q_3'(\epsilon\underline{\zeta} ) \underline{v}^2 ]\partial_x\zeta \big\vert_{L^2}  \ \leq \ \epsilon C(\big\vert \underline{U} \big\vert_{H^s})\big\vert \partial_x\zeta \big\vert_{H^{s-1}},\]
so that  we finally get
\[
|B_2| \  \leq\  \epsilon C(\big\Vert U\big\Vert_{X^s_T})\big\vert U\big\vert_{X^s}^2.
\]

$-$ Control of $B_3=\epsilon\Big(\big[\Lambda^s, \underline{\mfT}^{-1}\mfQ[\epsilon\underline{\zeta},\underline{v}] +\varsigma \underline v\big]\partial_x v, \underline{\mfT}\Lambda^s v\Big)$.

Let us first use the definition of $\mfQ[\epsilon\underline{\zeta},\underline{v}]$~\eqref{defQ0Q1} to expand:
\begin{align*}
B_3&=\epsilon\Big([\Lambda^s, \underline{\mfT}^{-1}(q_1(\epsilon\underline{\zeta})q_3(\epsilon\underline{\zeta})\underline{v}\cdot)] \partial_x v, \underline{\mfT}\Lambda^s v\Big) \ + \ \mu\epsilon\kappa \Big([\Lambda^s, \underline{\mfT}^{-1}\partial_x ((\partial_x\underline{v})\ \cdot)] \partial_x v, \underline{\mfT}\Lambda^s v\Big)\\
&\qquad + \ \epsilon\varsigma \Big([\Lambda^s, \underline{v}] \partial_x v, \underline{\mfT}\Lambda^s v\Big) \\
&\equiv B_{31}+B_{32}+B_{33}.
\end{align*}
In order to estimate $B_{31}$ and $B_{32}$, one proceeds as for the control of $B_2$. One can check

\[\underline{\mfT} [\Lambda^s, \underline{\mfT}^{-1}(q_1(\epsilon\underline{\zeta})q_3(\epsilon\underline{\zeta})\underline{v}\cdot)]\partial_x v = -[\Lambda^s, \underline{\mfT}] \underline{\mfT}^{-1}(q_1(\epsilon\underline{\zeta})q_3(\epsilon\underline{\zeta})\underline{v}\partial_x v )+ [\Lambda^s,q_1(\epsilon\underline{\zeta})q_3(\epsilon\underline{\zeta})\underline{v} ]\partial_x v.
\]
As above, using Cauchy-Schwarz inequality, Corollary~\ref{col:comwithT} and Lemma~\ref{K-P}, one obtains
\[
|B_{31}| \  \leq\  \epsilon  C(\big\Vert \underline{U}\big\Vert_{X^s_T})\big\vert U\big\vert_{X^s}^2.
\]

In the same way,
\[\underline{\mfT} [\Lambda^s, \underline{\mfT}^{-1}(\partial_x ((\partial_x\underline{v})\ \cdot)]\partial_x v = -[\Lambda^s, \underline{\mfT}] \underline{\mfT}^{-1}(\partial_x ((\partial_x\underline{v})(\partial_x v))] +\partial_x \big( [\Lambda^s,\partial_x\underline{v} ]\partial_x v\big).
\]
Again, using Cauchy-Schwarz inequality, Corollary~\ref{col:comwithT} and Lemma~\ref{K-P}, one has
\begin{align*}
\mu\Big\vert\big([\Lambda^s, \underline{\mfT}] \underline{\mfT}^{-1}(\partial_x ((\partial_x\underline{v})(\partial_x v))] ,\Lambda^s v\big)\Big\vert \ &\leq \ c_s \mu\big|\underline{\zeta}\big|_{H^{s}}\big| \partial_x ((\partial_x\underline{v})(\partial_x v))\big|_{H^{s-1}}\big|v\big|_{H^{s+1}_\mu}\\
&\leq\ c_s\big|\underline{\zeta}\big|_{H^{s}}\big| \underline{v}\big|_{H^{s+1}_\mu}\big|v\big|_{H^{s+1}_\mu}^2,
\end{align*}
and
\[\mu\Big|\big( \partial_x \big( [\Lambda^s,\partial_x\underline{v} ]\partial_x v\big),\Lambda^s v\big)\Big| \ = \ \mu\Big|\big(  [\Lambda^s,\partial_x\underline{v} ]\partial_x v,\Lambda^s \partial_x  v\big)\Big| \leq c_s\big| \underline{v}\big|_{H^{s+1}_\mu}\big|v\big|_{H^{s+1}_\mu}^2.\]
 Thus we proved
\[
|B_{32}| \  \leq\  \epsilon  C(\big\Vert \underline{U}\big\Vert_{X^s_T})\big\vert U\big\vert_{X^s}^2.
\]

Finally, we turn to $B_{33}=\epsilon\varsigma \Big([\Lambda^s, \underline{v}] \partial_x v, \underline{\mfT}\Lambda^s v\Big)$. From Lemma~\ref{Lem:mfT}, one has
\[ |B_{33}|\leq \big\vert [\Lambda^s, \underline{v}] \partial_x v \big\vert_{H^1_\mu}\big\vert \Lambda^s v\big\vert_{H^1_\mu}\leq \big\vert [\Lambda^s, \underline{v}] \partial_x v \big\vert_{L^2}\big\vert \Lambda^s v\big\vert_{H^1_\mu}+\sqrt\mu\big\vert \partial_x\big( [\Lambda^s, \underline{v}] \partial_x v \big) \big\vert_{L^2}\big\vert \Lambda^s v\big\vert_{H^1_\mu}.\]
Note the identity
\[ \partial_x \big( [\Lambda^s,\underline{v} ]\partial_x v\big) \ = \ \big( [\Lambda^s,\partial_x\underline{v} ]\partial_x v\big)+ \big( [\Lambda^s,\underline{v} ]\partial_x^2 v\big),\]
so that Lemma~\ref{K-P} yields $\sqrt\mu\big\vert \partial_x\big( [\Lambda^s, \underline{v}] \partial_x v \big) \big\vert_{L^2}\leq \big\vert \underline{v}\big\vert_{H^{s+1}_\mu}\big\vert v \big\vert_{H^{s+1}_\mu}$.

Altogether, we proved
\begin{equation}\label{est:2}
|\big(\big[\Lambda^s,A[\underline{U}]\big]\partial_x U,S[\underline{U}]\big]\Lambda^s U\big) |\leq \epsilon C(\big\Vert \underline U\big\Vert_{X^s_T})\big\vert U\big\vert_{X^s}^2.
\end{equation}
\bigskip

\noindent $\bullet$ {\em Estimate of $\frac12\big(\Lambda^s U,\big[\partial_t,  S[\underline U]\big]\Lambda^s U\big)$.}

One has
\begin{align*}
\big(\Lambda^s U,\big[\partial_t,  S[\underline U]\big]\Lambda^s U\big)&\equiv\dsp(\Lambda^sv,\Big[\partial_t,\underline{\mfT}\Big]\Lambda^sv)
+(\Lambda^s\zeta,\Big[\partial_t,\frac{Q_0(\epsilon\underline{\zeta})}{f(\epsilon\underline{\zeta})}\Big]\Lambda^s\zeta)\\
&=\Big(\Lambda^s v,\big(\partial_t q_1(\epsilon\underline{\zeta}) \big)\Lambda^s v\Big)+\mu\Big(\Lambda^s v,\partial_x\big((\partial_tq_2(\epsilon\underline{\zeta}))( \partial_x\Lambda^s v)\big) \Big)\\
&\qquad \qquad
+\Big(\Lambda^s\zeta,\partial_t \Big(\frac{Q_0(\epsilon\underline{\zeta})}{f(\epsilon\underline{\zeta})}\Big)\Lambda^s\zeta\Big)\\
&=\epsilon\kappa_1\Big(\Lambda^s v,(\partial_t \underline{\zeta}) \Lambda^s v\Big)-\mu\epsilon\kappa_2\Big(\Lambda^s \partial_x v,(\partial_t\underline{\zeta})\Lambda^s \partial_x v) \Big)\\
&\qquad \qquad+\epsilon\Big(\Lambda^s\zeta,\frac{Q_0'(\epsilon\underline{\zeta})f(\epsilon\underline{\zeta})-Q_0(\epsilon\underline{\zeta})f'(\epsilon\underline{\zeta})}{f(\epsilon\underline{\zeta})^2}(\partial_t\underline{\zeta})\Lambda^s\zeta\Big)
\end{align*}

From Cauchy-Schwarz inequality and since $\underline{\zeta}$ satisfies~\eqref{CondDepth}, one deduces
\[\big\vert \ \frac12\big(\Lambda^s U,\big[\partial_t,  S[\underline U]\big]\Lambda^s U\big)\ \big\vert\ \leq\  \epsilon C(\big\Vert \partial_t \underline \zeta \big\Vert_{L^\infty([0,T/\epsilon]\times \RR)},\big\Vert \underline \zeta \big\Vert_{L^\infty([0,T/\epsilon]\times \RR)}) \big\vert U\big\vert_{X^s}^2,\]
and continuous Sobolev embedding yields
\begin{equation}\label{est:3}
\big\vert \ \frac12\big(\Lambda^s U,\big[\partial_t,  S[\underline U]\big]\Lambda^s U\big)\ \big\vert\ \leq\  \epsilon C(\big\Vert \underline U\big\Vert_{X^s_T}) \big\vert U\big\vert_{X^s}^2.
\end{equation}

\bigskip

Finally, let us recall the assumption of the Lemma:
\begin{align}\label{est:4}
\big( \Lambda^s F,S[\underline U]\Lambda^s U\big)\leq  \epsilon\ C_F\big\vert U\big\vert_{X^s}^2+ f(t)\big\vert U\big\vert_{X^s} .
\end{align}
\bigskip

We now plug~\eqref{est:1},~\eqref{est:1'},~\eqref{est:2},~\eqref{est:3},  and~\eqref{est:4} into~\eqref{qtsctr}. Using Lemma~\ref{lemmaes}, it follows
\[ \frac12 \frac{d}{dt}E^s(U)^2 \ \leq \ C_0\epsilon E^s(U)^2+ E^s(U)f(t),\]
with $C_0=C(\big\Vert \underline U\big\Vert_{X^s_T}+C_F)$, and consequently
\[ \frac{d}{dt}E^s(U) \ \leq \ \epsilon C_0 E^s(U)+ f(t)\ .\]
Now, for any $\lambda \in \RR$, one has
\[
e^{\epsilon\lambda t}\partial_t(e^{-\epsilon\lambda t}E^s(U))=-\epsilon\lambda E^s(U) +\frac{d}{dt} E^s(U).
\]
Thus with $\lambda=C_0$, one has for all $t\in [0,\frac{T}{\epsilon}]$, 
\[
\frac{d}{dt} (e^{-\epsilon\lambda t}E^s(U)) \ \leq \  f(t)e^{-\epsilon\lambda t} .\]
Integrating this differential inequality yields
\[
	\forall t\in [0,\frac{T}{\epsilon}],\qquad
	E^s(U)(t)\leq e^{\epsilon\lambda t}E^s(U_0)+ C_0\int^{t}_{0} e^{\epsilon\lambda( t-t')}f(t')dt'.
\]
This concludes the proof of Lemma~\ref{Lem:Hs}.
\end{proof}

\subsection{Well-posedness of the linear system}\label{ssec:linearWP}
Let us now state the main result of this section.
\begin{Proposition}\label{ESprop}
Set $\p=(\mu,\epsilon,\delta,\gamma,\bo)\in\P_{\rm CH}$ (see~\eqref{eqn:defRegimeCHmr}) and $s\geq s_0+1$ with $s_0>1/2$, and let $\underline{U}^\p=(\underline{\zeta}^\p, \underline{v}^\p)^\top\in X^{s}_T$ (see Definition~\ref{defispace}) be such that~\eqref{CondDepth},\eqref{CondEllipticity} are satisfied for $t\in [0,T/\epsilon]$, uniformly with respect to $\p\in\P_{\rm CH}$. For any $U_0\in X^{s}$, there exists 
a unique solution to~\eqref{SSlsys}, $U^\p\in C^0([0,T/\epsilon]; X^{s})\cap C^1([0,T/\epsilon]; X^{s-1})\subset X^s_T$, and
$C_0,\lambda_T= C(\big\Vert \underline U\big\Vert_{ X^{s}_T},T,M_{\rm CH},h_{01}^{-1},h_{02}^{-1})$, independent of $\p\in\P_{\rm CH}$, such that
 one has the energy estimates 
\[
\forall\  0\leq t\leq\frac{T}{\epsilon}, \quad E^s(U^\p)(t)\leq e^{\epsilon\lambda_{T} t}E^s(U_0)\quad \text{ and } \quad
E^{s-1}( \partial_t U^\p) \leq C_0 e^{\epsilon\lambda_{T} t}E^s(U_0).
\]
\end{Proposition}
\begin{proof}
Existence and uniqueness of a solution to the initial value problem~\eqref{SSlsys} follows, by standard techniques, from the {\em a priori} estimate~\eqref{energyestimate} in Lemma~\ref{Lem:Hs}: 
\begin{align}\label{energyestimateinproof}
E^s(U)(t)\leq e^{\epsilon\lambda t}E^s(U_0),\end{align}
(since $F\equiv0$, and omitting the index $\p$ for the sake of simplicity.)  We briefly recall the argument below, and refer to~\cite{TaylorII,M'etivier08}, for example, for more details.

\medskip

First, let us notice that using the system of equation~\eqref{SSlsys}, one can deduce an energy estimate on the time-derivative of the solution. Indeed, one has
\begin{align}
\big\vert\partial_t U\big\vert_{X^{s-1}}&= \big\vert -A_0[\underline{U}]\partial_x U-A_1[\underline{U}]\partial_x U\big\vert_{X^{s-1}}\nn \\
&\leq \big\vert \epsilon f'(\epsilon\underline{\zeta})\underline{v}\partial_x v+f(\epsilon\underline{\zeta})\underline{v}\partial_x v\big\vert_{H^{s-1}}\nn \\
&\quad + \big\vert  \mathfrak{T}[\epsilon\underline{\zeta}]^{-1}\big(Q_0(\epsilon\underline{\zeta})\partial_x\zeta+\epsilon \mathfrak{Q}[\epsilon\underline{\zeta},\underline{v}]\partial_x v+\epsilon^2Q_1(\epsilon\underline{\zeta},\underline{v})\partial_x\zeta\big)+\epsilon\varsigma\underline{v}\partial_x v\big\vert_{H^s_\mu} \nn \\
&\leq C(\big\vert \underline U\big\vert_{X^s})\big\vert U\big\vert_{X^s} \leq C_0 e^{\epsilon\lambda_{T} t}E^s(U_0),
\label{energyestimatederivativeinproof}\end{align}
where we use Lemmata~\ref{Moser} and~\ref{proprim}.

The completion of the proof is as follows. In order to construct a solution to~\eqref{SSlsys}, we use a sequence of Friedrichs mollifiers, defined by $J_\nu\equiv (1-\nu\partial_x^2)^{-1/2}$ ($\nu>0$), in order to reduce our system to ordinary differential equation systems on $X^s$, which are solved uniquely by Cauchy-Lipschitz theorem. Estimates~\eqref{energyestimateinproof},\eqref{energyestimatederivativeinproof} hold for each $U_\nu\in C^0([0,T/\epsilon];X^{s})$, uniformly in $\nu>0$. One deduces that a subsequence converges towards $U\in L^2([0,T/\epsilon];X^s)$, a (weak) solution of the Cauchy problem~\eqref{SSlsys}. By regularizing the initial data as well, one can show that the system induces a smoothing effect in time, and that the solution $U\in C^0([0,T/\epsilon];X^{s})\cap C^1([0,T/\epsilon];X^{s-1})$ is actually a strong solution. The uniqueness is a straightforward consequence of~\eqref{energyestimateinproof} (with $U_0\equiv 0$) applied to the difference of two solutions. 
\end{proof}

\section{Proof of existence, stability and convergence}\label{existenceandconvergence}

In this section we prove the main results of this paper.  We start by proving an {\em a priori} estimate on the difference of two possible solutions. Existence and uniqueness of the solution of the Cauchy problem for our new Green-Naghdi system, in the Camassa-Holm regime $\epsilon=O(\sqrt{\mu})$ and over large times, is then deduced from the linear analysis of the previous section and this {\em a priori} estimate. The estimate also provides
\begin{itemize}
\item the stability of the solution with respect to the initial data, thus the Cauchy problem for our system is well-posed in the sense of Hadamard, in Sobolev spaces; see subsection~\ref{ssec:WP}.
\item the stability of the solution with respect to a perturbation of the equation, which allows, together with the well-posedness, to fully justify our system (and any other well-posed, consistent model); see section~\ref{sec:fulljustification}.
\end{itemize}

\subsection{One more a priori estimate}

In this subsection, we control the difference of two solutions of the nonlinear system, with different initial data and right-hand sides. More precisely, we prove the following result.
\begin{Proposition} \label{prop:stability}
 Let $(\mu,\epsilon,\delta,\gamma,\bo)\in\P_{\rm CH}$ and $s\geq s_0+1$, $s_0>1/2$, and assume that there exists $U_i$ for $i\in \{1,2\}$, such that
 $U_{i}=(\zeta_{i},v_{i})^\top \in  X^{s}_{T}$, $U_2\in L^\infty([0,T/\epsilon);X^{s+1})$, $\zeta_1$ satisfies~\eqref{CondDepth},\eqref{CondEllipticity} on $[0,T/\epsilon]$, with $h_{01},h_{02}>0$, 
 and $U_i$ satisfy
\begin{align*}
		\partial_t U_1\ +\ A_0[U_1]\partial_x U_1\ +\ A_1[U_1]\partial_x U_1\ =\ F_1 \ ,  \\
		\partial_t U_2\ +\ A_0[U_2]\partial_x U_2\ +\ A_1[U_2]\partial_x U_2\ =\ F_2 \ ,
\end{align*}		
with $F_i\in L^1([0,T/\epsilon];X^{s})$. Then there exists constants $C_0=C(M_{\rm CH},h_{01}^{-1},h_{02}^{-1},\epsilon\big\vert U_1\big\vert_{X^s},\epsilon\big\vert U_2\big\vert_{X^s})$ and $\lambda_T=C_0\times C( \vert U_2\vert_{L^\infty([0,T/\epsilon);X^{s+1})})$ such that
 \begin{equation*}
	\forall t\in [0,\frac{T}{\epsilon}],\quad
	E^s(U_1-U_2)(t)\leq e^{\epsilon\lambda_{T} t}E^s(U_{1}\id{t=0}-U_{2}\id{t=0})+C_0\int^{t}_{0} e^{\epsilon\lambda_T( t-t')} E^s(F_1-F_2)(t')dt'.
\end{equation*}

\end{Proposition}
\begin{proof}
When multiplying the equations satisfied by $U_i$ on the left by $S[U_i]$, one obtains
	\begin{align*}
	S[U_1]\dsp\partial_t U_1+\Sigma_0[U_1]\partial_x U_1+\Sigma_1[U_1]\partial_x U_1= S[U_1]F_1\\
	S[U_2]\dsp\partial_t U_2+\Sigma_0[U_2]\partial_x U_2+\Sigma_1[U_2]\partial_x U_1= S[U_2]F_2;
    \end{align*}
with $\Sigma_0[U]=S[U]A_0[U]$ and $\Sigma_1[U]=S[U]A_1[U]$.
Subtracting the two equations above, and defining  $V=U_1-U_2\equiv (\zeta,v)^\top$ one obtains
\begin{multline*}
	S[U_1]\dsp\partial_t V+\Sigma_0[U_1]\partial_x V+\Sigma_1[U_1]\partial_x V=\\
	S[U_1](F_1-F_2)-(\Sigma_0[U_1]+\Sigma_1[U_1]-\Sigma_0[U_2]-\Sigma_1[U_2])\partial_xU_2-(S[U_1]-S[U_2])(\partial_tU_2-F_2).
\end{multline*}
We then apply $S^{-1}[U_1]$ and deduce the following system satisfied by $V$:
\begin{equation}\label{systemstability}
	\left\lbrace
	\begin{array}{l}
	\dsp\partial_t V+A_0[U_1]\partial_x V+A_1[U_1]\partial_x V=F \\
	\dsp V(0)=(U_{1}-U_{2})\id{t=0},
	\end{array}\right.
\end{equation}
where, 
\begin{multline}\label{rhsstability}
F\ \equiv \  F_1-F_2\ -\  S^{-1}[U_1]\big(\Sigma_0[U_1]+\Sigma_1[U_1]-\Sigma_0[U_2]-\Sigma_1[U_2]\big)
\partial_xU_2\\
-\ S^{-1}[U_1]\big(S[U_1]-S[U_2]\big)(\partial_tU_2-F_2).
\end{multline}
We wish to use the energy estimate of Lemma~\ref{Lem:Hs} to the linear system~\eqref{systemstability}. Thus one needs to control accordingly the right hand side $F$.

In order to do so, we take advantage of the following Lemma.
\begin{Lemma}\label{lemst}
 Let $(\mu,\epsilon,\delta,\gamma,\bo)\in\P_{\rm CH}$ and $s\geq s_0>1/2$. Let $V=(\zeta_v,v)^\top$, $W=(\zeta_w,w)^\top\in X^{s}$ and $U_1=(\zeta_1,v_1)^\top$, $U_2=(\zeta_2,v_2)^\top\in X^{s}$ such that there exists $h>0$ with
\[ 1-\epsilon\zeta_1\geq h>0,\ \quad 1-\epsilon\zeta_2\geq h>0,\ \quad \frac1\delta+\epsilon\zeta_1\geq h>0,\ \quad \frac1\delta+\epsilon\zeta_2\geq h>0.\]
Then one has
\begin{align}\label{eq:est-dif-S}
\Big\vert \ \Big(\ \Lambda^s \big( S[U_1]-S[U_2]\big) V \ , \ W \ \Big) \ \Big\vert \ & \leq \ \epsilon\ C \  \big\vert U_1-U_2\big\vert_{X^s}\big\vert V\big\vert_{X^s}\big\vert W\big\vert_{X^0} \\
\label{eq:est-dif-SA}
\Big(\ \Lambda^s \big( S[U_1]A[U_1]-S[U_2]A[U_2]\big) V \ , \ W \ \Big)  \ & \leq \ \epsilon\ C\ \big\vert U_1-U_2\big\vert_{X^s}\big\vert V\big\vert_{X^s}\big\vert W\big\vert_{X^0}
\end{align}
with $C=C(M_{\rm CH},h^{-1},\epsilon\big\vert U_1\big\vert_{X^s},\epsilon\big\vert U_2\big\vert_{X^s})$, and denoting $A[\cdot]\equiv A_0[\cdot]+A_1[\cdot]$.
\end{Lemma}
\begin{proof}[Proof of Lemma~\ref{lemst}]
Let $V=(\zeta_v,v)^\top$, $W=(\zeta_w,w)^\top\in X^{0}$ and $U_1=(\zeta_1,v_1)^\top$, $U_2=(\zeta_2,v_2)^\top\in X^{s}$. By definition of $S[\cdot]$ (see~\eqref{defS}), one has
\[  \Big(\ \Lambda^s \big( S[U_1]-S[U_2]\big) V \ , \ W \ \Big)  \ = \ \Big( \
\Lambda^s \big(\frac{Q_0(\epsilon\zeta_1)}{f(\epsilon\zeta_1)}-\frac{Q_0(\epsilon\zeta_2)}{f(\epsilon\zeta_2)}\big) \zeta_v \ , \ \zeta_w\Big)  \ + \ 
 \Big( \ \Lambda^s \big( \mfT[\epsilon\zeta_1]-\mfT[\epsilon\zeta_2]\big)v\ , \ w \ \Big).\]
Now, one can check that
\begin{align*}
 \mfT[\epsilon\zeta_1]v-\mfT[\epsilon\zeta_2]v \ &= \  \big(q_1(\epsilon\zeta_1)-q_1(\epsilon\zeta_2)\big)v-\mu\nu\partial_x\big\{\big(q_2(\epsilon\zeta_1)-q_2(\epsilon\zeta_2)\big)\partial_x v\big\}   \\
\ &= \ \epsilon \Big( \kappa_1(\zeta_1-\zeta_2)v-\mu\nu\partial_x\big\{\kappa_2(\zeta_1-\zeta_2)\partial_x v\big\}   \Big),
\end{align*}
so that, after one integration by part, and using Cauchy-Schwarz inequality and Lemma~\ref{Moser}, one has
\begin{equation}\label{est-dif-mfT} \Big|\  \Big( \ \Lambda^s \big( \mfT[\epsilon\zeta_1]-\mfT[\epsilon\zeta_2]\big)v\ , \ w \ \Big) \ \Big| \ \leq \ \epsilon\ C(\kappa_1,\nu\kappa_2) \big\vert \zeta_1-\zeta_2\big\vert_{H^s} \big\vert v\big\vert_{H^{s+1}_\mu} \big\vert w\big\vert_{H^{1}_\mu} .\end{equation}
In the same way, we remark that one can write $\frac{Q_0(X)}{f(X)}$ as a rational function:
\[ \frac{Q_0(X)}{f(X)}\ = \  (\gamma+\delta)\big(\frac{(1+X)(1-X+\gamma(\delta^{-1}+X))}{(1-X)(\delta^{-1}+X)} \ \equiv \ \frac{P(X)}{Q(X)}.\]
It follows
\begin{align*}
\frac{Q_0(\epsilon\zeta_1)}{f(\epsilon\zeta_1)}-\frac{Q_0(\epsilon\zeta_2)}{f(\epsilon\zeta_2)} \ &= \ \frac{P(\epsilon\zeta_1)Q(\epsilon\zeta_2)-P(\epsilon\zeta_2)Q(\epsilon\zeta_1)}{Q(\epsilon\zeta_1)Q(\epsilon\zeta_2)} \\
\ &= \ \frac{\big(P(\epsilon\zeta_1)-P(\epsilon\zeta_2)\big)Q(\epsilon\zeta_2)-P(\epsilon\zeta_2)\big(Q(\epsilon\zeta_1)-Q(\epsilon\zeta_2)\big)}{Q(\epsilon\zeta_1)Q(\epsilon\zeta_2)} .
\end{align*}
Since $P(X)$ and $Q(X)$ are polynomials, and using Lemma~\ref{Moser}, it is straightforward to check that for $s\geq s_0>1/2$, one has
\[ \big\vert \big(P(\epsilon\zeta_1)-P(\epsilon\zeta_2)\big)Q(\epsilon\zeta_2)-P(\epsilon\zeta_2)\big(Q(\epsilon\zeta_1)-Q(\epsilon\zeta_2)\big)\big\vert_{H^s} \ \leq \ \epsilon C(\epsilon\big\vert \zeta_1\big\vert_{H^s},\epsilon\big\vert \zeta_2\big\vert_{H^s})\big\vert \zeta_1-\zeta_2\big\vert_{H^s}\]
Now, applying Cauchy-Schwarz inequality and Lemma~\ref{Moser} together with Lemma~\ref{lem:f/h}, one deduces that as long as
\[ h_1(\epsilon\zeta_1),h_2(\epsilon\zeta_1),h_1(\epsilon\zeta_2),h_2(\epsilon\zeta_2) \  \geq \  h \ > \ 0,\]
one has (again thanks to continuous Sobolev embedding for $s\geq s_0>1/2$)
\begin{equation}\label{est-dif-Q}\Big\vert  \Big( \
\Lambda^s \big(\frac{Q_0(\epsilon\zeta_1)}{f(\epsilon\zeta_1)}-\frac{Q_0(\epsilon\zeta_2)}{f(\epsilon\zeta_2)}\big) \zeta_v \ , \ \zeta_w\Big) \Big\vert \ \leq \ \epsilon\ C \ \big\vert \zeta_1-\zeta_2\big\vert_{H^s} \big\vert \zeta_v\big\vert_{H^s}\big\vert \zeta_w\big\vert_{L^2},\end{equation}
with $C=C(M_{\rm CH},h^{-1},\epsilon\big\vert \zeta_1\big\vert_{H^s},\epsilon\big\vert \zeta_2\big\vert_{H^s})$.
Estimates~\eqref{est-dif-mfT} and~\eqref{est-dif-Q} yield~\eqref{eq:est-dif-S}.
\medskip

Let us now turn to~\eqref{eq:est-dif-SA}. One has
\begin{align}\label{eq:decomp-SU} \Big(\ \Lambda^s \big( S[U_1]A[U_1]-S[U_2]A[U_2]\big) V \ , \ W \ \Big)  \ &= \ \epsilon \Big( 
\Lambda^s \big(\frac{Q_0(\epsilon\zeta_1)}{f(\epsilon\zeta_1)}f'(\epsilon\zeta_1)v_1-\frac{Q_0(\epsilon\zeta_2)}{f(\epsilon\zeta_2)}f'(\epsilon\zeta_2)v_2\big) \zeta_v \ , \ \zeta_w\Big)  \\
&\quad + \ \Big( \
\Lambda^s \big(Q_0(\epsilon\zeta_1)-Q_0(\epsilon\zeta_2) \big) v \ , \ \zeta_w\Big) \nn \\
&\quad  + \ \Big( \ \Lambda^s \big(Q_0(\epsilon\zeta_1)-Q_0(\epsilon\zeta_2) \big)  \zeta_v \ , \ w\ \Big) \nn \\
&\quad  + \ \epsilon^2 \Big( \ \Lambda^s \big( Q_1(\epsilon\zeta_1,v_1)- Q_1(\epsilon\zeta_2,v_2)\big)  \zeta_v \ , \ w\ \Big) \nn \\
&\quad +
 \epsilon\Big( \ \Lambda^s \big( \mfQ[\epsilon\zeta_1,v_1]-\mfQ[\epsilon\zeta_2,v_2]\big)v\ , \ w \ \Big)\nn\\
&\quad  + \ \epsilon\varsigma \Big( \ \Lambda^s \big( \mfT[\epsilon\zeta_1](v_1\ v)- \mfT[\epsilon\zeta_2](v_2\ v)\big)  \ , \ w\ \Big) . \nn
 \end{align}
The second and third terms in the right-hand-side of~\eqref{eq:decomp-SU} may be estimated exactly as in~\eqref{est-dif-Q}, and we do not detail the precise calculations. The first term follows in the same way, using the decomposition
\begin{multline*}\epsilon  \big(\frac{Q_0(\epsilon\zeta_1)}{f(\epsilon\zeta_1)}f'(\epsilon\zeta_1)v_1-\frac{Q_0(\epsilon\zeta_2)}{f(\epsilon\zeta_2)}f'(\epsilon\zeta_2)v_2\big) \ = \  \big(\frac{Q_0(\epsilon\zeta_1)}{f(\epsilon\zeta_1)}f'(\epsilon\zeta_1)-\frac{Q_0(\epsilon\zeta_2)}{f(\epsilon\zeta_2)}f'(\epsilon\zeta_2)\big) (\epsilon v_1)\\ + \ \epsilon (v_1-v_2) \frac{Q_0(\epsilon\zeta_2)}{f(\epsilon\zeta_2)}f'(\epsilon\zeta_2),\end{multline*}
so that one has
\begin{multline*} \left\vert \Big( \
\Lambda^s \big(\frac{Q_0(\epsilon\zeta_1)}{f(\epsilon\zeta_1)}f'(\epsilon\zeta_1)v_1-\frac{Q_0(\epsilon\zeta_2)}{f(\epsilon\zeta_2)}f'(\epsilon\zeta_2)v_2\big) \zeta_v \ , \ \zeta_w\Big) \right\vert \leq C(\epsilon \big\vert v_1\big\vert_{H^s} ) \epsilon  \big\vert \zeta_1-\zeta_2\big\vert_{H^s} \big\vert \zeta_v\big\vert_{H^s}\big\vert \zeta_w\big\vert_{L^2}\\+C(\epsilon \big\vert \zeta_2\big\vert_{H^s} ) \epsilon\big\vert v_1-v_2\big\vert_{H^s} \big\vert \zeta_v\big\vert_{H^s}\big\vert \zeta_w\big\vert_{L^2}.
\end{multline*}
The fourth term is similar, as $\epsilon^2 Q_1(\epsilon\zeta_i,v_i)=Q_1(\epsilon\zeta_i,\epsilon v_i)$ is a bivariate polynomial. Let us detail the last two estimates. One has
\begin{multline}\label{eq:decomp-Ta}  \big(\mfQ[\epsilon\zeta_1,v_1]-\mfQ[\epsilon\zeta_2,v_2]\big)v \ = \ 2\big(q_1(\epsilon\zeta_1)q_3(\epsilon\zeta_1)v_1 -q_1(\epsilon\zeta_2)q_3(\epsilon\zeta_2)v_2\big)v+\mu\kappa\partial_x\big(v\partial_x(v_1-v_2)\big).\end{multline}
Again, the contribution of the first term in~\eqref{eq:decomp-Ta} is estimated as above (recalling that this term is multiplied by a $\epsilon$- factor), and the contribution of the last term in~\eqref{eq:decomp-Ta} is estimated below:
\[
 \left\vert \epsilon\mu\kappa \Big( \ \Lambda^s  \partial_x\big(v\partial_x(v_1-v_2)\big) \ , \ w \ \Big)\right\vert \leq C \epsilon\kappa\big\vert v_1-v_2\big\vert_{H^{s+1}_\mu} \big\vert v\big\vert_{H^s}\big\vert w\big\vert_{H^1_\mu}.
\]
We conclude by estimating the last term in~\eqref{eq:decomp-SU}. One has 
\begin{align*}
 \mfT[\epsilon\zeta_1](v_1\ v)-\mfT[\epsilon\zeta_2](v_2\ v)  &= \big(q_1(\epsilon\zeta_1)v_1-q_1(\epsilon\zeta_2)v_2\big)v-\mu\nu\partial_x\big\{\big(q_2(\epsilon\zeta_1)\partial_x(v_1\ v)-q_2(\epsilon\zeta_2)\partial_x(v_2\ v)\big\}   \\
\ &= \  \big( \mfT[\epsilon\zeta_1]-\mfT[\epsilon\zeta_2]\big)(v_1\ v) +  \epsilon(v_1-v_2)\big(q_1(\epsilon\zeta_2)v\big)\\
&\qquad\qquad -\mu\nu\partial_x\big\{q_2(\epsilon\zeta_2)\partial_x \big((v_1-v_2)\ v\big)\big\}   .
\end{align*}
One finally uses Cauchy-Schwarz inequality, Lemma~\ref{Moser} as well as~\eqref{est-dif-mfT}, and obtain
\begin{align*} \Big|\ \epsilon \Big( \ \Lambda^s \big( \mfT[\epsilon\zeta_1](v_1\ v)-\mfT[\epsilon\zeta_2](v_2\ v)\big) \ , \ w \ \Big) \ \Big| \ &\leq \ \epsilon^2\ C(\kappa_1,\nu\kappa_2) \big\vert \zeta_1-\zeta_2\big\vert_{H^s} \big\vert v_1 \ v\big\vert_{H^{s+1}_\mu} \big\vert w\big\vert_{H^{1}_\mu} \\
&\quad + \epsilon\ C(\kappa_1\epsilon\big\vert \zeta_2\big\vert_{H^s}) \big\vert v_1-v_2\big\vert_{H^s} \big\vert  v\big\vert_{H^{s}} \big\vert w\big\vert_{L^2} 
\\
&\quad +\nu\epsilon\ C(\kappa_2\epsilon\big\vert \zeta_2\big\vert_{H^s}) \big\vert v_1-v_2\big\vert_{H^{s+1}_\mu} \big\vert  v\big\vert_{H^{s+1}_\mu} \big\vert w\big\vert_{H^{1}_\mu} 
.\end{align*}
Altogether, we obtain~\eqref{eq:est-dif-SA}, and the Lemma is proved.
\end{proof}
\bigskip

Let us continue the proof of Proposition~\ref{prop:stability}, by estimating $F$ defined in~\eqref{rhsstability}, that we recall:
\begin{multline}\label{rhsstability2}
F\ \equiv \  F_1-F_2\ -\  S^{-1}[U_1]\big(\Sigma_0[U_1]+\Sigma_1[U_1]-\Sigma_0[U_2]-\Sigma_1[U_2]\big)
\partial_xU_2\\
-\ S^{-1}[U_1]\big(S[U_1]-S[U_2]\big)(\partial_tU_2-F_2).
\end{multline}
More precisely, we want to estimate
\begin{align*}
\big(\ \Lambda^s F\ ,\ S[U_1]\Lambda^s V\ \big) &= \big(\ \Lambda^s F_1-\Lambda^s F_2\ ,\ S[U_1]\Lambda^s V\ \big) \\
&\quad +\big(\ \Lambda^s(\Sigma_0[U_1]+\Sigma_1[U_1]-\Sigma_0[U_2]-\Sigma_1[U_2])
\partial_xU_2\ ,\ \Lambda^s V\ \big) \\
&\quad -\big(\ \Lambda^s (S[U_1]-S[U_2])(\partial_tU_2-F_2)\ ,\ \Lambda^s V\ \big) \\
&\quad +\big(\ \big[\Lambda^s,S^{-1}[U_1]\big](\Sigma_0[U_1]+\Sigma_1[U_1]-\Sigma_0[U_2]-\Sigma_1[U_2])
\partial_xU_2\ ,\ S[U_1]\Lambda^s V\ \big)\\
&\quad -\big(\ \big[\Lambda^s,S^{-1}[U_1]\big](S[U_1]-S[U_2])(\partial_tU_2-F_2)\ ,\ S[U_1]\Lambda^s V\ \big)\\
&= (I)+(II)+(III)+(IV)+(V).
\end{align*}

Let us estimate each of these terms. The contribution of $(I)$ is immediately bounded using Lemma~\ref{lemS}:
\begin{equation}
\big\vert \ (I)\ \big\vert \ \leq\ C \vert F_1-F_2\vert_{X^s}\vert V\vert_{X^s},
\end{equation}
with $C=C(M_{\rm CH},h^{-1},\epsilon\big\vert U_1\big\vert_{L^\infty})$.
\medskip

The contributions of $(II)$ and $(III)$ follow from Lemma~\ref{lemst}. Indeed, recalling that $V\equiv U_1-U_2$,~\eqref{eq:est-dif-S} yields immediately
\begin{equation}
\big\vert \ (III)\ \big\vert \ \leq\ C\epsilon \vert \partial_tU_2-F_2\vert_{X^s}\vert V\vert_{X^s}^2,
\end{equation}
and 
\eqref{eq:est-dif-SA} yields 
\begin{equation}
\big\vert \ (II)\ \big\vert \ \leq\ C\epsilon \vert \partial_x U_2\vert_{X^s}\vert V\vert_{X^s}^2,
\end{equation}
with $C=C=C(M_{\rm CH},h^{-1},\epsilon\big\vert U_1\big\vert_{X^s},\epsilon\big\vert U_2\big\vert_{X^s})$.
\medskip

Finally, we control $(IV)$ and $(V)$ using Lemma~\ref{lemS},~\eqref{eq:est-com-S-S}:
\[  \big\vert \big(\ \big[\Lambda^s,S^{-1}[U_1]\big] U
\ ,\ S[U_1]\Lambda^s V\ \big)\big\vert \leq C \vert U\vert_{H^{s-1}}\vert V\vert_{X^s},
\]
with $C=C(M_{\rm CH},h^{-1},\epsilon\big\vert \zeta_1\big\vert_{H^s})$.

Thus it remains to estimate $\vert U \vert_{H^{s-1}}$, where $U\equiv U_{(i)}\equiv 
(\Sigma_0[U_1]+\Sigma_1[U_1]-\Sigma_0[U_2]-\Sigma_1[U_2])\partial_xU_2$ or $U\equiv U_{(ii)}\equiv (S[U_1]-S[U_2])(\partial_tU_2-F_2)$.

We proceed as in Lemma~\ref{lemst}, helped by the fact that one is allowed lose one derivative in our estimates. Let $W\equiv \partial_tU_2-F_2\equiv (\zeta_w,w)^\top$. One has
\[U_{(ii)}\equiv (S[U_1]-S[U_2])W\equiv \begin{pmatrix}
\Big( \dsp\frac{Q_0(\epsilon\zeta_1)}{f(\epsilon\zeta_1)}-
\dsp\frac{Q_0(\epsilon\zeta_2)}{f(\epsilon\zeta_2)}\Big)\zeta_w\\
\Big(\mfT[\epsilon\zeta_1]-\mfT[\epsilon\zeta_2]\Big)w
\end{pmatrix}
\equiv \begin{pmatrix}
\zeta_{(ii)}\\ u_{(ii)}
\end{pmatrix}.\]
Recall that
\[ \mfT[\epsilon\zeta_1]w-\mfT[\epsilon\zeta_2]w \ = \ \epsilon \Big( \kappa_1(\zeta_1-\zeta_2)w-\mu\nu\partial_x\big\{\kappa_2(\zeta_1-\zeta_2)\partial_x w\big\}   \Big),
\]
so that one has straightforwardly
\[
\vert u_{(ii)}\vert_{H^{s-1}} \leq \epsilon  C(\kappa_1,\nu\kappa_2)\vert \zeta_1-\zeta_2\vert_{H^{s}}\vert w\vert_{H^{s+1}_{\mu}}.
\]
As for the first component, we apply~\eqref{est-dif-Q} and deduce
\[
\vert \zeta_{(ii)}\vert_{H^{s-1}}^2 = \big(\Lambda^{s-1}\zeta_{(ii)}, \Lambda^{s-1}\zeta_{(ii)} \big)\leq \epsilon  C \vert \zeta_1-\zeta_2\vert_{H^{s-1}}\vert \zeta_w\vert_{H^{s-1}}\vert \zeta_{(ii)}\vert_{H^{s-1}}.
\]
It follows 
\begin{equation}
\big\vert \ (V)\ \big\vert \ \leq\ C\epsilon \vert \partial_t U_2-F_2 \vert_{X^s}\vert V\vert_{X^s}^2,
\end{equation}
with $C=C(M_{\rm CH},h^{-1},\epsilon\big\vert U_1\big\vert_{X^s},\epsilon\big\vert U_2\big\vert_{X^s})$.

Now, recall $U_{(i)}\equiv 
(\Sigma_0[U_1]+\Sigma_1[U_1]-\Sigma_0[U_2]-\Sigma_1[U_2])\partial_xU_2$. Proceeding as above, one obtains
\[\big\vert \ U_{(i)}\ \big\vert_{H^{s-1}} \ \leq\ C \vert \partial_x U_2\vert_{X^{s}}\vert V\vert_{X^s},  \]
and thus
\begin{equation}
\big\vert \ (IV)\ \big\vert \ \leq\ C\epsilon \vert \partial_x U_2\vert_{X^{s}}\vert V\vert_{X^s}^2.
\end{equation}

Altogether, we proved (using Lemma~\ref{lemmaes}) that $F$, as defined in~\eqref{rhsstability}, satisfies
\begin{equation}\label{estF}
\vert \Big(\Lambda^s F,S[U_1]\Lambda^sV\Big)\vert\leq C_0 ( \vert \partial_x U_2\vert_{X^{s}}+ \vert \partial_t U_2-F_2 \vert_{X^{s}})\epsilon E^s(V)^2+C_0 E^s(V)E^s(F_1-F_2).
\end{equation}
with $C_0=C(M_{\rm CH},h^{-1},\epsilon\big\vert U_1\big\vert_{X^s},\epsilon\big\vert U_2\big\vert_{X^s})$. Notice also that by the system satisfied by $U_2$, one has (see detailed calculations in~\eqref{energyestimatederivativeinproof})
\begin{align*}
 \vert \partial_t U_2-F_2 \vert_{X^{s}} &\equiv - \vert  (A_0[U_2]+A_1[U_2])\partial_xU_2 \vert_{X^s} \\
 &\leq C( \vert  U_2 \vert_{X^{s+1}} )
\end{align*}

We can now conclude by Lemma~\ref{Lem:Hs}, and the proof is complete.
\end{proof}

\subsection{Well-posedness result}\label{ssec:WP}
In this section, we prove the well-posedness of the Cauchy problem for our new Green-Naghdi model~\eqref{eq:Serre2} in the sense of Hadamard. Existence and uniqueness of the solution is given by Theorem~\ref{thbi1}, while the stability with respect to the initial data is provided by Theorem~\ref{th:stabilityWP}. These results correspond to Theorems~\ref{thbi1mr} and~\ref{th:stabilityWPmr}, as stated in section~\ref{mr}.
\begin{Theorem}[Existence and uniqueness]\label{thbi1}
Let $\p= (\mu,\epsilon,\delta,\gamma,\bo) \in \P_{\rm CH}$ and  $s\geq s_0+1$, $s_0>1/2$, and assume
	 $U_0=(\zeta_0,v_0)^\top\in X^s$ satisfies~\eqref{CondDepth},\eqref{CondEllipticity}. Then there exists
         a maximal time $T_{\max}>0$, uniformly bounded from below with respect to $\p\in \P_{\rm CH} $, such that the system of
         equations~\eqref{eq:Serre2} admits
	 a unique solution $U=(\zeta,v)^\top \in C^0([0,T_{\max});X^s)\cap C^1([0,T_{\max});X^{s-1})$ with the initial value $(\zeta_0,v_0)\id{t=0}=(\zeta_0,v_0)$,
         and preserving the conditions~\eqref{CondDepth},\eqref{CondEllipticity} (with different lower bounds) for any $t\in [0,T_{\max})$.
         
   Moreover, there exists $T^{-1},C_0,\lambda=  C(M_{\rm CH},h_{01}^{-1},h_{02}^{-1},\big\vert U_0\big\vert_{X^{s}})$, independent of $\p\in\P_{\rm CH}$, such that $T_{\max}\geq T/\epsilon$, and one has the energy estimates 
\[\forall\ 0\leq t\leq\frac{T}{\epsilon}\ , \qquad 
\big\vert U(t,\cdot)\big\vert_{X^{s}} \ + \ 
\big\vert \partial_t U(t,\cdot)\big\vert_{X^{s-1}}  \leq C_0 e^{\epsilon\lambda t}\ .
\]
If $T_{\max}<\infty$, one has
         \[ \vert U(t,\cdot)\vert_{X^{s}}\longrightarrow\infty\quad\hbox{as}\quad t\longrightarrow T_{\max},\]
         or one of the two conditions~\eqref{CondDepth},\eqref{CondEllipticity} ceases to be true as $ t\longrightarrow T_{\max}$.
\end{Theorem}
\begin{proof}
We construct a sequence of approximate solution $(U^n=(\zeta^n,u^n))_{n\ge 0}$ through the
induction relation
\begin{equation}\label{approximatesys}
         U^0=U_0,\quad\mbox{ and }\quad
	\forall n\in\NN, \quad
	\left\lbrace
	\begin{array}{l}
	\dsp\partial_t U^{n+1}+A[U^n]\partial_x U^{n+1}
	=0;
        \\
	\dsp U^{n+1}_{\vert_{t=0}}=U_0.
	\end{array}\right.
\end{equation}
By Proposition~\ref{ESprop}, there is a unique solution $U^{n+1}\in C^0([0,T_{n+1}/\epsilon];X^s)\cap C^1([0,T_{n+1}/\epsilon];X^{s-1})$ to
\eqref{approximatesys} if $U^{n}\in C^0([0,T_n/\epsilon];X^s)\cap C^1([0,T_n/\epsilon];X^{s-1})\subset X^s_T$, and satisfies~\eqref{CondDepth},\eqref{CondEllipticity}.
\medskip

\noindent {\em Existence and uniform control of the sequence $U^n.$}
The existence of $T'>0$ such that the sequence $U^n$ is  uniquely defined, controlled in $X^s_{T'}$, and satisfies~\eqref{CondDepth},\eqref{CondEllipticity}, uniformly with respect to $n\in\NN$, is obtained by induction, as follows.

Proposition~\ref{ESprop} yields
\begin{equation}\label{energyestimates-n} E^s(U^{n+1})(t) \leq  e^{\epsilon \lambda_n t}E^s(U_0)\quad \text{ and }\quad \big\vert \partial_t U^{n+1}(t,\cdot)\big\vert_{X^{s-1}} \leq C_n E^s( U^{n+1})\leq C_n e^{\epsilon \lambda_n t}E^s(U_0),\end{equation}
with  $\lambda_n,C_n=C(M_{\rm CH},h_{01,n}^{-1},h_{02,n}^{-1},\big\Vert U^n\big\Vert_{X^s_{T_n}})$, provided $U^n\in X^s_{T_n} $ satisfies~\eqref{CondDepth},\eqref{CondEllipticity} with positive constants $h_{01,n},h_{02,n}$ on $[0,T_n/\epsilon]$. 

Since $U^{n}=(\zeta^n,v^n)^\top$ satisfies~\eqref{approximatesys}, one has
\begin{equation*}
\partial_t \zeta^{n+1}=-\epsilon f'(\epsilon\zeta^{n})v^{n}\partial_x\zeta^{n+1}-f(\epsilon\zeta^{n})\partial_x v^{n+1}.
\end{equation*}
Using continuous Sobolev embedding of $H^{s-1}$ into $L^\infty$ ($s-1>1/2$), and since $\zeta^n$ satisfies~\eqref{CondDepth},\eqref{CondEllipticity} with $h_{01,n},h_{02,n}$ on $[0,T_n/\epsilon]$, one deduces that
 \begin{equation}\label{zeta_t}
\vert \partial_t \zeta^{n+1} \vert_{L^{\infty}}
\leq C(M_{\rm CH},h_{01,n}^{-1},h_{02,n}^{-1}\big)\big\Vert U^n\big\Vert_{X^s_{T_n}}.
\end{equation}
Let $g^{n+1}=a+b\epsilon\zeta^{n+1}
$, where $(a,b)\in \{(1,-1),(\frac{1}{\delta},1), (1,\kappa_1),(1,\kappa_2)\}$.
One has
\[
g^{n+1}=g\id{t=0}+b\epsilon\int_{0}^{t}\partial_t\zeta^{n+1},
\]
so that~\eqref{zeta_t} yields
\[\vert g^{n+1}-g^{n+1}\id{t=0}\vert_{L^\infty}\leq \epsilon t\ \times\ b   C(M_{\rm CH},h_{01,n}^{-1},h_{02,n}^{-1}\big)\big\Vert U^n\big\Vert_{X^s_{T_n}}.\]
Now, one has $g^{n+1}\id{t=0}\equiv g^0\id{t=0} \ge \min(h_{01,0}^{-1},h_{02,0}^{-1})>0$, independent of $n$. Thus one can easily prove (by induction) that it is possible to chose $T'>0$ such that $g^{n+1}>\alpha/2$ holds on $[0,T'/\epsilon]$, and the above energy estimates hold uniformly with respect to $n$, on $[0,T'/\epsilon]$.

More precisely, one has that
$\zeta^n$ satisfies~\eqref{CondDepth},\eqref{CondEllipticity} with $h_{01}/2,h_{02}/2>0$ 
  and the estimates
\begin{equation}\label{energyestimates-no-n} 
E^s(U^{n})(t) \leq  e^{\epsilon \lambda t}E^s(U_0)\quad \text{ and }\quad \big\vert \partial_t U^n(t,\cdot)\big\vert_{X^{s-1}} \leq C_0 e^{\epsilon \lambda t}E^s(U_0),\end{equation}
on $[0,T'/\epsilon]$, where $\lambda,C_0=C(M_{\rm CH},h_{01}^{-1},h_{02}^{-1},\big\vert U_0\big\vert_{X^s})$ are uniform with respect to $n$.
 
\bigskip\noindent {\em Convergence of $U^n$ towards a solution of the nonlinear problem.} We now conclude by proving that the sequence $U^n$ converges towards a solution of our nonlinear problem. In order to do so, let us define $V^n\equiv  U^{n+1}-U^{n}$. $V_n$ satisfies the system
\begin{equation}\label{systemstability-n}
	\left\lbrace
	\begin{array}{l}
	\dsp\partial_t V^n+A_0[U^n]\partial_x V+A_1[U^n]\partial_x V=F^n \\
	\dsp V\id{t=0}\equiv 0,
	\end{array}\right.
\end{equation}
where, 
\begin{equation}\label{rhsstability-n}
F^n\equiv - S^{-1}[U^n](\Sigma_0[U^n]+\Sigma_1[U^n]-\Sigma_0[U^{n-1}]-\Sigma_1[U^{n-1}])
\partial_xU^n-S^{-1}[U^n](S[U^n]-S[U^{n-1}])\partial_tU^n.
\end{equation}
We wish to use the energy estimate of Lemma~\ref{Lem:L2} to the linear system~\eqref{systemstability}. Thus one needs to control accordingly the right hand side $F^n$.

More precisely, we want to estimate
\begin{align*}
\big(\ F^n\ ,\ S[U^n] V^n\ \big) &= -\big(\ (\Sigma_0[U^n]+\Sigma_1[U^n]-\Sigma_0[U^{n-1}]-\Sigma_1[U^{n-1}])
\partial_xU^n \ ,\  V^n\ \big) \\
&\quad -\big(\ (S[U^n]-S[U^{n-1}])\partial_tU^n\ ,\  V^n\ \big) .
\end{align*}
 Proceding as in Lemma~\ref{lemst}, one can easily deduce
\[\big\vert \big(\ F^n\ ,\ S[U^n] V^n\ \big) \big\vert \leq \epsilon C \big\vert U^n-U^{n-1}\big\vert_{X^0}\big\vert V^n\big\vert_{X^0}(\big\vert \partial_xU^n \big\vert_{X^0}+\big\vert \partial_tU^n\big\vert_{X^0}),\]
with $C=C(M_{\rm CH},h_{01}^{-1},h_{02}^{-1},\epsilon\big\vert U^{n-1}\big\vert_{W^{1,\infty}},\epsilon\big\vert U^n\big\vert_{W^{1,\infty}})$.

Using the uniform control of $U^n,\partial_t U^n$ in~\eqref{energyestimates-no-n}, one deduces
\[\big\vert \big(\ F^n\ ,\ S[U^n] V^n\ \big) \big\vert \leq \epsilon C_0 \big\vert V^{n-1}\big\vert_{X^0}\big\vert V^n\big\vert_{X^0},\]
with $C_0$ independent of $n$. Thus Lemma~\ref{Lem:L2} yields
\[	\forall t\in [0,\frac{T'}{\epsilon}],\qquad
	E^0(V^n)(t)\leq  \epsilon C_0 \int^{t}_{0} e^{\epsilon\lambda( t-t')}	E^0(V^{n-1})(t') dt'\leq \epsilon C \int^{t}_{0}  E^0(V^{n-1})(t') dt',
\]
where $C$ is independent of $n$ and $t$. Hence
\[	\forall t\in [0,\frac{T'}{\epsilon}],\qquad
	E^0(V^n)(t)\leq \frac{\epsilon^n t^n C^n}{n!} \sup_{t'\in[0,T'/\epsilon]}  E^0(V^0)(t') ,
\]
and 
the sequence $U^n\equiv U^0+\sum V^n$ converges in $C^0([0,T'/\epsilon];X^0)$.
\medskip

\noindent {\em Completion of the proof.}
Since $U^n$ converges in $C^0([0,T'/\epsilon];X^0)$ and is uniformly bounded in $X^s$,
standard interpolation arguments imply that the sequence $U^n$ converges in $C^0([0,T'/\epsilon];X^{s'})$, for any $s'<s$. Similarly, $\partial_t U^n$ converges in $C^0([0,T'/\epsilon];X^{s'-1})$. Choosing $s'-1>1/2$, one may pass to the limit all the terms in system~\eqref{approximatesys}, and one deduces that the limit $U$ is a solution of system~\eqref{eq:Serre2}. Passing to the limit the properties of $U^n$, and in particular the energy estimates~\eqref{energyestimates-no-n}, one deduces $U\in L^\infty([0,T/\epsilon];X^s),\partial_t U\in L^\infty ([0,T/\epsilon];X^{s-1})$, and $U$ satisfies the energy estimate of the Theorem (using Lemma~\ref{lemmaes}), and preserves the conditions~\eqref{CondDepth},\eqref{CondEllipticity} for any $t\in [0,\frac{T'}{\epsilon}]$,  independently of $\p\in\P_{\rm CH} $.

Finally, as in Proposition~\ref{ESprop} (with $\underline{U}\equiv U$), one has $U\in C^0([0,T'/\epsilon];X^s)\cap C^1([0,T'/\epsilon];X^{s-1})$. The uniqueness of $U$ follows from the stability result of Proposition~\ref{prop:stability} with $F_1\equiv F_2\equiv 0$, and one can therefore define a maximal time of existence of the solution, that we denote $T_{\max}$.
$T_{\max}$ is bounded from below by $T'/\epsilon>0$, and the behavior of the solution as $t\to T_{\max}$ if $T_{\max}<\infty$ follows from standard continuation arguments.
\end{proof}

\begin{Theorem}[Stability] \label{th:stabilityWP}
Let $ (\mu,\epsilon,\delta,\gamma,\bo) \in \P_{\rm CH}$ and $s\geq s_0+1$ with $s_0>1/2$, and assume
	 $U_{0,1}=(\zeta_{0,1},v_{0,1})^\top\in X^{s}$ and $U_{0,2}=(\zeta_{0,2},v_{0,2})^\top\in X^{s+1}$ satisfies~\eqref{CondDepth},\eqref{CondEllipticity}. Denote $U_j$ the solution to~\eqref{eq:Serre2} with $U_j\id{t=0}=U_{0,j}$.Then there exists $T^{-1},\lambda,C_0= C(M_{\rm CH},h_{01}^{-1},h_{02}^{-1},\big\vert U_{0,1}\big\vert_{X^s},\vert U_{0,2}\vert_{X^{s+1}})$ such that
 \begin{equation*}
	\forall t\in [0,\frac{T}{\epsilon}],\qquad
	\big\vert (U_1-U_2)(t,\cdot)\big\vert_{X^s} \leq C_0 e^{\epsilon\lambda t} \big\vert U_{1,0}-U_{2,0}\big\vert_{X^s}.
\end{equation*}
\end{Theorem}
\begin{proof}
The existence and uniform control of the solution $U_1$ (resp. $U_2$) in $L^\infty([0,T/\epsilon];X^s)$ (resp. $L^\infty([0,T/\epsilon];X^{s+1})$) is provided by Theorem~\ref{thbi1}.
The proposition is then a direct consequence of the {\em a priori} estimate of Proposition~\ref{prop:stability}, with $F_1=F_2=0$, and Lemma~\ref{lemmaes}.
\end{proof}

\section{Full justification of asymptotic models}\label{sec:fulljustification}
We conclude our work by explaining how the results of the previous sections allow to fully justify our system (and other consistent ones) as asymptotic model for the propagation of internal waves. A model is said to be {\em fully justified} (using the terminology of~\cite{Lannes}) if the Cauchy problem for both the full Euler system and the asymptotic model is well-posed for a given class of initial data, and over the relevant time scale; and if the solutions with corresponding initial data remain close.
As described in~\cite[Section~6.3]{Lannes13}, the full justification of  a system (S) follows from:
\begin{itemize}
\item (Consistency) One proves that families of solutions to the full Euler system, existing and controlled over the relevant time scale satisfies the system (S) up to a small residual. 
\item (Existence) One proves that families of solutions to the full Euler system as above do exist. This difficult step is ensured by Theorem~5 (or Theorem~6 for large times) in~\cite{Lannes13}, provided that a stability criterion is satisfied (see details therein).
\item (Convergence) One proves that the solutions of the full Euler system, and the ones of the system (S), with corresponding initial data, remain close over the relevant time scale. 
\end{itemize}
 The last step supposes that the Cauchy problem for the model is well-posed, and is a consequence of the stability of its solutions with respect to perturbations of the equation, so that the first two steps of the procedure (consistency and existence) yield the conclusion (convergence), and therefore the full justification of the model. Let us refer to Theorem~7 in~\cite{Lannes13} for the application of such procedure for the full justification of the so-called {\em shallow-water/shallow-water} asymptotic model, which corresponds to our system, when withdrawing $\O(\mu)$ terms.
\medskip

The consistency of our model has been given in Theorem~\ref{th:ConsSerreVarmr}. The well-posedness of the Cauchy problem is stated in Theorem~\ref{thbi1mr}, and the stability results is a consequence of Proposition~\ref{prop:stability}. Thus we have all the ingredients for the full justification of our model, stated in Theorem~\ref{th:convergencemr}, and that we recall below.
\begin{Theorem}[Convergence] \label{th:convergence}
Let $\p\equiv (\mu,\epsilon,\delta,\gamma,\bo)\in \P_{\rm CH}$ (see~\eqref{eqn:defRegimeCHmr}) and $s\geq s_0+1$ with $s_0>1/2$, and let $U^0\equiv(\zeta^0,\psi^0)^\top\in H^{s+N}(\RR)^2$, $N$ sufficiently large, satisfy the hypotheses of Theorem 5 in~\cite{Lannes13}
as well as~\eqref{CondDepth},\eqref{CondEllipticity}. Then there exists $C,T>0$, independent of $\p$, such that
\begin{itemize}
\item There exists a unique solution $U\equiv (\zeta,\psi)^\top$ to the full Euler system~\eqref{eqn:EulerCompletAdim}, defined on $[0,T]$ and with initial data $(\zeta^0,\psi^0)^\top$ (provided by Theorem 5 in~\cite{Lannes13});
\item There exists a unique solution $U_a\equiv (\zeta_a,v_a)^\top$ to our new model~\eqref{eq:Serre2mr}, defined on $[0,T]$ and with initial data $(\zeta^0,v^0)^\top$ (provided by Theorem~\ref{thbi1mr});
\item With $\bar v\equiv\bar v[\zeta,\psi]$, defined as in~\eqref{eqn:defbarvmr}, one has for any $t\in[0,T]$,
\[ \big\vert (\zeta,\bar v)-(\zeta_a,v_a) \big\vert_{L^\infty([0,t];X^s)}\leq C \ \mu^2\ t.\]
\end{itemize}
\end{Theorem}
\begin{proof}
As stated above, the existence of $U$ is provided by Theorem 5 in~\cite{Lannes13}, and the existence of $U_a$ is given by our Theorem~\ref{thbi1mr} (we choose $T$ as the minimum of the existence time of both solutions; it is bounded from below, independently of $\p\in\P_{\rm CH}$). If $N$ is large enough, then $U\equiv(\zeta,\psi)^\top$ satisfies the assumptions of our consistency result, Theorem~\ref{th:ConsSerreVarmr}, and therefore $(\zeta,\bar v)^\top$  solves~\eqref{eq:Serre2mr} up to a residual $R=(r_1,r_2)^\top$, with $\big\vert R\big\vert_{L^\infty([0,T];H^s)}\leq C(M_{\rm CH},h_{01}^{-1},\big\vert U^0\big\vert_{H^{s+N}})(\mu^2+\mu\epsilon^2)$. The result follows from the stability Proposition~\ref{prop:stability}, with $F_1=(r_1,\mfT[\epsilon\zeta]^{-1}r_2)^\top$ and $F_2=0$.
\end{proof}

\bigskip

In addition to allowing the complete, full justification of our model, the results of the previous sections allow to rigorously justify any lower order, well-posed and consistent model. We quickly show how to apply the procedure to such models, with the example of the so-called Constantin-Lannes decoupled approximation, introduced by one of the authors in~\cite{Duchene13}, and that we recall below.
\begin{Definition}[Constantin-Lannes decoupled approximation model] \label{def:CL}
Let $\zeta^0,v^0$ be given scalar functions, and set parameters $(\mu,\epsilon,\delta,\gamma,\bo)\in\P_{\rm CH}$, as defined in~\eqref{eqn:defRegimeCHmr}, and $(\lambda,\theta)\in\RR^2$. The Constantin-Lannes decoupled approximation is then
\[ U_{\text{CL}}\ \equiv \ \Big(v_+(t,x-t)+v_-(t,x+t),(\gamma+\delta)\big(v_+(t,x-t)-v_-(t,x+t)\big)\Big),\]
 where ${v_\pm}\id{t=0} \ = \ \frac12(\zeta^0\pm\frac{v^0}{\gamma+\delta})\id{t=0}$ and $v_\pm=(1\pm\mu\lambda \partial_x^2)^{-1}v_\pm^\lambda$ with $v_\pm^\lambda$ satisfying
\begin{multline} \label{eq:CL}
 \partial_t v_\pm^\lambda \ \pm \ \epsilon\alpha_1 v_\pm^\lambda\partial_x v_\pm^\lambda \ \pm \ \epsilon^2 \alpha_2 (v_\pm^\lambda)^2\partial_x v_\pm^\lambda\ \pm \ \epsilon^3 \alpha_3^{\theta,\lambda}(v_\pm^\lambda)^3\partial_x v_\pm \\
\pm\ \mu\nu^{\theta,\lambda}_x \partial_x^3 v_\pm^\lambda \ - \ \mu\nu^{\theta,\lambda}_t \partial_x^2\partial_t v_\pm^\lambda \ \pm \ \mu\epsilon\partial_x\big(\kappa_1^{\theta,\lambda} v_\pm^\lambda\partial_x^2 v_\pm^\lambda+\kappa_2^{\theta}(\partial_x v_\pm^\lambda)^2\big) \ = \ 0, 
\end{multline}
with parameters defined as follows:
\begin{equation}\label{eqn:parameters}
\begin{array}{c} \displaystyle
\alpha_1=\frac32\frac{\delta^2-\gamma}{\gamma+\delta},\quad \alpha_2=-3\frac{\gamma\delta(\delta+1)^2}{(\gamma+\delta)^2},\quad \alpha_3=-5\frac{\delta^2(\delta+1)^2\gamma(1-\gamma)}{(\gamma+\delta)^3},\\ \displaystyle
 \nu_t^{\theta,\lambda}\equiv \frac\theta6\frac{1+\gamma\delta}{\delta(\delta+\gamma)} + \lambda , \qquad \nu_x^{\theta,\lambda}\equiv \frac{1-\theta}6\frac{1+\gamma\delta}{\delta(\delta+\gamma)}-\frac1{2\bo} -\lambda,\\ \displaystyle
 \kappa_1^{\theta,\lambda} \equiv \frac{(1+\gamma\delta)(\delta^2-\gamma)}{3\delta(\gamma+\delta)^2}(1+\frac{1-\theta}4)-\frac{(1-\gamma)}{6(\gamma+\delta)}+\lambda\frac32\frac{\delta^2-\gamma}{\gamma+\delta} , \\ \displaystyle
 \kappa_2^\theta \equiv \frac{(1+\gamma\delta)(\delta^2-\gamma)}{3\delta(\gamma+\delta)^2}(1+\frac{1-\theta}4)-\frac{(1-\gamma)}{12(\gamma+\delta)}.\end{array}
\end{equation}
\end{Definition}

\begin{Remark}
The scalar equation~\eqref{eq:CL} has been introduced as a model for gravity surface wave (one layer of homogeneous fluid) in~\cite{Johnson02}, and its rigorous justification has been developed in~\cite{ConstantinLannes09}. The justification is to be understood in the following sense: if one chooses carefully the initial data (so as to focus on only one direction of propagation), then the Constantin-Lannes equation provides a good (unidirectional) approximation of the flow. The reason why a bi-directional, decoupled approximation as above has not been developed in the water-wave setting is that lower order scalar equations offer the same accuracy. The specificity of internal waves lies in the existence of a critical ratio ($\delta^2=\gamma$) for which quadratic nonlinearities vanish, thus calling for higher order decoupled models, especially in the Camassa-Holm regime.
\end{Remark}

The Cauchy problem for the equation~\eqref{eq:CL} has been proved to be locally well-posed in~\cite{ConstantinLannes09}, and a property of persistence of spatial decay at infinity has been proved in~\cite{Duchene13}. We state these results below.

\begin{Proposition}[Well-posedness and persistence]\label{prop:WPI}
Let $u^0 \in H^{s+1}$, with $s\geq s_0+1,\ s_0>1/2$. Let the parameters in~\eqref{eqn:parameters} be such that $\mu,\epsilon,\nu^{\theta,\lambda}_t>0$, and define $m>0$ such that
\begin{equation}\label{condparam}
\nu^{\theta,\lambda}_t+(\nu^{\theta,\lambda}_t)^{-1}+\mu+\epsilon +|\alpha_1|+|\alpha_2|+|\alpha_3|+|\nu^{\theta,\lambda}_x|+|\kappa^{\theta,\lambda}_1|+|\kappa^{\theta}_2| \ \leq \ m.
\end{equation}
Then there exists $T=C\big(m,\big\vert u^0 \big\vert_{H^{s+1}_\mu}\big)$ and a unique $u \in C^0([0,T/\epsilon); H^{s+1}_\mu)\cap C^1([0,T/\epsilon);H^{s}_\mu)$ such that $u$ satisfies~\eqref{eq:CL} and initial data $u\id{t=0}=u^0$.

 Moreover, $u$ satisfies the following energy estimate for $0\leq t\leq T/\epsilon$:
\[
\big\Vert \partial_t u \big\Vert_{L^\infty([0,T/\epsilon);H^{s}_\mu)} \ + \ \big\Vert u \big\Vert_{L^\infty([0,T/\epsilon);H^{s+1}_\mu)} \ \leq \ C(m, \big\vert u^0 \big\vert_{H^{s+1}_\mu}) .
\]

\medskip

 Assume additionally that for fixed $n, k\in\NN$, the function $x^{j} {u^0}\in H^{s+\b s}$, with $0\leq j \leq n$ and $\b s=k+1+2(n-j)$. Then 
 there exists $T=C\big(m,n,k,\sum_{j=0}^n\big\vert x^j {u^0}\big\vert_{H^{s+k+1+2(n-j)}_\mu}\big)$
such that for $0\leq t\leq T\times\min(1/\epsilon,1/\mu)$, one has
\[ \big\Vert x^n\partial_k \partial_t {u}\big\Vert_{L^\infty([0,t);H^{s}_\mu)} + \big\Vert x^n \partial_k {u}\big\Vert_{L^\infty([0,t);H^{s+1}_\mu)} \ \leq \ C\Big( m,n,k,\sum_{j=0}^n\big\vert x^j {u^0}\big\vert_{H^{s+k+1+2(n-j)}_\mu}\Big) .\]
\end{Proposition}

The decoupled model of Definition~\ref{def:CL} is consistent with our new model~\eqref{eq:Serre2mr}, in the following way.
\begin{Proposition}[Consistency]\label{prop:decompositionI}
Let $\zeta^0,v^0\in H^{s+6}$, with $s\geq s_0+1,\ s_0> 1/2$. For $(\mu,\epsilon,\delta,\gamma,\bo)=\p\in\P_{\rm CH}$, we denote $U_{\text{CL}}^\p$ the unique solution of the CL approximation, as defined in Definition~\ref{def:CL}. For some given $M^\star_{s+6}>0$, sufficiently large, assume that there exists $T^\star>0$ and a family $(U_{\text{CL}}^\p)_{\p\in\P_{\rm CH}}$ such that
\[ T^\star \ = \ \max\big\{ \ T\geq0\quad \text{such that}\quad \big\Vert U_{\text{CL}}^\p \big\Vert_{L^\infty([0,T);H^{s+6})}+\big\Vert \partial_t U_{\text{CL}}^\p \big\Vert_{L^\infty([0,T);H^{s+5})} \ \leq \ M^\star_{s+6}\ \big\} \ .\]

Then there exists $U^c=U^c[U_{\text{CL}}^\p]$ such that $U\equiv U_{\text{CL}}^\p+U^c$ satisfies the Green-Naghdi system~\eqref{eq:Serre2mr} up to a remainder $R$ bounded for $t\in [0,T^\star]$ by
\[ \big\Vert R\big\Vert_{L^{\infty}([0,t];H^s)} \ \leq \ C\ \max(\epsilon^2(\delta^2-\gamma)^2,\mu^2)\ (1+\sqrt t) ,\]
with $C=C(M^\star_{s+6},M_{\rm CH},|\lambda|,|\theta|)$, and the corrector term $U^c$ is estimated as
\[ \big\Vert U^c \big\Vert_{L^\infty([0,t];H^{s})}+\big\Vert \partial_t U^c \big\Vert_{L^\infty([0,t];H^{s})} \leq C\ \max(\epsilon(\delta^2-\gamma),\mu) \min( t, \sqrt t) .\]
\smallskip

Additionally, if there exists $\alpha>1/2$, $M^\sharp_{s+6},\ T^\sharp>0$ such that 
\[ \sum_{k=0}^6\big\Vert (1+x^2)^\alpha \partial_x^k U_{\text{CL}}^\p \big\Vert_{L^\infty([0,T);H^{s})} +\sum_{k=0}^5\big\Vert (1+x^2)^\alpha \partial_x^k\partial_t U_{\text{CL}}^\p \big\Vert_{L^\infty([0,T);H^{s})} \ \leq \ M^\sharp_{s+6}\ ,\]
then the remainder term, $R$ is uniformly bounded for $t\in[0,T^\sharp]$:
\[ \big\Vert R\big\Vert_{L^{\infty}([0,t];H^s)} \ \leq \  C\ \max(\epsilon^2(\delta^2-\gamma)^2,\mu^2),\]
with $C=C(M^\sharp_{s+6},M_{\rm CH},|\lambda|,|\theta|)$ and $U^c$ is uniformly estimated as
\[ \big\Vert U^c \big\Vert_{L^\infty([0,t];H^{s})}+\big\Vert \partial_t U^c \big\Vert_{L^\infty([0,t];H^{s})} \leq C\ \max(\epsilon(\delta^2-\gamma),\mu) \min( t, 1). \]
\end{Proposition}
\begin{proof}
A similar statement concerning the consistency of the Constantin-Lannes decoupled approximation towards the original Green-Naghdi system~\eqref{eqn:GreenNaghdiMean} (in the shallow water regime~\eqref{eqn:defRegimeSWmr} and neglecting the surface tension contributions) has been given in~\cite[Proposition~1.12]{Duchene13}. Taking into account the surface tension contribution only requires a slight modification in one of the parameters (one has $\nu_{x}^{\theta,\lambda} =\frac{1-\theta}6\frac{1+\gamma\delta}{\delta(\delta+\gamma)}-\lambda$ in~\cite{Duchene13}), as shows tedious but straightforward calculations. Finally, using the boundedness of $U_{\text{CL}}^\p$ assumed in the statement, and following the proof of Theorem~\ref{th:ConsSerreVarmr}, one easily checks that the corresponding result holds with regards to our system~\eqref{eq:Serre2mr}, in the more stringent Camassa-Holm regime~\eqref{eqn:defRegimeCHmr}.
\end{proof}

\begin{Remark}
 Let us note that Proposition~\ref{prop:WPI} ensures that the above result is not empty, but on the contrary is valid for long times (provided that $\nu_t^{\theta,\lambda}>0$ and the initial data sufficiently smooth). More precisely, if the initial data $(\zeta^0,v^0)=U^0\in H^{s+7},s\geq s_0+1,\ s_0>1/2$, then for any $\p\in\P_{\rm CH}$ and $\lambda,\theta$ such that $\nu_t^{\theta,\lambda}>\nu_0>0$, then there exists $C_1,C_2$
  independent of $\p$, such that for $M^\star_{s+6}\geq C_1$, the decoupled approximate solution satisfies the uniform bound of Proposition~\ref{prop:decompositionI}, with $T^\star\geq C_2/\epsilon$.

Moreover, upon additional condition on the decaying in space of the initial data, there exists
$C_1,C_2$ such that for any $M^\sharp_{s+6}\geq C_1$, one has
$ T^\sharp\geq C_2/\max(\epsilon,\mu)$.
\end{Remark}
\medskip

The above properties, thanks to the well-posedness and stability of our Green-Naghdi type system proved in this work, are sufficient to fully justify the solutions of the Constantin-Lannes decoupled approximation as approximate solutions of our model, {\em and therefore as approximate solutions of the full Euler system~\eqref{eqn:EulerCompletAdim}}.

\begin{Proposition}[Convergence of the decoupled model]\label{cor:convergence} Let $\p\equiv (\mu,\epsilon,\delta,\gamma,\bo)\in \P_{\rm CH}$, and set $\lambda,\theta\in\RR$ such that~\eqref{condparam} holds. Let $U^0\equiv(\zeta^0,\psi^0)^\top\in H^{s+N}$, $N$ sufficiently large, satisfy the hypotheses of Theorem 5 in~\cite{Lannes13}
as well as~\eqref{CondDepth},\eqref{CondEllipticity}. Denote by ${U}^\p\equiv(\zeta,\psi)^\top$ the solution of the full Euler system~\eqref{eqn:EulerCompletAdim} and by  $U_{\text{CL}}^\p$ the decoupled approximation defined in Definition~\ref{def:CL}, with initial data $U^0$. Then there exists $C,T>0$, independent of $\p\in \P_{\rm CH}$, such that for any $0\leq t\leq T$, one has
\[\big\Vert U_{\text{CL}}^\p-\overline{U}^\p\big\Vert_{L^\infty([0,t];H^s)} \ \leq \ C\ \big( \varepsilon_0 \ \min(t,t^{1/2}) (1 \ + \ \varepsilon_0 t ) \big),\]
with $\varepsilon_0 \ \equiv \ \max(\epsilon(\delta^2-\gamma),\mu)$, and $\overline{U}^\p\equiv(\zeta, \bar v[\zeta,\psi])^\top$ where $\bar v[\zeta,\psi]$ is defined as in~\eqref{eqn:defbarvmr}.

Moreover, if the initial data is sufficiently localized in space, as in the second part of  Proposition~\ref{prop:decompositionI}, then one has
\[ \big\Vert U_{\text{CL}}^\p-\overline{U}^\p\big\Vert_{L^\infty([0,t];H^s)} \ \leq \ C\ \big(\varepsilon_0 \ \min(t,1) (1 \ + \ \varepsilon_0 t ) \big).\]
\end{Proposition}
\begin{proof}
By Proposition~\ref{prop:decompositionI}, we know that for any $U_{\text{CL}}^\p$, there exists $U^c=U^c[U_{\text{CL}}^\p]$ such that $U\equiv U_{\text{CL}}^\p+U^c$ satisfies the Green-Naghdi system~\eqref{eq:Serre2mr}, up to a small remainder $R$. Controlling the difference $\big\Vert U-{U_{\text{GN}}}^\p\big\Vert$, where ${U_{\text{GN}}}^\p$ is the solution of the Green-Naghdi system~\eqref{eq:Serre2mr} with corresponding initial data, is done exactly as in Theorem~\ref{th:convergencemr}, and we omit the proof. The result is then a straightforward consequence of the triangular inequality, as $\big\Vert U^\p-{U_{\text{GN}}}^\p\big\Vert$ is estimated in Theorem~\ref{th:convergencemr}, and $\big\Vert U^c[U_{\text{CL}}^\p] \big\Vert$ in Proposition~\ref{prop:decompositionI}.
\end{proof}

\appendix

\section{Product and commutator estimates in Sobolev spaces.}
Let us recall here some product as well as commutator estimates in Sobolev spaces, used throughout the present paper.

\begin{Lemma}[product estimates]\label{Moser}~\\
Let $s\geq 0$, one has $\forall f,g\in H^s(\RR)\bigcap L^\infty(\RR),$ one has
\[ \big\vert \ f \ g\ \big\vert_{H^s} \ \lesssim \ \big\vert \ f \  \big\vert_{L^\infty} \big\vert \  g\  \big\vert_{H^s}+\big\vert \ f \  \big\vert_{H^s} \big\vert \  g\ \big\vert_{L^\infty}.
\]
If $s\geq s_0>1/2$, one deduces thanks to continuous embedding of Sobolev spaces,
\[ \big\vert \ f \ g\ \big\vert_{H^s} \ \lesssim\ \big\vert \ f \  \big\vert_{H^s} \big\vert \  g\ \big\vert_{H^s} .\]
More generally, for $s\geq 0$ and $s_0>1/2$, one has $\forall f\in H^s(\RR)\bigcap H^{s_0}(\RR),g\in H^s(\RR)$, 
\[
\big\vert \ f \ g\ \big\vert_{H^s}\lesssim \big\vert \ f \  \big\vert_{H^{s_{0}}}\big\vert \ g \  \big\vert_{H^s}+\left\langle \big\vert \ f \  \big\vert_{H^s} \big\vert \ g \  \big\vert_{H^{s_0}}\right\rangle_{s>{s_{0}}},
\]

Let $F\in C^\infty(\RR)$ such that $F(0)=0$. If $g\in H^s(\RR)\bigcap L^\infty(\RR)$ with $s\geq 0$, one has $F(g)\in H^s(\RR)$ and
\[ \big\vert\ F(g) \ \big\vert_{H^s} \ \leq \ C(\big\vert g\big\vert_{L^\infty},\big\vert F\big\vert_{C^\infty})\big\vert  \ g \ \big\vert_{H^s}.\]

Finally, we will use
\[ \big\vert \ f \ g\ \big\vert_{H^{s+1}_\mu} \ \leq \ C\ \big\vert \ f \  \big\vert_{H^{s+1}_\mu} \big\vert \  g \big\vert_{H^{s+1}_\mu} .\]
\end{Lemma}
The first estimates are classical (see~\cite{KatoPonce88,AlinhacGerard,Lannes06}), and the last one follows straightforwardly from
\[\big\vert \ f \ g\ \big\vert_{H^{s+1}_\mu}^2 \ \lesssim \ \big\vert \ f \ g\ \big\vert_{H^{s}}^2+\mu  \big\vert \ \partial_x( f \ g) \ \big\vert_{H^{s}}^2 \lesssim \ \big\vert \ f \  \big\vert_{H^s}^2 \big\vert \  g\ \big\vert_{H^s}^2+\mu  \big\vert \ g\ \partial_x f  \ \big\vert_{H^{s}}^2+\mu \big\vert \ f\ \partial_x g \ \big\vert_{H^{s}}^2.\]

We know recall commutator estimate, mainly due to the Kato-Ponce~\cite{KatoPonce88}, and recently improved by Lannes~\cite{Lannes06} (see Theorems 3 and 6):
\begin{Lemma}[commutator estimates]\label{K-P}~\\
For any $s\geq0$, and $\partial_x f, g\in L^\infty(\RR)\bigcap H^{s-1}(\RR),$ one has
\[ \big\vert\ [\Lambda^s,f]g\ \big\vert_{L^2}\ \lesssim\ \big\vert \ \partial_x f\ \big\vert_{H^{s-1}}\big\vert \ g\ \big\vert_{L^\infty}+\big\vert \ \partial_x f\ \big\vert_{L^\infty}\big\vert \ g\ \big\vert_{H^{s-1}}  .
\]
Thanks to continuous embedding of Sobolev spaces, one has for $s\geq s_0+1, \ s_0>\frac{1}{2},$ 
\[ \big\vert\ [\Lambda^s,f]g\ \big\vert_{L^2}\ \lesssim\ \big\vert \ \partial_x f\ \big\vert_{H^{s-1}}\big\vert\ g\ \big\vert_{H^{s-1}} .
\]
More generally, for any $s\geq0$ and $s_0>1/2$, $\partial_x f, g\in H^{s_0}(\RR)\bigcap H^{s-1}(\RR),$ one has
\[ \big\vert\ [\Lambda^s,f]g\ \big\vert_{L^2}\ \lesssim \big\vert \ \partial_x f\ \big\vert_{H^{s_0}}\big\vert\ g\ \big\vert_{H^{s-1}}+\left\langle\big\vert \ \partial_x f\ \big\vert_{H^{s-1}}\big\vert\ g\ \big\vert_{H^{s_0}}\right\rangle_{s>{s_{0}+1}}.
\]

\end{Lemma}

We conclude this section with two corollaries of Lemma~\ref{Moser}, used in particular in the proof of Theorem~\ref{th:ConsSerreVar}.  

\begin{Lemma} \label{lem:f/h}
Let $f,\zeta\in L^\infty\bigcap H^{\bar s}$, with $\bar s\geq 0$ and $h_1=1-\epsilon\zeta$, with $h_1(\epsilon\zeta)\geq h>0$ for any $x\in\RR$. Then one has
\[  \big\vert \frac1{h_1} f \big\vert_{H^{\bar s}} \ \leq \ C(h^{-1},\epsilon\big\vert\zeta\big\vert_{L^{\infty}})\big( \big\vert f\big\vert_{H^{\bar s}}+ \epsilon \big\vert \zeta\big\vert_{H^{\bar s}}\big\vert f\big\vert_{L^\infty}\big).\]
\end{Lemma}
\begin{proof}
We make use of the identity
\[ \frac1{h_1} f \ = \ \frac1{1-\epsilon\zeta}f \ = \ f \ + \ \frac{\epsilon\zeta}{1-\epsilon\zeta}f.\]
Moser's tame estimates (Lemma~\ref{Moser}) yields
 \begin{align*}\big\vert \frac1{h_1} f \big\vert_{H^{\bar s}} \ &\leq \ \big\vert f \big\vert_{H^{\bar s}} \ + \ \big\vert \frac{\epsilon\zeta}{1-\epsilon\zeta} f \big\vert_{H^{\bar s}}\\
 &\lesssim \ \big\vert f \big\vert_{H^{\bar s}} \ + \ \big\vert \frac{\epsilon\zeta}{1-\epsilon\zeta} \big\vert_{L^\infty}\big\vert f \big\vert_{H^{\bar s}} \ + \ \big\vert \frac{\epsilon\zeta}{1-\epsilon\zeta}\big\vert_{H^{\bar s}}\big\vert f \big\vert_{L^\infty}.
 \end{align*}
 The only non-trivial term to estimate is now $\big\vert \frac{\epsilon\zeta}{1-\epsilon\zeta}\big\vert_{H^{\bar s}}$. Using that $h_1=1-\epsilon\zeta \geq h>0$, we introduce a function $F\in C^\infty(\RR)$ such that 
 \[F(X) \ = \ \begin{cases}
 \frac{X}{1-X} & \text{ if } 1-X\geq h>0,\\ 0 & \text{ if } 1-X\leq 0.
 \end{cases}
\]
The function $F$ satisfies the hypotheses of Lemma~\ref{Moser}, and one has
\[ \big\vert \frac{\epsilon\zeta}{1-\epsilon\zeta}\big\vert_{H^{\bar s}} \ = \ \big\vert F(\epsilon\zeta)\big\vert_{H^{\bar s}} \ \leq \ C(\big\vert \epsilon\zeta\big\vert_{L^\infty},h^{-1})\big\vert \epsilon\zeta \big\vert_{H^{\bar s}}.\]
The Lemma is now straightforward.\end{proof}

Let us apply this Lemma to the rigorous estimate of specific expansions.
\begin{Lemma} \label{lem:f/h1}
Let $f,\zeta\in L^\infty([0,T);H^{\bar s})$, with $\bar s\geq s_0>1/2$, such that $h_1\equiv 1-\epsilon\zeta\geq h>0$ for any $(t,x)\in[0,T)\times\RR$. Then one has
\[  \big\vert \frac1{h_1} f - f\big\vert_{H^{\bar s}} \ \leq \ \epsilon\ C(h^{-1},\epsilon\big\vert \zeta\big\vert_{H^{\bar s}})\big\vert \zeta\big\vert_{H^{\bar s}}\big\vert f\big\vert_{H^{\bar s}}\]
and
\[  \big\vert \frac1{h_1} f -(1+\epsilon\zeta)f\big\vert_{H^{\bar s}} \ \leq \ \epsilon^2\ C(h^{-1},\epsilon\big\vert \zeta\big\vert_{H^{\bar s}})\big\vert \zeta\big\vert_{H^{\bar s}}^2\big\vert f\big\vert_{H^{\bar s}}.\]
\end{Lemma}
\begin{proof}
Let us first remark that the formal expansions are straightforward: \[\frac1{1-X}=1+\O(X)\qquad \text{ and }\qquad\frac1{1-X}=1+X+\O(X^2)\qquad \text{ for }\qquad |X|<1.\]

The rigorous estimate is obtained thanks to Lemma~\ref{lem:f/h}, applied to the identities
\[   \frac1{h_1} f - f \ = \ \frac{\epsilon\zeta}{h_1}f \ , \quad \text{ and } \quad \frac1{h_1} f -(1+\epsilon\zeta)f \ = \ \frac{\epsilon^2\zeta^2}{h_1}f.\]

Indeed, one has for $s\geq s_0>1/2$,
\[ \big\vert \frac1{h_1} f - f\big\vert_{H^{\bar s}} \ = \ \big\vert \frac{\epsilon\zeta}{h_1}f \big\vert_{H^{\bar s}} \ \lesssim \ \big\vert \epsilon\zeta \big\vert_{H^{\bar s}} \big\vert\frac{1}{h_1}f \big\vert_{H^{\bar s}} \ \leq \  \epsilon C(h^{-1},\epsilon\big\vert  \zeta\big\vert_{H^{\bar s}})\big\vert f \big\vert_{H^{\bar s}} ,\]
where we used Lemmata~\ref{Moser} and~\ref{lem:f/h}.
The second estimate, concerning $\big\vert \frac1{h_1} f -(1+\epsilon\zeta)f\big\vert_{H^{\bar s}} $, is obtained identically.
\end{proof}

\bibliographystyle{abbrv}
\def\cprime{$'$}

\end{document}